\documentclass[11pt]{amsart}
\usepackage[margin=1in]{geometry}
\usepackage{amssymb}
\usepackage{amsmath}
\usepackage{comment}
\usepackage{latexsym}
\usepackage{amsfonts}
\usepackage{graphicx}
\usepackage{mathrsfs}
\usepackage[normalem]{ulem}
\usepackage[utf8]{inputenc}
\usepackage{xcolor}
\usepackage{hyperref}
\hypersetup{
    colorlinks=true,
    linkcolor={blue!50!black},
    filecolor=red,      
    urlcolor={blue!50!black},
    citecolor={green!50!black}
}
\numberwithin{equation}{section} 

\newtheorem{theorem}{Theorem}[section]

\newtheorem{claim}[theorem]{Claim}

\newtheorem{corollary}[theorem]{Corollary}
\newtheorem{definition}[theorem]{Definition}
\newtheorem{example}[theorem]{Example}
\newtheorem{lemma}[theorem]{Lemma}

\newtheorem{proposition}[theorem]{Proposition}

\newtheorem{remark}[theorem]{Remark}

\newcommand{\Acal}{\mathcal A}
\newcommand{\dfplus}[1]{\delta_{#1}^{+}}
\newcommand{\Gcal}{\mathcal G}
\newcommand{\Jcal}{\mathcal J}
\newcommand{\Nash}{\mathcal N}
\newcommand{\Pcal}{\mathcal P}

\newcommand{\R}{\mathbb{R}}

\newcommand{\tr}{\operatorname{tr}}
\newcommand{\Ric}{\operatorname{Ric}}
\newcommand{\Rm}{\operatorname{Rm}}

\protected\def\vts{
  \ifmmode
    \mskip0.5\thinmuskip
  \else
    \ifhmode
      \kern0.08334em
    \fi
  \fi
}

\begin{document}

\title[The Fisher Metric of the Ricci Flow Heat Kernel]{The Fisher Metric of the Ricci Flow Heat Kernel}

\author{Bennett Chow}
\address{Department of Mathematics, University of California San Diego}
\email{bechow@ucsd.edu}

\author{Robert Koirala}
\address{Department of Mathematics, University of California San Diego}
\email{rkoirala@ucsd.edu}

\begin{abstract}
We introduce and study a Fisher information metric \(g^F_\tau\) associated to
the conjugate heat kernel of a Ricci flow \((M^n,g_t)\). This tensor measures,
at a fixed scale, how the pointed heat-kernel measure changes when the base
point is moved. We prove that \(g^F_\tau\) is monotone in scale and satisfies
\(g^F_\tau\le g_t\).  We relate its trace to the pointed Nash entropy and prove
a matrix square identity for the Fisher defect \(g_t-g^F_\tau.\) This identity gives a rigidity theorem for the equality case; on closed connected flows one has the strict inequalities \(0<g^F_\tau<g_t\) at every positive scale, while in the complete case equality forces a Euclidean splitting.

We also develop several consequences of this point of view.  These include a
sharp reverse Poincar\'e inequality for the heat semigroup, a contraction formula
for \(\varphi\)-divergences along conjugate heat flow, and a canonical
construction of heat-kernel splitting maps from large eigenvalues of the
Fisher endomorphism. As applications, we relate pointed Nash entropy close to \(0\) to small Fisher deficit at comparable scales, and we obtain a codimension-one Fisher-metric criterion for applying Bamler's \(\varepsilon\)-regularity theorem.
\end{abstract}

\maketitle 

\setcounter{tocdepth}{1}
\tableofcontents

\section{Introduction}

Let \((M^n,g_t)\) be a Ricci flow on a closed manifold, and let \(K(x,t;y,s)\) denote its heat kernel. Let 
\begin{equation}
    d\nu_{x,t;s}(y)=K(x,t;y,s)\,dg_s(y)
\end{equation}
denote the conjugate heat-kernel measure based at \((x,t)\) and evaluated at the earlier time \(s\).

We shall view the heat kernel as a \emph{two-point function} reflecting the geometry of \(g_t\); see Brendle \cite{Brendle2014TwoPoint} for a survey of two-point
functions in geometric analysis.  In the Ricci-flow setting, this heat-kernel
viewpoint is central in Bamler's work on singularity formation and the
structure theory of singular flows
\cite{Bamler2020A,MR4623543,Bamler2020C}. It played a crucial role already in Perelman's work \cite[\S 9]{Perelman1}, where he proved a fundamental differential Li--Yau--Hamilton Harnack inequality for the conjugate heat equation. For fixed \((y,s)\), \(K(\cdot,\cdot;y,s)\) solves the forward heat equation in the \((x,t)\)-variables, while for fixed \((x,t)\), the family \(K(x,t;\cdot,\cdot)\) solves the conjugate heat equation in the \((y,s)\)-variables. Thus the variables \((x,t)\) and \((y,s)\) play complementary (dual) roles. We call \((x,t)\) the \emph{source}, or base point, variable and regard \((y,s)\) as the \emph{target}, or integration, variable. This source--target duality is one of the organizing principles of the paper.

In this paper we study the symmetric nonnegative definite \(2\)-tensor at the base point \((x,t)\) defined by
\begin{equation}
    (g^F_\tau)_{(x,t)}
    =
    2\tau
    \int_M
    d_x\log K(x,t;y,t-\tau)\otimes d_x\log K(x,t;y,t-\tau)
    \,d\nu_{x,t;t-\tau}(y).
\end{equation}
Equivalently, for any \(V,W\in T_xM\)
\begin{equation}\label{eq:intro-fisher-metric}
    (g^F_\tau)_{(x,t)}(V,W)
    =
    2\tau
    \int_M
    \nabla_V^x\log K(x,t;y,t-\tau)\,
    \nabla_W^x\log K(x,t;y,t-\tau)
    \,d\nu_{x,t;t-\tau}(y).
\end{equation}
We call this tensor the \emph{Fisher metric}, in analogy with the Fisher information metric in statistics and information geometry; see 
Fisher
\cite{Fisher1925}, Rao \cite{MR15748}, and Amari and Nagaoka \cite{MR1800071}. It is \(1/(2\tau)\) times the covariance tensor, with respect to the target heat-kernel measure \(d\nu_{x,t;t-\tau}\), of the \textit{source score} $y\longmapsto 2\tau\,\nabla_x \log K(x,t;y,t-\tau)$. In the static (model) Euclidean case, \(g^F_\tau\) is equal to the Euclidean metric for every \(\tau>0\).

The source--target duality of the heat kernel is a duality between two canonical infinitesimal motions of the same pointed heat-kernel family. On the source side, fixing \(\tau\) and varying the base point \(x\) gives a map \(x\mapsto \nu_{x,t;t-\tau}\), whose tangent in the direction \(V\in T_xM\) is represented by the score \(2\tau\,\nabla_V^x\log K(x,t;\cdot,t-\tau)\). In the Gaussian model these are the translation modes, since
\[
    2\tau\,\nabla_V^x\log K_{\mathbb R^n}(x,t;y,t-\tau)
    =
    \langle y-x,V\rangle. 
\]
By contrast, varying the target time \(s=t-\tau\) moves the metric-measure
space
\[
    \tau\longmapsto
    \bigl(M,g_{t-\tau},\nu_{x,t;t-\tau}\bigr).
\]
In this paragraph the base point \((x,t)\) is fixed, and we suppress it from the
notation by writing
\[
    K_\tau(y):=K(x,t;y,t-\tau),
    \qquad
    d\nu_\tau(y):=d\nu_{x,t;t-\tau}(y)
    =
    K_\tau(y)\,dg_{t-\tau}(y)
    =
    (4\pi\tau)^{-n/2}e^{-f_\tau(y)}\,dg_{t-\tau}(y).
\]
We call \(f_\tau\) the \emph{heat-kernel potential}. All differential operators below
are taken in the target variable \(y\) with respect to the metric \(g_{t-\tau}\). Since
\[
    \partial_\tau K_\tau
    =
    \Delta_{g_{t-\tau}}K_\tau
    -
    R_{g_{t-\tau}}K_\tau,
    \qquad
    \partial_\tau dg_{t-\tau}
    =
    R_{g_{t-\tau}}\,dg_{t-\tau},
\]
we have the identity of top-dimensional forms, which was also independently observed by Yongjia Zhang,
\begin{align}\label{eq: target time Lie heat ker meas}
    \bigl(\partial_\tau+\mathcal L_{\nabla f_\tau}\bigr)\nu_\tau= \bigl(\Delta_{g_{t-\tau}}K_\tau\bigr)\,dg_{t-\tau}+\operatorname{div}_{g_{t-\tau}}
    \bigl(K_\tau\nabla f_\tau\bigr)\,dg_{t-\tau}=0.
\end{align}
Equivalently, along the flow \(\Phi_\tau\) generated by the time-dependent vector
field \(\nabla f_\tau\), the pullback measure \(\Phi_\tau^*\nu_\tau\) is independent
of \(\tau\). Hence \(\nabla f_\tau\) is the canonical target-side velocity: the gauge
velocity that freezes the conjugate heat-kernel measure.

In this same moving gauge, the target metric evolves by
\begin{align}\label{eq:freezing-of-conjugate-heat-kernel}
    \bigl(\partial_\tau+\mathcal L_{\nabla f_\tau}\bigr)g_{t-\tau}
    =
    2\bigl(\Ric_{g_{t-\tau}}+\nabla^2 f_\tau\bigr).
\end{align}
Therefore
\begin{equation}\label{eq: target time Lie rescaled metric GRS}
    \bigl(\partial_\tau+\mathcal L_{\nabla f_\tau}\bigr)
    \bigl(\tau^{-1}g_{t-\tau}\bigr)
    =
    2\tau^{-1}
    \left(
        \Ric_{g_{t-\tau}}
        +
        \nabla^2 f_\tau
        -
        \frac{1}{2\tau}g_{t-\tau}
    \right).
\end{equation}
Thus the source variation detects translation-type modes of the pointed heat
kernel, whereas the target-time variation detects the shrinker defect of the
rescaled metric in the canonical heat-kernel gauge. Equality on the target side is
therefore precisely dilation invariance, equivalently the \emph{gradient shrinking
soliton} equation.

The first basic facts about the Fisher metric are that
\begin{align}
    g^F_\tau \to g_t \qquad \text{as } \tau\downarrow0
\end{align}
and, for every positive scale,
\begin{align}\label{eq:fisher-less-than-metric}
    g^F_\tau \le g_t.
\end{align}
The latter inequality \eqref{eq:fisher-less-than-metric} is Bamler's matrix gradient estimate for the heat kernel \cite[Proposition 4.2]{Bamler2020A}. We also prove that, for fixed \((x,t)\), the Fisher metric is monotone nonincreasing as a function of the scale \(\tau\). Equivalently, if \(\Jcal_\tau=g_t^{-1}g^F_\tau\) is the \emph{Fisher endomorphism}, then the \emph{Fisher eigenvalues} of \(\Jcal_\tau\) lie in \([0,1]\) and decrease with \(\tau\). The proof of this monotonicity already illustrates the source--target duality: the scale derivative of \(g^F_\tau\) is expressed in terms of a target Dirichlet tensor, and the sign follows from the sharp Hein--Naber Poincar\'e inequality in the target variable \cite[Theorem 1.10]{HeinNaber2014CPAM}.
This reflects a recurring theme in
Ricci flow: sharp analytic inequalities for the heat kernel often encode precise geometric information.

We then consider the \textit{Fisher defect}
\begin{equation}
    D^F_\tau:=g_t-g^F_\tau. 
\end{equation}
This tensor is nonnegative by \eqref{eq:fisher-less-than-metric}. It satisfies a \emph{tensor} heat-operator identity whose right-hand side is a square of the Gaussian Hessian defect of the heat kernel. After taking the trace, this identity is closely related to Bamler's formula for the pointed Nash entropy \(\Box \mathcal{N}_s^*\leq 0\) \cite[Theorem 5.9]{Bamler2020A}. Thus \(D^F_\tau\) may be regarded as a \emph{tensorial} refinement of the corresponding scalar entropy quantity.

The Fisher defect satisfies a matrix square identity (Proposition \ref{prop:matrix-square-refines-Nash} below), which implies a rigidity statement. If the defect vanishes in a nonzero direction at a positive scale, then the associated null directions are parallel and flat. In particular, on a connected closed Ricci flow, the Fisher defect is positive definite for every positive scale. Thus all Fisher eigenvalues are strictly less than \(1\) when \(\tau>0\).

We also record analytic consequences of the same circle of estimates. The matrix Fisher estimate gives a sharp reverse Poincar\'e inequality for the heat semigroup. The \(\varphi\)-divergence viewpoint then gives a density-side interpretation of the Fisher metric: \(\varphi\)-divergences contract under conjugate heat flow, linearize to the Fisher metric along the pointed heat-kernel family, and yield spatial estimates for the map \(x\mapsto \nu_{x,t;t-\tau}\) with respect to relative entropy, squared Hellinger distances, and total variation distance. The relative entropy estimate, together with the heat-kernel Talagrand inequality \cite{MR1392331}, gives alternative and unified proof of the known Wasserstein monotonicity of Bamler \cite[Lemma 2.7]{Bamler2020A}, McCann and Topping \cite{MR2666905} and the \(1\leq p\leq 2\) case of Arnaudon--Coulibaly--Thalmaier \cite[Theorem 4.1]{MR2790368}. These estimates are sharp on the Euclidean Gaussian model.

We then use large (close to $1$) Fisher eigenvalues to produce canonical heat-kernel splitting maps. If the first \(k\) eigenvalues of \(\Jcal_\tau\) are close to \(1\), then at a comparable scale the corresponding heat-score functions are almost linear in a heat-kernel-weighted sense. This gives splitting maps without a self-similarity assumption, which typically requires the pinching of an entropy functional. In the collapsing case, one does not have entropy pinching, but our methods still produce splitting maps in the non-collapsing directions. In the codimension-one case, the same construction gives a Fisher eigenvalue criterion for Bamler's \(\varepsilon\)-regularity theorem \cite[Proposition~16.1]{Bamler2020C}. Finally, we show that, at comparable scales, the smallness of the pointed Nash entropy \(-\Nash\) is quantitatively equivalent to all Fisher eigenvalues being close to \(1\). This gives another precise sense in which the Fisher defect is a tensorial refinement of the pointed Nash entropy.

The paper is organized as follows. Section~\ref{sec2} introduces the Fisher metric and proves its basic monotonicity properties. Section~\ref{sec3} relates the Fisher defect to the pointed Nash entropy. Section~\ref{sec4} proves the equality-rigidity theorem. Section~\ref{sec5} proves the sharp reverse Poincar\'e inequality for the heat semigroup. Section~\ref{sec:phi-divergence-contraction} develops the \(\varphi\)-divergence interpretation of the Fisher metric and derives the spatial divergence and Wasserstein monotonicity along conjugate heat flows. Section~\ref{sec:large-fisher-eigenvalues} uses large Fisher eigenvalues to construct heat-kernel strong splitting maps and proves a Fisher \(\varepsilon\)-regularity criterion from large eigenvalues. Section~\ref{sec:fisher-epsilon-regularity} proves equivalence of smallness of \(-\Nash\) and all Fisher eigenvalues being close to \(1\). Most proofs are deferred to Section~\ref{sec7}.

The present paper develops the foundational Fisher-metric formalism.  We expect
this viewpoint to have further applications to Ricci-flow structure theory, and
we plan to pursue these in subsequent work.

\subsection*{Acknowledgments} 

We are thankful to Richard Bamler, Max Hallgren, Ayush Khaitan, Yi Lai, Yongjia Zhang for helpful feedback. We are especially grateful to Richard Bamler for helpful discussions and encouragement. He informed us that some of the results of this paper were independently known to him in related form.

\subsection*{LLM Usage Disclosure}

The authors acknowledge the use of AI as an interactive writing and checking aid during the preparation of this manuscript. The authors directed this use, selected and substantially revised any generated text, and independently verified the mathematical statements and proofs. The authors take full responsibility for the final content of the paper.

\section{The Fisher metric and its monotonicity}\label{sec2}

\subsection{Definition of the Fisher metric and endomorphism}

The basic object of our study is given by the following: 

\begin{definition}
    Given a Ricci flow, the associated \emph{Fisher metric} is defined by
    \begin{equation}
        g^F_\tau = 2\tau \int_M \nabla \log K \otimes \nabla \log K d\nu = 2\tau \int_M \frac{\nabla K \otimes \nabla K}{K} dg_s \geq 0. 
    \end{equation}
    More explicitly, the Fisher metric $g^F_\tau(t)$ at time $t$ and scale $\tau$ is given by
    \begin{align}\label{eq:FisherMetric1}
    \begin{split}
        (g^F_\tau)_{(x,t)} (V,W) 
        & = 2\tau \int_M \nabla_V^x \log K(x,t;y,s) \nabla_W^x \log K(x,t;y,s) \,  d\nu_{x,t;s}(y)  \\
        &  = 8\tau \int_M \nabla_V^x \sqrt{K} (x,t;y,s) \nabla_W^x \sqrt{K} (x,t;y,s) \,  dg_{s}(y) 
    \end{split}
    \end{align}
    for $V,W \in T_xM$, where $s = t-\tau$. Evidently, $g^F_\tau$ is a positive semi-definite symmetric 2-tensor.
    
    The \emph{Fisher endomorphism} $\Jcal_\tau$ is the corresponding $g_{t}$-self-adjoint endomorphism of $T_{x}M$, namely
    \begin{equation}
        g_{t}(\Jcal_\tau V,W)=g_\tau^F(V,W). 
    \end{equation}
    In other words, $\Jcal_\tau:=g_t^{-1}g^F_\tau$.
\end{definition}
Thus \(g^F_\tau\) is a \emph{base-point}, or \emph{source-side}, tensor: its integrand is the \(L^2(d\nu_{x,t;s})\)-pairing of the logarithmic first variations \(\nabla^x\log K\) of the conjugate heat-kernel measure as the base point \((x,t)\) is varied, while the \emph{target} variable \((y,s)\) is integrated out.

\begin{remark}
    Implicitly, we have assumed that the Ricci flow is defined on the time interval $[t-\tau,t]$. We will make this assumption, usually without mention, throughout this paper. 
\end{remark}

\begin{example}
    For static Euclidean space $(\mathbb{R}^n,g_{\mathbb{E}})$, the Fisher metric $g^F_\tau \equiv  g_{\mathbb{E}}$, independent of $\tau>0$.
\end{example}

\subsubsection{The Fisher metric as \(\tau\to0\)}

\begin{lemma}\label{lem:FisherMetricAsTauToZero}
    We have
    \begin{equation}\label{eq:Fisher-metric-asymptotic}
        g^F_\tau (t) \longrightarrow g_t \qquad\text{as }\tau\downarrow0.
    \end{equation}
\end{lemma}

\begin{proof}
    This follows from the expansion for the heat kernel and its first derivatives. See \S \ref{subsec:FisherMetTauToZero}. 
\end{proof}

\subsection{The matrix gradient estimate for the Fisher metric}

\begin{theorem}\label{thm:BamProp4.2}
    The matrix gradient estimate says that
    \begin{equation}\label{eqn: Fisher met leq g met}
        g^F_\tau \leq g_t, 
    \end{equation}
    that is, $(g^F_\tau)(t) \leq g_t$ for all $\tau>0$.
\end{theorem}

\begin{proof}
    See \cite[Proposition 4.2]{Bamler2020A}. Below we also give alternative proofs, in Remark \ref{rem:AltProofBamlerIneq} using Fisher metric monotonicity from the Hein--Naber Poincaré inequality, and in Remark \ref{rem:FisherDefectNonnegative} using the tensor maximum principle. These two alternative proofs are dual to each other in the sense that the Poincaré inequality is in the target variable while the tensor maximum principle is in the source variable.
\end{proof}

\subsection{Scale equation for the Fisher metric}

Fix \((x,t)\in M\times I\), let \(\tau>0\), and set \(s:=t-\tau. \) Let
\[
    K_\tau(y):=K(x,t;y,s),
    \qquad
    d\nu_\tau(y):=K_\tau(y)\,dg_{s}(y).
\]
For \(V\in T_xM\), define the \emph{source score}
\begin{equation}
    z_V^\tau:=2\tau \, \nabla_V^x\log K(x,t;y,s)  \,. 
\end{equation} 
Note that, in the Euclidean model, the source scores are exactly the linear coordinate functions centered at \(x\).

By differentiating
\[
    \int_M K(x,t;y,s)\,dg_{s}(y)=1
\]
in the \(V^x\)-direction, we obtain the \emph{source-score centering identity}
\begin{equation}\label{eq:score-coordinate-mean-summary}
    \int_M z_V^\tau\,d\nu_\tau=0. 
\end{equation}
We have for \(V,W\in T_xM\),
\begin{equation}\label{eq:score-coordinate-cov-summary}
    g^F_\tau (V,W) = \frac{1}{2\tau} \int_M z_V^\tau z_W^\tau \,d\nu_\tau \,.
\end{equation}

Define the \emph{target Dirichlet tensor}
\begin{equation}\label{eq:target_Dirichlet_tensor}
    G_\tau(V,W)
    :=
    \int_M
    \left\langle
    \nabla_y z_V^\tau,\nabla_y z_W^\tau
    \right\rangle_{g_{s}}
    d\nu_\tau,
\end{equation} 
which is a positive semi-definite symmetric 2-tensor. In particular, $G_\tau(V,V)$ is the Dirichlet energy of $z_V^\tau$ with respect to the heat-kernel measure. Let $\mathcal G_\tau$ be the corresponding $g_{t}$-self-adjoint endomorphism of $T_{x}M$, namely
\begin{equation}\label{eq:target_Dirichlet_11}
    g_{t}(\mathcal G_\tau V,W)=G_\tau (V,W).
\end{equation}

\begin{proposition}[Scale equation for the Fisher metric]\label{prop:ScaleEqFisherMetric}
    The variation in the time-scale of the Fisher metric is given by
    \begin{equation}\label{eq:gF-scale-equation-summary}
        \tau\,\partial_\tau g^F_\tau
        =
        g^F_\tau-G_\tau. 
    \end{equation}
\end{proposition}

\begin{proof} 
    This is a straightforward calculation. See \S \ref{subsec:ScaleEqFisherMetric}.
\end{proof}

\subsection{Fisher metric monotonicity}

A basic fact, which follows directly from the sharp Hein--Naber Poincar\'e inequality \cite[Theorem 1.10]{HeinNaber2014CPAM}, is that the Fisher metric is a monotone function of its scale:

\begin{proposition}[Fisher metric monotonicity]\label{prop:FisherMetMonot}
    For \(t\) fixed,
    \begin{equation}\label{eq:FisherMetricMonotonicity}
        \partial_\tau g^F_\tau(t)\le 0
    \end{equation}
    as an inequality of symmetric bilinear forms.
\end{proposition}

\begin{proof}
    This follows from the Hein--Naber Poincar\'e inequality in the target and \eqref{eq:gF-scale-equation-summary}. See \S \ref{subsec:PfFisherMetricMonot}.
\end{proof}

\begin{remark}\label{rem:AltProofBamlerIneq}
    Theorem \ref{thm:BamProp4.2} now follows from Lemma \ref{lem:FisherMetricAsTauToZero} and Proposition \ref{prop:FisherMetMonot}.
\end{remark}

\subsection{The heat operator of the Fisher defect tensor}

The \emph{Fisher defect tensor} is defined by
\begin{equation}
    D^F_\tau:=g_t-g^F_\tau. 
\end{equation}
From Theorem \ref{thm:BamProp4.2}, $0 \leq D^F_\tau \leq g_t$.\smallskip 

We define the source-side \emph{Gaussian Hessian defect}:
\begin{equation}\label{eq:source-H-def-Nash-section}
    \mathsf H^{(x)}
    :=
    \nabla_x^2\log K+\frac1{2\tau}g_t.
\end{equation}

\begin{lemma}[Mean source hessian identity] 
\label{lem:mean-Hessian-identity}
    For all source vectors $V,W\in T_xM$,
    \begin{equation}\label{eq:mean-Hessian-identity}
        \int_M \mathsf H^{(x)} (V,W)\,d\nu_\tau
        =\frac1{2\tau}D^F_\tau(V,W).
    \end{equation}
\end{lemma}

\begin{proof}
    See \S \ref{subsec:mean-Hessian-identity}.
\end{proof}

A fundamental tensor Bochner formula is:

\begin{proposition}[Matrix square identity]
\label{prop:matrix-square-refines-Nash}
    The Fisher defect tensor satisfies
    \begin{align}\label{eq:deficit-square-refines-Nash}
    \begin{split}
        \Box_{x,t}
        \left(
        \frac{\tau}{2}D^F_{\tau,ab}
        \right)
        &=
        2\tau^2
        \int_M
        \mathsf H^{(x)}_{ca}\mathsf H^{(x)}_{cb}
        \,d\nu_{x,t;s}\\
        &= \frac{1}{2}(D^F_{\tau})^2_{ab}+2\tau^2\int_M \left(\mathsf H^{(x)}-\frac1{2\tau}D^F_\tau\right)_{ca}\left(\mathsf H^{(x)}-\frac1{2\tau}D^F_\tau\right)_{cb}\,d\nu_{x,t;s}.
    \end{split}
    \end{align}
\end{proposition}

\begin{proof}
    The first identity is a straightforward calculation while the second identity is an application of Lemma \ref{lem:mean-Hessian-identity}. See \S \ref{subsec:matrix-square-refines-Nash}.
\end{proof}

\begin{remark}[Fisher heat logistic equation] \label{rem:fisher-heat-logistic-equation}
In the Uhlenbeck-frame convention for \(\Box_{x,t}\) with \(s\) fixed and
\(\tau=t-s\), equation
\eqref{eq:deficit-square-refines-Nash} is equivalent to
\begin{equation}
\label{eq:fisher-heat-logistic-equation}
    \Box_{x,t}g^F_\tau
    =
    \frac1{\tau}
    \left(
        g^F_\tau-(g^F_\tau)^2
    \right)
    -
    4\tau
    \int_M
    \left(
        \mathsf H^{(x)}-\frac1{2\tau}D^F_\tau
    \right)^2
    d\nu_{x,t;s}.
\end{equation}
Here, products of symmetric \(2\)-tensors are taken using \(g_t\).
\end{remark}

\begin{remark}
    The tensor \(\mathsf H^{(x)}\) should be viewed as a source-side splitting defect: it vanishes on the Euclidean Gaussian model, where the source scores are affine translation modes, and its squared \(d\nu\)-average measures the failure of those heat-kernel translation modes to remain parallel, in contrast to the target-side shrinker defect introduced below.
\end{remark}

\begin{remark}\label{rem:FisherDefectNonnegative}
    Directly from \eqref{eq:deficit-square-refines-Nash}, Hamilton's tensor maximum principle \cite[Theorem 9.1]{Hamilton1982} and \cite[\S 4]{Hamilton1986}, and the asymptotics of the heat kernel, we obtain an alternate proof that $D^F_{\tau,ab} \geq 0$; that is, $g^F_\tau \leq g_t$. 
\end{remark}

\subsection{Second scale variation of the Fisher metric}

Throughout this subsection the source point \((x,t)\) is fixed and suppressed from
the notation:
\[
    K_\tau(y):=K(x,t;y,t-\tau),
    \qquad
    d\nu_\tau(y):=K_\tau(y)\,dg_{t-\tau}(y)
    =
    (4\pi\tau)^{-n/2}e^{-f_\tau(y)}\,dg_{t-\tau}(y).
\]
All target gradients, Hessians, Laplacians, divergences, traces, norms, and inner
products are taken in the \(y\)-variable with respect to the metric \(g_{t-\tau}\).

Define the target \emph{shrinker defect tensor}
\begin{equation}\label{eq:target-shrinker-defect}
    S_\tau
    :=
    \Ric_{g_{t-\tau}}
    +
    \nabla_y^2 f_\tau
    -
    \frac{1}{2\tau}g_{t-\tau},
\end{equation}
and \emph{Perelman's Harnack quantity}
\[
    w_\tau
    :=
    \tau
    \left(
        2\Delta_{g_{t-\tau}}f_\tau
        -
        |\nabla_y f_\tau|_{g_{t-\tau}}^2
        +
        R_{g_{t-\tau}}
    \right)
    +
    f_\tau-n.
\]
We also write
\[
    \Delta_{f_\tau}u
    :=
    \Delta_{g_{t-\tau}}u
    -
    \left\langle
        \nabla_y f_\tau,\nabla_yu
    \right\rangle_{g_{t-\tau}}
\]
for the weighted target Laplacian.

We record two identities relating the source scores to the shrinker defect
tensor and Perelman's Harnack quantity.

\begin{lemma}[Source-target derivative identities]
\label{lem:target-hessian-source-score}
For every \(V\in T_xM\),
\begin{subequations}
\begin{align}
    \nabla_y^2 z_V^\tau
    &=
    -2\tau\,\nabla_V^x S_\tau,
    \label{eq:hessian-score-source-shrinker}\\
    \Delta_{f_\tau}z_V^\tau
    +
    \frac1{2\tau}z_V^\tau
    &=
    -\nabla_V^x w_\tau.
    \label{eq:drift-laplace-score-source-perelman-harnack}
\end{align}
\end{subequations}
\end{lemma}

\begin{proof}
    See \S \ref{subsec:target-hessian-source-score}.
\end{proof}

We have the following scale almost-convexity formula for the Fisher metric.

\begin{proposition}
\label{prop:second-scale-variation-of-the-fisher-metric}
For every \(V,W\in T_xM\),
\begin{equation}\label{eq:second-scale-variation-fisher}
\begin{aligned}
    \frac{d^2}{d\tau^2}g^F_\tau(V,W)
    &=
    \frac{2}{\tau}
    \int_M
    \left\langle
        \nabla_y^2 z_V^\tau,
        \nabla_y^2 z_W^\tau
    \right\rangle_{g_{t-\tau}}
    \,d\nu_\tau
    +
    \frac{4}{\tau}
    \int_M
    S_\tau
    \bigl(
        \nabla_y z_V^\tau,
        \nabla_y z_W^\tau
    \bigr)
    \,d\nu_\tau                                      \\
    &=
    8\tau
    \int_M
    \left\langle
        \nabla_V^x S_\tau,
        \nabla_W^x S_\tau
    \right\rangle_{g_{t-\tau}}
    \,d\nu_\tau
    +
    \frac{4}{\tau}
    \int_M
    S_\tau
    \bigl(
        \nabla_y z_V^\tau,
        \nabla_y z_W^\tau
    \bigr)
    \,d\nu_\tau .
\end{aligned}
\end{equation}
\end{proposition}

\begin{proof}
    See \S\ref{subsec:second-scale-variation-of-the-fisher-metric}.
\end{proof}

\begin{remark}[Convexity and rigidity]
    Assume that, for the fixed source \((x,t)\), \(S_\tau\equiv0\) for every
    \(\tau\in I\), where \(I\subset(0,\infty)\) is an interval. Then
    \eqref{eq:second-scale-variation-fisher} gives
    \[
        \frac{d^2}{d\tau^2}g^F_\tau(V,V)
        =
        \frac2\tau
        \int_M
        |\nabla_y^2z_V^\tau|_{g_{t-\tau}}^2
        \,d\nu_\tau
        \ge0.
    \]
    Thus \(\tau\longmapsto g^F_\tau(V,V)\) is convex on \(I\), and it is affine
    if and only if \(\nabla_y^2z_V^\tau\equiv0\) for every \(\tau\in I\).
    By \eqref{eq:hessian-score-source-shrinker}, this is equivalent to
    \(\nabla_V^xS_\tau\equiv0\) for every \(\tau\in I\).

    If \(\tau\mapsto g^F_\tau(V,V)\) is affine on \(I\) and
    \(g^F_{\tau_0}(V,V)>0\) for some \(\tau_0\in I\), then
    \(z_V^{\tau_0}\) is nonconstant because
    \(\int_M z_V^{\tau_0}\,d\nu_{\tau_0}=0\). Since
    \(\nabla_y^2z_V^{\tau_0}\equiv0\), the gradient
    \(\nabla_yz_V^{\tau_0}\) is a nonzero parallel vector field. Thus, if
    \((M,g_{t-\tau_0})\) is complete, its universal cover splits isometrically
    off an \(\mathbb R\)-factor.

    Finally, if \(I=(0,\infty)\), then every affine function
    \(\tau\longmapsto g^F_\tau(V,V)\) is constant: Fisher monotonicity makes
    its slope nonpositive, while \(g^F_\tau(V,V)\ge0\) for all \(\tau>0\)
    rules out a negative slope. Whenever the short-time limit in
    Lemma~\ref{lem:FisherMetricAsTauToZero} applies, this constant is
    \(g_t(V,V)\), so that \(g^F_\tau(V,V)\equiv g_t(V,V)\).
\end{remark}

The same computation used in the proof of
Proposition~\ref{prop:second-scale-variation-of-the-fisher-metric} gives the
following logarithmic-scale almost-convexity formula. For every
\(V,W\in T_xM\),
\begin{align}
\begin{split}
    \frac{d^2}{d(\log\tau)^2}g_\tau^F(V,W)
    &=
    2\tau
    \int_M
    \left(
        \Delta_{f_\tau}z_V^\tau+\frac{1}{2\tau}z_V^\tau
    \right)
    \left(
        \Delta_{f_\tau}z_W^\tau+\frac{1}{2\tau}z_W^\tau
    \right)
    d\nu_\tau                                      \\
    &\qquad+
    2\tau
    \int_M
    S_\tau(\nabla_yz_V^\tau,\nabla_yz_W^\tau)
    \,d\nu_\tau                                   \\
    &=
    2\tau
    \int_M
    \bigl(\nabla_V^x w_\tau\bigr)
    \bigl(\nabla_W^x w_\tau\bigr)
    \,d\nu_\tau
    +
    2\tau
    \int_M
    S_\tau
    \bigl(
        \nabla_y z_V^\tau,
        \nabla_y z_W^\tau
    \bigr)
    \,d\nu_\tau .
\end{split}
\end{align}

\section{Relation with pointed Nash entropy}\label{sec3}

We now explain how the Fisher defect tensor recovers, after taking its trace, the heat-operator formula for the pointed Nash entropy as calculated by Bamler \cite{Bamler2020A}. 
Thus the Fisher defect is a tensorial refinement of the corresponding scalar entropy quantity.

\subsection{The pointed Nash entropy}

Recall that the heat-kernel potential \(f=f_{x,t;s}\) is defined by
\begin{equation}
    K(x,t;y,s)
    =
    (4\pi\tau)^{-n/2}e^{-f(y)},
    \qquad
    \tau=t-s. 
\end{equation}
The \emph{pointed Nash entropy} based at \((x,t)\) and scale \(\tau\) is
\begin{equation}\label{eq:pointed-nash-def}
    \mathcal N_{x,t}(\tau)
    :=
    \int_M f(y)\,d\nu_{x,t;t-\tau}(y)-\frac n2. 
\end{equation}
Equivalently,
\begin{equation}\label{eq:pointed-nash-entropy-form}
    \mathcal N_{x,t}(\tau)
    =
    -\int_M K(x,t;y,t-\tau)\log K(x,t;y,t-\tau)\,dg_{t-\tau}(y)
    -\frac n2\log(4\pi\tau)-\frac n2. 
\end{equation}
In Euclidean space one has \(\mathcal N_{x,t}(\tau)\equiv 0.\) In general \(\mathcal N_{x,t}(\tau)\le0\), so \(\mathcal N_{x,t}\) measures the entropy defect of the conjugate heat kernel relative to the Euclidean Gaussian.

\subsection{The trace of the Fisher defect}

For fixed \(s<t\), define
\begin{equation}\label{eq:E-Nash-heat-operator-def}
    E(x,t;s)
    :=
    \Box_{x,t}\mathcal N_{x,t}(t-s).
\end{equation}
By \cite[Theorem 5.9]{Bamler2020A},
\begin{equation}\label{eq:E-Fisher-information}
    -\frac{n}{2\tau}\leq E(x,t;s)
    =
    \int_M |\nabla_x\log K(x,t;y,s)|^2\,d\nu_{x,t;s}(y)
    -
    \frac n{2\tau}\leq 0.
\end{equation}
We relate how the second inequality is related to the defect tensor. Since
\[
    \operatorname{tr}_{g_t}g^F_\tau
    =
    2\tau
    \int_M |\nabla_x\log K(x,t;y,s)|^2\,d\nu_{x,t;s}(y),
\]
we obtain
\begin{equation}\label{eq:trace-J-Nash}
    \operatorname{tr}\Jcal_\tau
    =
    \operatorname{tr}_{g_t}g^F_\tau
    =
    n+2\tau E(x,t;s)
    =
    n+2\tau\,\Box_{x,t}\mathcal N_{x,t}(\tau).
\end{equation}
Equivalently,
\begin{equation}\label{eq:trace-deficit-Nash}
    \operatorname{tr}_{g_t}D^F_\tau
    =
    -2\tau E(x,t;s)
    =
    -2\tau\,\Box_{x,t}\mathcal N_{x,t}(\tau).
\end{equation}
Thus the scalar Nash heat-operator quantity is exactly the normalized trace of the Fisher defect. In particular, the matrix gradient estimate \(D^F_\tau\ge0\) implies \(E= \Box_{x,t}\mathcal N_{x,t}(\tau)\le0\) which is the second inequality in \eqref{eq:E-Fisher-information}.

\subsection{The scalar square identity}

Tracing the matrix square identity \eqref{eq:deficit-square-refines-Nash} gives the following scalar identity.

\begin{corollary}[Scalar square identity for pointed Nash entropy]
\label{cor:scalar-Nash-square-identity}
    We have
    \begin{equation}\label{eq:BoxE-raw-Nash-section}
        \Box_{x,t}E
        =
        \frac n{2\tau^2}
        -
        2\int_M|\nabla_x^2\log K|^2\,d\nu_{x,t;s}.
    \end{equation}
    Equivalently,
    \begin{equation}\label{eq:tau2E-square}
        \Box_{x,t}\big(\tau^2E\big)
        =
        -2\tau^2
        \int_M|\mathsf H^{(x)}|^2\,d\nu_{x,t;s}
        \le -\frac{2}{n}\tau^2E^2 \le0,
    \end{equation}
    where $\mathsf H^{(x)}$ is defined by \eqref{eq:source-H-def-Nash-section}.
\end{corollary}

\begin{proof}
    This follows by tracing Proposition~\ref{prop:matrix-square-refines-Nash}. See \S\ref{subsec:scalar-Nash-square-identity}.
\end{proof}

\begin{remark}
    Together with the short-time asymptotic \(\tau^2E = \frac{\tau}{2}\bigl(\operatorname{tr}\Jcal_\tau-n\bigr) \rightarrow0\) as \(\tau\downarrow0,\) \eqref{eq:tau2E-square} gives another maximum-principle proof of \(E=\Box_{x,t}\mathcal N_{x,t}(\tau)\le0,\) recovering the scalar inequality in \cite[(5.10)]{Bamler2020A}.
\end{remark}

\section{Fisher equality rigidity}\label{sec4}

We have the following Fisher-metric-detected splitting/rigidity theorem for Ricci flows. This is a consequence of the Bochner formula \eqref{eq:deficit-square-refines-Nash} and Hamilton's strong maximum principle.

\begin{theorem}[Fisher equality rigidity]\label{thm:fisher-equality-rigidity}
    Let \((M^n,g(t))\) be a connected Ricci flow, without boundary, on a time interval containing \((s,t_0]\), where \(s<t_0\). Set \(\tau:=t-s.\) Suppose that the Fisher deficit tensor satisfies \(D^F_\tau\ge 0\) for \(s<t\le t_0\) (we make this assumption because we are not assuming that \(M\) is closed). Suppose that for some \(x_0\in M\) and some nonzero vector \(V_0\in T_{x_0}M\) one has
    \[
        D^F_{\tau_0}(V_0,V_0)=0,
        \qquad
        \tau_0:=t_0-s. 
    \]
    Then the nullspaces \(\mathcal K_x := \ker D^F_{\tau_0}(x,t_0;s)\) form a nontrivial parallel distribution \(\mathcal{K}\) on \((M,g(t_0))\). Moreover, for every \(x\in M\), every \(V\in\mathcal K_x\), and every \(y\in M\),
    \begin{equation}\label{eq:null-H-vanishing}
        \mathsf H^{(x)}(\cdot,V)(x,t_0;y,s)=0,
    \end{equation}
    where \(\mathsf H^{(x)} := \nabla_x^2\log K(x,t_0;y,s) + \frac{1}{2\tau_0}g_{t_0}\). Consequently,
    \begin{equation}\label{eq:curvature-annihilates-K}
        \Rm_{g(t_0)}(W,Z)V=0
    \end{equation}
    for all \(W,Z\in T_xM\) and all \(V\in\mathcal K_x\).
    
    If, in addition, \((M,g(t_0))\) is complete, then the universal cover splits isometrically as
    \[
        (\widetilde M,\widetilde g(t_0))
        \cong
        (\mathbb R^k,g_{\mathrm{Euc}})\times (N,h(t_0)),
    \]
    where \(k=\dim\mathcal K,\) and the lifted distribution \(\widetilde{\mathcal K}\) is tangent to the Euclidean factor.
    
    Finally, the local Hamilton strong maximum principle also gives space-time invariance of the null distribution on \((s,t_0]\). Hence, if the metrics \(g(t)\) are complete for all \(s<t\le t_0\), then the lifted Ricci flow splits on \((s,t_0]\) as
    \[
        (\widetilde M,\widetilde g(t))
        \cong
        (\mathbb R^k,g_{\mathrm{Euc}})\times (N,h(t)),
        \qquad
        s<t\le t_0,
    \]
    where the Euclidean factor is static and \(h(t)\) evolves by Ricci flow.
\end{theorem}

\begin{proof}
    See \S \ref{subsec:PfFisherEqualityRigidity}.
\end{proof}

The Fisher splitting theorem has the following consequence for Ricci flows on closed manifolds.

\begin{corollary}[Strict Fisher eigenvalue bounds on closed flows]\label{cor:strict-closed-fisher}
    Let \((M^n,g(t))\) be a connected closed Ricci flow. Then for every \(s<t\),
    \[
        0<g^F_{t-s}(x,t)<g_t(x) 
    \]
    as quadratic forms at every \(x\in M\). Equivalently, every Fisher eigenvalue satisfies
    \[
        0 <  \lambda_i(x,t;t-s)<1.
    \]
    Thus neither degeneracy of the Fisher metric nor exact equality in the matrix gradient estimate can occur on a connected closed Ricci flow at positive scale.
\end{corollary}

\begin{proof}
    See \S \ref{subsec:PfStrictFisherDeficitClosedFlows}.
\end{proof}

\section{Sharp reverse Poincar\'e inequality}\label{sec5}

In this section, we use the matrix gradient estimate \(g_\tau^F\le g_t\) to prove a sharp reverse Poincar\'e inequality for the heat semigroup.

For \(\phi\in C^\infty(M)\), define its heat average based at \((x,t)\) by
\begin{equation}
    ( P_{s,t}\phi ) (x)
    :=
    \int_M \phi(y)\,K(x,t;y,s)\,dg_s(y)
    =
    \int_M \phi\,d\nu_{x,t;s},
\end{equation} 
where \(P_{s,t}\) is the forward heat operator from time \(s\) to time \(t\).

Since 
\begin{align*}
    d_x ( P_{s,t}\phi ) (V) = \int_M \bigl(\phi-  P_{s,t}\phi (x)\bigr)\,
\nabla^x_V \log K \, d\nu_{x,t;s},
\end{align*}
the matrix gradient estimate may be repackaged as the following Fisher-form reverse Poincaré inequality.

\begin{corollary}[Fisher-form reverse Poincar\'e inequality]
\label{cor:sharp-reverse-poincare-hellinger}
    For every \(\phi\in C^\infty(M)\),
    \begin{equation}\label{eq:tensor-reverse-poincare}
        2\tau \vts 
        d_x(P_{s,t}\phi)\otimes d_x(P_{s,t}\phi)
        \le
        \operatorname{Var}_{x,t;s}(\phi) \vts  g^F_\tau (x)
        \le
        \operatorname{Var}_{x,t;s}(\phi) \,g_t(x),
    \end{equation}
    where
    \[
        \operatorname{Var}_{x,t;s}(\phi)
        :=
        \int_M
        \left(\phi- ( P_{s,t}\phi ) (x)\right)^2
        d\nu_{x,t;s}
        =
        \int_M \phi^2\,d\nu_{x,t;s}
        -
        \left(\int_M \phi\,d\nu_{x,t;s}\right)^2. 
    \]
    Equivalently,
    \begin{equation}\label{eq:sharp-reverse-poincare}
        2\tau\,|\nabla_x (P_{s,t}\phi)|_{g_t}^2
        \le
        2\tau\,|\nabla_x (P_{s,t}\phi)|_{(g_\tau^F)^{-1}}^2
        \le
        \operatorname{Var}_{x,t;s}(\phi).
    \end{equation}
\end{corollary}

\begin{proof}
    See \S \ref{subsec:SharpReversePoincare}.
\end{proof}

\begin{remark}
    This is also called the Fisher--Cram\'{e}r--Rao inequality. The constant in \eqref{eq:sharp-reverse-poincare} is sharp. On Euclidean space,
    \[
        K(x,t;y,s)
        =
        (4\pi\tau)^{-n/2}
        \exp\left(-\frac{|x-y|^2}{4\tau}\right),
    \]
    and for any linear function \(\phi(y)=\langle a,y\rangle\), one has
    \[
        (P_{s,t}\phi)(x)=\langle a,x\rangle,
        \qquad
        \operatorname{Var}_{x,t;s}(\phi)=2\tau |a|^2.
    \]
    Thus equality holds in \eqref{eq:sharp-reverse-poincare}.
\end{remark}

\section{\texorpdfstring{$\varphi$}{varphi}-divergence and the Fisher metric}
\label{sec:phi-divergence-contraction}

\subsection{Data processing and Fisher monotonicity}

Let \(\varphi:(0,\infty)\to\mathbb R\) be convex and \(C^2\), with \(\varphi(1)=0\). For probability measures \(\mu_0,\mu_1\) with \(\mu_0\ll\mu_1\), define the \emph{\(\varphi\)-divergence}
\begin{equation}
    D_\varphi(\mu_0\mid\mu_1)
    :=
    \int_M
    \varphi\left(\frac{d\mu_0}{d\mu_1}\right)d\mu_1,
\end{equation}
see R\'{e}nyi \cite{MR132570}, Csisz\'{a}r \cite{MR219345}, Morimoto \cite{MR167200}, and Ali and Silvey \cite{MR196777}. Adding a multiple of \(u-1\) to \(\varphi(u)\) does not change \(D_\varphi\). Thus, when convenient, we may also assume \(\varphi'(1)=0\), in which case \(\varphi\ge0\).

The quantity \(D_\varphi\) measures relative-density distinguishability. This is complementary to Wasserstein distance, which measures transport through the time-slice metric. In this section, the two viewpoints meet: the \(\varphi\)-divergences linearize to the Fisher metric, while the KL estimate below (see Corollary \ref{cor:KL-spatial-estimate}) gives an alternative density-divergence proof of the Wasserstein monotonicity of Bamler \cite[Lemma~2.7]{Bamler2020A}, McCann--Topping \cite{MR2666905} and the \(1\leq p\leq 2\) case of Arnaudon--Coulibaly--Thalmaier \cite[Theorem 4.1]{MR2790368}.

\begin{example}
    If \(d\mu_i=\rho_i\,dg\), then
    \[
        \varphi_{\rm KL}(u)=u\log u-u+1,
        \qquad
        \varphi_{\rm TV}(u)=\frac12|u-1|,
        \qquad
        \varphi_{\rm H}(u)=\frac12(\sqrt u-1)^2
    \]
    give relative entropy (Kullback–Leibler divergence), total variation, and squared Hellinger distance:
    \begin{align*}
        D_{\varphi_{\rm KL}}(\mu_0\mid\mu_1)
        &=
        \int_M
        \rho_0\log\frac{\rho_0}{\rho_1}\,dg,\\
        D_{\varphi_{\rm TV}}(\mu_0\mid\mu_1)
        &=
        {\rm TV}(\mu_0,\mu_1)
        =
        \frac12\int_M|\rho_0-\rho_1|\,dg,\\
        D_{\varphi_{\rm H}}(\mu_0\mid\mu_1)
        &=
        d_{\rm H}^2(\mu_0,\mu_1)
        =
        \frac12
        \int_M
        \left(\sqrt{\rho_0}-\sqrt{\rho_1}\right)^2dg
    \end{align*}
as associated $\varphi$-divergences.
\end{example}

\begin{remark}
    On a metric space of diameter \(D\),
    \[
        W_p^p(\mu_0,\mu_1)
        \le
        D^p\,{\rm TV}(\mu_0,\mu_1),
    \]
    where $W_p$ is the $p$-Wasserstein distance.
    For probability measures,
    \[
        d_{\rm H}^2\le {\rm TV}\le \sqrt2\,d_{\rm H},
        \qquad
        2\,{\rm TV}^2\le D_{\varphi_{\rm KL}},
        \qquad
        2\,d_{\rm H}^2\le D_{\varphi_{\rm KL}}.
    \]
    Thus divergences can dominate Wasserstein distances once one has a diameter bound. However, Wasserstein distances do not necessarily dominate divergences: if \(x\neq y\), then
    \[
        {\rm TV}(\delta_x,\delta_y)=d_{\rm H}(\delta_x,\delta_y)=1,
        \qquad \text{but} \qquad
        W_p^p(\delta_x,\delta_y)=d(x,y)^p\to0
    \]
    as \(y\to x\).
\end{remark}

We record the differential form of the data-processing inequality for \(\varphi\)-divergences; see \cite{MR167200,MR219345,MR196777}. It is the density-divergence analogue of Wasserstein monotonicity for conjugate heat flows.

\begin{proposition}[\(\varphi\)-divergence contraction]\label{prop:divergence-contraction}
    Let \((M^n,g_s)_{s\in I}\) be a closed Ricci flow, and let \(v_0,v_1\) be positive solution to the conjugate heat equation \(\Box^*v_i=0\), \(i=0,1\), normalized by \(\int_M v_i(\cdot,s)\,dg_s=1\). Set \(d\mu_i(s):= v_i(\cdot,s)\,dg_s\). Then
    \begin{align}\label{eq:divergence-monotonicity}
        \frac{d}{ds}D_\varphi(\mu_0(s)\mid\mu_1(s))
        =
        \int_M
        \varphi''\left(\frac{v_0}{v_1}\right)
        \left|\nabla\left(\frac{v_0}{v_1}\right)\right|_{g_s}^2
        d\mu_1(s)
        \ge 0.
    \end{align}
    If \(\varphi''>0\) and \(M\) is connected, then equality on a nontrivial time interval holds if and only if \(\mu_0(s)=\mu_1(s)\) on that interval.
\end{proposition}

\begin{proof}
    See \S\ref{subsec:proofs-phi-divergences}.
\end{proof}

The next result shows that the Fisher metric, after normalization, is the universal infinitesimal \(\varphi\)-divergence metric of the pointed heat-kernel family.

\begin{proposition}[Infinitesimal \(\varphi\)-divergence metric]
\label{prop:infinitesimal-phi-divergence}
    Let \(d\nu_{x,t;s}=K(x,t;\cdot,s)\,dg_s,\) \(\tau=t-s. \) Then for every \(V\in T_xM\),
    \begin{align}\label{eq:infinitesimal-divergence-is-fisher}
        \lim_{r\to0}
        \frac{1}{r^2}
        D_\varphi
        \left(
        \nu_{\exp_x(rV),t;s}
        \mid
        \nu_{x,t;s}
        \right)
        =
        \frac{\varphi''(1)}{4\tau}g_\tau^F(V,V).
    \end{align}
\end{proposition}

\begin{proof}
    See \S\ref{subsec:infinitesimal-phi-divergence}.
\end{proof}

\begin{remark}
    Assume \(\varphi''(1)>0\). If \(x_r\) is any smooth curve with \(x_0=x\), then
    \[
        \left.
        \frac{d}{dr}
        D_\varphi
        \left(
        \nu_{x_r,t;t-\tau}
        \mid
        \nu_{x,t;t-\tau}
        \right)
        \right|_{r=0}
        =
        \varphi'(1)
        \left.
        \frac{d}{dr}
        \int_M
        K(x_r,t;y,t-\tau)\,dg_{t-\tau}(y)
        \right|_{r=0}
        =
        0,
    \]
    since each heat kernel has unit mass. Thus \(D_\varphi\) vanishes to first order along the diagonal, and \eqref{eq:infinitesimal-divergence-is-fisher} gives
    \[
        \left.
        \operatorname{Hess}_{x_0}
        \left[
        \frac{2\tau}{\varphi''(1)}
        D_\varphi
        \left(
        \nu_{x_0,t;t-\tau}
        \mid
        \nu_{x,t;t-\tau}
        \right)
        \right]
        \right|_{x_0=x}
        =
        g_\tau^F.
    \]
    Thus \(g_\tau^F\) is recovered as the Hessian in the first variable of any nondegenerate \(\varphi\)-divergence along the diagonal, just as a Riemannian metric is recovered as the Hessian of one-half the squared distance:
    \[
        \left.
        \operatorname{Hess}_{x_0}
        \left(
        \frac12d_g^2(x_0,x)
        \right)
        \right|_{x_0=x}
        =
        g.
    \]
\end{remark}

Thus the \(\varphi\)-divergence contraction linearizes to monotonicity of \( \tau\mapsto\tau^{-1}g_\tau^F\). Note that the scale identity \eqref{eq:gF-scale-equation-summary} gives the exact tensorial formula
\begin{equation}\label{eq:gF-over-tau-monotone-Phi}
    \partial_\tau
    \left(
    \frac{1}{16\tau}g_\tau^F
    \right)
    =
    -\frac{1}{16\tau^2}G_\tau
    \le0. 
\end{equation}
To recover monotonicity of \(g_\tau^F\) itself, we use the following sharp \(\varphi\)-Sobolev inequality (see Chafa\"i \cite{MR2081075}).

\begin{proposition}[\(\varphi\)-Sobolev inequality]
\label{prop:sharp-phi-sobolev-heat-kernel}
    Let \(\varphi\in C^4((0,\infty))\) be convex, with \(\varphi(1)=0\), \(\varphi''>0\), and \(\left(1/\varphi''\right)''\le0. \) Fix \((x,t)\). Then for every \(\tau>0\) with \(t-\tau\in I\), and every smooth positive \(u\) with \(\int_M u\,d\nu_{x,t;t-\tau}=1,\) one has
    \begin{equation}\label{eq:sharp-phi-sobolev-heat-kernel}
        D_\varphi
        \left(
        u\, \nu_{x,t;t-\tau}
        \mid
        \nu_{x,t;t-\tau}
        \right)
        \le
        \tau
        \int_M
        \varphi''(u)|\nabla u|^2_{g_{t-\tau}}
        \,d\nu_{x,t;t-\tau}. 
    \end{equation}
\end{proposition}

\begin{proof}
    See \S\ref{subsec:sharp-phi-sobolev-heat-kernel}.
\end{proof}

\begin{remark}
    For \(\varphi(r)=r\log r-r+1\), Proposition~\ref{prop:sharp-phi-sobolev-heat-kernel} recovers the sharp weighted log-Sobolev inequality of Hein--Naber \cite[Theorem~1.10(2)]{HeinNaber2014CPAM}. For \(\varphi(r)=(r-1)^2\), it recovers their sharp weighted Poincaré inequality \cite[Theorem~1.10(1)]{HeinNaber2014CPAM}. The general version is inspired their calculation.
\end{remark}

Using the data processing inequality and the \(\varphi\)-Sobolev inequality, we obtain the strong data processing inequality.

\begin{corollary}[Strong data-processing inequality]
\label{cor:scale-sharp-phi-SDPI}
    Under the hypotheses of Proposition~\ref{prop:sharp-phi-sobolev-heat-kernel}, let \(\eta_\tau\) be any positive normalized conjugate heat evolution. Then 
    \[ 
        \tau\longmapsto \tau D_\varphi \left( \eta_\tau \mid \nu_{x,t;t-\tau} \right)
    \] 
    is nonincreasing. Equivalently, if \(0<\tau_0<\tau_1\), then
    \begin{equation}\label{eq:scale-sharp-phi-SDPI}
        D_\varphi
        \left(
        \eta_{\tau_1}
        \mid
        \nu_{x,t;t-\tau_1}
        \right)
        \le
        \frac{\tau_0}{\tau_1}
        D_\varphi
        \left(
        \eta_{\tau_0}
        \mid
        \nu_{x,t;t-\tau_0}
        \right).
    \end{equation}
\end{corollary}

\begin{proof}
    See \S\ref{subsec:scale-sharp-phi-SDPI}.
\end{proof}

Applying \eqref{eq:scale-sharp-phi-SDPI} to \(\eta_\tau=\nu_{\exp_x(rV),t;t-\tau}\) and then using \eqref{eq:infinitesimal-divergence-is-fisher} as \(r\to0\), gives
\[
    g_{\tau_1}^F(V,V)\le g_{\tau_0}^F(V,V),
    \qquad
    0<\tau_0<\tau_1.
\]
Thus Fisher monotonicity is the linearization of the strong data-processing inequality. 

\begin{remark}[Second scale variation of \(D_{\varphi_{\rm KL}}\)]
    We can similarly obtain the second scale variation of the Fisher metric in Proposition \ref{prop:second-scale-variation-of-the-fisher-metric} from the second scale variation of \(D_{\varphi_{\rm KL}}\). To that end, fix \(t\), and let \(u_\tau,v_\tau>0\) be normalized solutions of \(\partial_\tau w_\tau = \Delta_{g_{t-\tau}}w_\tau-R_{g_{t-\tau}}w_\tau\). Set
    \begin{align*}
        d\mu_\tau^u&:=u_\tau\,dg_{t-\tau},
        &
        d\mu_\tau^v&:=v_\tau\,dg_{t-\tau},
        &
        h_\tau&:=\log\frac{u_\tau}{v_\tau}, 
    \end{align*}
    write
    \[
        u_\tau=(4\pi\tau)^{-n/2}e^{-f_\tau^u},
        \qquad
        v_\tau=(4\pi\tau)^{-n/2}e^{-f_\tau^v},
    \]
    and define the shrinker defect tensors
    \begin{align*}
        S_\tau^u
        :=
        \Ric_{g_{t-\tau}}+\nabla^2f_\tau^u
        -\frac1{2\tau}g_{t-\tau},
        \qquad
        S_\tau^v
        :=
        \Ric_{g_{t-\tau}}+\nabla^2f_\tau^v
        -\frac1{2\tau}g_{t-\tau}.
    \end{align*}
    Then using \eqref{eq:divergence-monotonicity} we can compute
    \begin{equation}\label{eq:relative-entropy-second-variation}
    \begin{aligned}
        \frac{d^2}{d\tau^2} D_{\varphi_{\rm KL}}(\mu_\tau^u\mid \mu_\tau^v)
        &=
        2\int_M|S_\tau^u-S_\tau^v|^2\,d\mu_\tau^u
        +\frac2\tau\int_M|\nabla h_\tau|^2\,d\mu_\tau^u
        +4\int_M
        S_\tau^v(\nabla h_\tau,\nabla h_\tau)\,d\mu_\tau^u.
    \end{aligned}
    \end{equation}
    Using
    \[
        u_\tau^{(r)}
        =
        K(\exp_x(rV),t;\cdot,t-\tau),
        \qquad
        v_\tau
        =
        K(x,t;\cdot,t-\tau)
    \]
    in \eqref{eq:relative-entropy-second-variation}, dividing \eqref{eq:relative-entropy-second-variation} by \(r^2\), letting \(r\to0\), and then using \eqref{eq:gF-scale-equation-summary} recovers Proposition~\ref{prop:second-scale-variation-of-the-fisher-metric}.
\end{remark}

\begin{example}
    The class of \(\varphi\) in Proposition \ref{prop:sharp-phi-sobolev-heat-kernel} contains
    \[
        \varphi_{\rm KL}(r)=r\log r-r+1,
        \qquad
        \varphi_{\chi^2}(r)=(r-1)^2,
    \]
    and, more generally,
    \[
        \varphi_q(r)=r^q-1-q(r-1),
        \qquad
        1<q\le2.
    \]
    However, it does not contain the Hellinger distance.
\end{example}

\subsection{Dual monotonicity and \texorpdfstring{\(L^p\)}{Lp} estimates for source scores}

Using the \(\varphi\)-Sobolev inequality (Proposition \ref{prop:sharp-phi-sobolev-heat-kernel}), we also obtain a monotonicity formula for a dual version of \(\varphi\)-divergence. 

Fix \((x,t)\), and write
\[
    d\nu_\tau(y)
    =
    K(x,t;y,t-\tau)\,dg_{t-\tau}(y),
    \qquad
    z_V^\tau(y)
    :=
    2\tau\,\nabla_V^x\log K(x,t;y,t-\tau).
\]
For a fixed scale \(\tau\), consider the convex functional on the space of probability densities
\[
    D_{\varphi,\tau}(u)
    :=
    \begin{cases}
    \displaystyle
    \int_M\varphi(u)\,d\nu_\tau,
    &
    u>0,\quad \int_Mu\,d\nu_\tau=1,\\[1.2ex]
    +\infty,
    &
    \text{otherwise}
    \end{cases}
\]
Its Legendre--Fenchel transform, with respect to the pairing \(\langle h,u\rangle_\tau := \int_Mhu\,d\nu_\tau,\) is 
\begin{align}\label{eq:dual-varphi-divergence}
    D_{\varphi,\tau}^*(h)
    :=
    \sup_{\substack{u>0\\ \int_Mu\,d\nu_\tau=1}}
    \left\{
        \int_M h u\,d\nu_\tau
        -
        D_\varphi(u\nu_\tau\mid\nu_\tau)
    \right\}.
\end{align}

Since replacing \(\varphi(r)\) by \(\varphi(r)+c(r-1)\) changes neither \(D_{\varphi,\tau}\) on normalized densities nor \(D_{\varphi,\tau}^*\), we henceforth normalize \(\varphi'(1)=0\) throughout this subsection.

\begin{proposition}[Dual monotonicity]
\label{prop:dual-nonlinear-fisher-monotonicity}
Let \(\varphi\) satisfy the hypotheses of
Proposition~\ref{prop:sharp-phi-sobolev-heat-kernel}. Assume, in addition, that
\begin{equation}
\label{eq:legendre-type-condition}
    \varphi'(0+)=-\infty,
    \qquad
    \varphi'(\infty)=+\infty .
\end{equation}
Fix \(V\in T_xM\). For each \(\tau>0\), let \(u_\tau\) be the maximizing density in
\eqref{eq:dual-varphi-divergence}. Then
\begin{equation}
\label{eq:dual-nonlinear-fisher-dissipation}
\begin{aligned}
    -\frac{d}{d\tau}
    \left(
        \frac{1}{2\tau}
        D_{\varphi,\tau}^*(z_V^\tau)
    \right)
    &=
    \frac{1}{2\tau}
    \left[
        \int_M
        \varphi''(u_\tau)
        |\nabla u_\tau|_{g_{t-\tau}}^2
        \,d\nu_\tau
        -
        \frac1\tau
        D_\varphi(u_\tau\nu_\tau\mid\nu_\tau)
    \right]
    \ge 0 .
\end{aligned}
\end{equation}
Consequently,
\begin{equation}
\label{eq:dual-nonlinear-fisher-endpoint}
    \frac{\varphi''(1)}{2\tau}
    D_{\varphi,\tau}^*(z_V^\tau)
    \le
    \frac12 |V|_{g_t}^2 .
\end{equation}
\end{proposition}

\begin{proof}
    See \S\ref{subsec:dual-nonlinear-fisher-monotonicity}.
\end{proof}

\begin{remark}
The condition \eqref{eq:legendre-type-condition} suffices to ensure that the maximizer in \eqref{eq:dual-varphi-divergence} exists. For details, see \S\ref{subsec:dual-nonlinear-fisher-monotonicity}.
\end{remark}

\begin{remark}
For \(\varphi(r)=r\log r-r+1,\) the dual functional is
\[
    D_{\varphi,\tau}^*(z_V^\tau)
    =
    \log\int_M e^{z_V^\tau}\,d\nu_\tau .
\]
Thus Proposition~\ref{prop:dual-nonlinear-fisher-monotonicity} implies that
\begin{align*}
    \tau \longmapsto \frac1{2\tau} \log\int_M e^{z_V^\tau}\,d\nu_\tau
\end{align*}
is non-increasing, and that the following inequality holds:
\[
    \log\int_M e^{z_V^\tau}\,d\nu_\tau
    \le
    \tau |V|_{g_t}^2 .
\]
Consequently, for \(p\ge1\), we recover the \(L^p\)-estimates for the logarithmic derivative of the heat kernel \cite[Proposition 4.2]{Bamler2020A}:
\begin{equation}
    \|z_V^\tau\|_{L^p(d\nu_\tau)}
    \le
    C\sqrt{p\tau}\,|V|_{g_t},
\end{equation}
where \(C\) is a universal constant. This estimate with the sharp Euclidean constant follows from the convex-order estimates for logarithmic derivatives of the conjugate heat kernel; see Koirala~\cite{koirala2026sharp}.
\end{remark}

\subsection{Spatial estimates for \texorpdfstring{$\varphi$}{varphi}-divergences}

We next estimate the dependence of
\[
    x\longmapsto \nu_{x,t;t-\tau}
\]
on the base point with respect to \(\varphi\)-divergences.

First, the strong data-processing inequality and short-time heat-kernel estimates give a relative entropy estimate.

\begin{corollary}[KL spatial estimate]
\label{cor:KL-spatial-estimate}
    For all \(x_0,x_1\in M\) and all \(\tau>0\) with \(t-\tau\in I\),
    \[
        D_{\varphi_{\rm KL}}
        \left(
        \nu_{x_1,t;t-\tau}
        \mid
        \nu_{x_0,t;t-\tau}
        \right)
        \le
        \frac{d_t^2(x_0,x_1)}{4\tau}.
    \]
\end{corollary}

\begin{proof}
    See \S\ref{subsec:KL-spatial-estimate}.
\end{proof}

\begin{remark}
    Let \(\gamma_{x,2\tau}\) be the Euclidean Gaussian with mean \(x\) and covariance \(2\tau I\). Then
    \[
        D_{\varphi_{\rm KL}}
        \left(
        \gamma_{x_1,2\tau}
        \mid
        \gamma_{x_0,2\tau}
        \right)
        =
        \frac{|x_1-x_0|^2}{4\tau},
    \]
    so Corollary~\ref{cor:KL-spatial-estimate} is sharp.
\end{remark}

Combining Corollary~\ref{cor:KL-spatial-estimate} with the heat-kernel Talagrand inequality gives a density-divergence proof of Wasserstein spatial estimates and monotonicity for \(1\le p\le2\). The cases \(p=1\) and \(p=2\) recover Bamler's \(W_1\)-monotonicity \cite[Lemma~2.7]{Bamler2020A} and the McCann--Topping \(W_2\)-monotonicity \cite{MR2666905}. The \(W_p\)-monotonicity for all \(p\ge1\) follows by an alternative coupling argument of Arnaudon--Coulibaly--Thalmaier \cite[Theorem~4.1]{MR2790368}; the restriction \(p\le2\) below is only a limitation of our argument, which passes through \(W_2\). Here \(W_{p,g_s}\) denotes the \(p\)-Wasserstein distance with respect to the time-slice distance \(d_{g_s}\).

\begin{corollary}[Wasserstein spatial estimate and monotonicity]
\label{cor:wasserstein-from-KL}
    For all \(x_0,x_1\in M\) and all \(\tau>0\) with \(t-\tau\in I\),
    \[
        W_{1,g_{t-\tau}}
        \left(
        \nu_{x_1,t;t-\tau},
        \nu_{x_0,t;t-\tau}
        \right)
        \le
        W_{2,g_{t-\tau}}
        \left(
        \nu_{x_1,t;t-\tau},
        \nu_{x_0,t;t-\tau}
        \right)
        \le
        d_t(x_0,x_1).
    \]
    More generally, if \(d\mu_i(s)=v_i(\cdot,s)\,dg_s\), \(i=0,1\), are positive normalized conjugate heat flows on an interval \(I'\subset I\), then for every \(1\le p\le2\) and every \(s_1<s_0\) in \(I'\),
    \[
        W_{p,g_{s_1}}
        \left(
        \mu_0(s_1),
        \mu_1(s_1)
        \right)
        \le
        W_{p,g_{s_0}}
        \left(
        \mu_0(s_0),
        \mu_1(s_0)
        \right).
    \]
    In particular, \(s\mapsto W_{p,g_s}(\mu_0(s),\mu_1(s))\) is nondecreasing.
\end{corollary}

\begin{proof}
    See \S \ref{subsec:wasserstein-from-KL}.
\end{proof}

We also obtain a spatial estimate for \(D_\varphi\) under a derivative
bound on \(\varphi^{1/p}\).

\begin{proposition}[Spatial \(\varphi\)-divergence estimate]
\label{prop:spatial-phi-divergence-estimate}
    Let \(1\le p\le2\). Let \(\varphi:(0,\infty)\to[0,\infty)\) be locally absolutely continuous, with \(\varphi(1)=0\), and assume that \(\varphi^{1/p}\) is locally absolutely continuous. Set
    \[
        A_{\varphi,p}
        :=
        \operatorname*{ess\,sup}_{r>0}
        r^{p-1}
        \left|
        \frac{d}{dr}\varphi(r)^{1/p}
        \right|^p. 
    \]
    If \(A_{\varphi,p}<\infty\), then for all \(x_0,x_1\in M\) and all \(\tau>0\) with \(t-\tau\in I\),
    \begin{equation}\label{eq:spatial-phi-divergence-estimate}
        D_\varphi
        \left(
        \nu_{x_1,t;t-\tau}
        \mid
        \nu_{x_0,t;t-\tau}
        \right)
        \le
        A_{\varphi,p}
        \left(
        \frac{d_t(x_0,x_1)}{\sqrt{2\tau}}
        \right)^p. 
    \end{equation}
    The same estimate holds with \(x_0\) and \(x_1\) interchanged.
\end{proposition}

\begin{proof}
    See \S\ref{subsec:spatial-phi-divergence-estimate}.
\end{proof}

\begin{remark}
    The hypothesis \(A_{\varphi,p}<\infty\) excludes \(\chi^2\), and this exclusion is necessary:
    \[
        D_{\varphi_{\chi^2}}
        \left(
        \gamma_{x_1,2\tau}
        \mid
        \gamma_{x_0,2\tau}
        \right)
        =
        \exp\left(\frac{|x_1-x_0|^2}{2\tau}\right)-1. 
    \]
    Thus no global estimate of the form
    \[
        D_{\varphi_{\chi^2}}
        \le
        C\,\left(\frac{|x_1-x_0|}{\sqrt\tau}\right)^{p}
    \]
    can hold uniformly.
\end{remark}

The endpoint cases \(p=2\) and \(p=1\) give the Hellinger and total variation estimates.

\begin{corollary}[Hellinger and total variation]
\label{cor:hellinger-tv-from-phi}
    For pointed heat kernels,
    \begin{subequations}
    \begin{align}
        \label{eq:hell-yeah-ger-estimate}
        d_{\rm H}
        \left(
        \nu_{x_1,t;t-\tau},
        \nu_{x_0,t;t-\tau}
        \right)
        &\le
        \frac{d_t(x_0,x_1)}{4\sqrt\tau},\\
        {\rm TV}
        \left(
        \nu_{x_1,t;t-\tau},
        \nu_{x_0,t;t-\tau}
        \right)
        &\le
        \frac{d_t(x_0,x_1)}{2\sqrt{2\tau}}. 
    \end{align}
    \end{subequations}
\end{corollary}

\begin{proof}
    See \S\ref{subsec:hellinger-tv-from-phi}.
\end{proof}

\begin{remark}
    In the Euclidean Gaussian model case,
    \[
        d_{\rm H}^2
        \left(
        \gamma_{x_1,2\tau},
        \gamma_{x_0,2\tau}
        \right)
        =
        1-
        \exp\left(-\frac{|x_1-x_0|^2}{16\tau}\right),
        \]
    so the constant \((4\sqrt{\tau})^{-1}\) in \eqref{eq:hell-yeah-ger-estimate} is infinitesimally sharp.
\end{remark}

\section{Large Fisher eigenvalues and heat-score splitting}
\label{sec:large-fisher-eigenvalues}

\subsection{Fisher deficits and good scales}
\label{subsec:fisher-good-scales}

Fix a spacetime point \(\mathbf x=(x,t)\). All endomorphisms in this subsection act on \(T_xM\), and all traces and inner products are computed using \(g_t\). The scale variable \(\tau>0\) is always a backward time scale.

If \(\Acal:T_xM\to T_xM\) is \(g_t\)-self-adjoint and \(\Pcal\subset T_xM\) is a \(k\)-plane, define
\begin{equation}\label{eq:trace11}
    \tr_{\Pcal}\Acal
    :=
    \sum_{i=1}^k g_t(\Acal e_i,e_i),
\end{equation}
where \(e_1,\ldots,e_k\) is any \(g_t\)-orthonormal basis of \(\Pcal\). This is independent of the chosen basis.

Let
\[
    1\ge
    \lambda_1(\mathbf x;\tau)
    \ge
    \cdots
    \ge
    \lambda_n(\mathbf x;\tau)
    \ge0
\]
be the eigenvalues of the Fisher endomorphism \(\Jcal_\tau(\mathbf x)\), listed with multiplicity. For \(k\in\{0,1,\ldots,n\}\), define the \emph{top Fisher deficit} by
\begin{equation}
    \dfplus{k}(\mathbf x;\tau)
    :=
    \sum_{i=1}^k
    \bigl(1-\lambda_i(\mathbf x;\tau)\bigr),
    \qquad
    \dfplus{0}(\mathbf x;\tau):=0. 
\end{equation} 
By the Ky Fan variational principle,
\begin{equation}\label{eq:sec6-top-deficit-Ky-Fan}
    \dfplus{k}(\mathbf x;\tau)
    =
    \min_{\substack{\Pcal\subset T_xM\\ \dim\Pcal=k}}
    \tr_{\Pcal}\bigl(I-\Jcal_\tau(\mathbf x)\bigr).
\end{equation}
A \emph{top Fisher \(k\)-plane} at \((\mathbf x,\tau)\) is any \(k\)-plane where the minimum in \eqref{eq:sec6-top-deficit-Ky-Fan} is attained; equivalently, it is spanned by eigenvectors corresponding to \(\lambda_1,\ldots,\lambda_k\). If there is multiplicity at the \(k\)-th eigenvalue, such a plane need not be unique. When a choice is needed, we may write it as \(\Pcal_k^F(\mathbf x;\tau)\).

If \(\dfplus{k}(\mathbf x;\tau)=0\), then \(\lambda_1(\mathbf x;\tau)=\cdots=\lambda_k(\mathbf x;\tau)=1\). Hence any top Fisher \(k\)-plane is contained in the nullspace of the Fisher defect \(D^F_\tau=g_t-g^F_\tau\). Thus Theorem~\ref{thm:fisher-equality-rigidity} implies  under completeness the universal cover splits off a Euclidean factor of dimension at least \(k\). In particular, on a connected closed Ricci flow, Corollary~\ref{cor:strict-closed-fisher} shows that this exact case cannot occur at positive scale.

The main goal of this section is to prove almost splitting when \(\delta_k^+\ll1\).

We use the scale identity
\begin{equation}\label{eq:scaleIdentity1-1tensor}
    \Gcal_\tau-\Jcal_\tau
    =
    -\partial_{\log\tau}\Jcal_\tau,
\end{equation}
which is Proposition~\ref{prop:ScaleEqFisherMetric} after composing with \(g_t^{-1}\), to propagate the bound \(\delta_k^+\ll1\) to a comparable smaller scale.

\begin{proposition}[Good comparable Fisher scale]
\label{prop:sec6-good-scale-fixed-plane}
    Fix \(\tau_0\), \(\theta\in(0,1)\), and let \(\Pcal\subset T_xM\) be a \(k\)-plane. Suppose \(\tr_{\Pcal}(I-\Jcal_{\tau_0})\le\delta.\) Then there exists \(\tau\in[\theta\tau_0,\tau_0]\) such that
    \begin{align}\label{eq:sec6-good-scale-estimates}
        \tr_{\Pcal}(I-\Jcal_\tau)\le\delta \qquad \text{and} \qquad  \tr_{\Pcal}(\Gcal_\tau-\Jcal_\tau)
        \le
        \frac{\delta}{\log(\theta^{-1})}.
    \end{align}
    In particular, if \(\dfplus{k}(\mathbf x;\tau_0)\le\delta\) and \(\Pcal\) is any top Fisher \(k\)-plane at scale \(\tau_0\), then there exists a scale \(\tau\in[\theta\tau_0,\tau_0]\) such that \eqref{eq:sec6-good-scale-estimates} holds for this plane.
\end{proposition}

\begin{proof}
    See \S~\ref{subsec:GoodCompScaleFisher}.
\end{proof}

\subsection{Canonical heat-score maps}
\label{subsec:canonical-heat-score-maps}

Throughout this subsection we fix a spacetime point $ \mathbf x=(x,t)$ and a time scale \(\tau>0\). We set $s:=t-\tau$. For notational convenience, we write for \(\ell\in[s,t)\)
\[
    d\nu_\ell(q)
    :=
    K(x,t;q,\ell)\,dg_\ell(q),
\]
and set $d\nu_t:=\delta_x$. Thus \(d\nu_\ell\) is the conjugate heat-kernel measure based at \((x,t)\) and evaluated at time \(\ell\).

\begin{definition}[Source scores and canonical heat-score maps]
\label{def:canonical-heat-score-map}
    For \(V\in T_xM\), recall the source score at scale \(\tau\) is \(z_V^\tau(y) = 2\tau\,\nabla_V^x\log K(x,t;y,s).\) Define the \emph{score heat evolution} \(v_V: M \times [s,t] \to \mathbb{R}\) by
    \begin{equation}\label{eq:score_heat_evolution}
        v_V(q,\ell)
        :=
        (P_{s,\ell}z_V^\tau)(q),
        \qquad s\le \ell\le t,
    \end{equation}
    so that \((\partial_\ell-\Delta_{g_\ell})v_V=0\) on \(M\times[s,t],\) with initial value \(v_V(\cdot,s)=z_V^\tau.\) Equivalently,
    \begin{equation}
        v_V(q,\ell)
        =
        \int_M K(q,\ell;y,s)z_V^\tau(y)\,dg_s(y).
    \end{equation} 
    
    If \(\Pcal\subset T_xM\) is a \(k\)-plane and \(e_1,\ldots,e_k\) is a \(g_t\)-orthonormal basis of \(\Pcal\), we write
    \begin{equation}
        z_i:=z_{e_i}^\tau,
        \qquad
        v_i:=v_{e_i}.
    \end{equation}
    The map
    \begin{equation}
        v=(v_1,\ldots,v_k):M\times[s,t]\to\R^k
    \end{equation} 
    is called the \emph{canonical heat-score map} associated to $(\mathbf x,\tau,\Pcal,e_1,\ldots,e_k)$.
\end{definition}

The point of the forward heat evolution is to turn the source-score observable \(z_V^\tau\) into a canonical parabolic coordinate function $v_V$ whose terminal gradient is \(\Jcal_\tau V\) and whose failure to be affine is measured exactly by its spacetime Hessian energy. In the model Euclidean case, \(v_V=z_V^\tau\) are time-independent linear coordinates.

We first prove that the heat score maps are centered; cf.~\cite[(2.9) in \S 2.2]{MR1800071}.

\begin{lemma}
\label{lem:sec6-forward-heat-score-centering}
    We have the \emph{heat-score centering identity}
    \begin{equation}
        \int_M v_V(\cdot,\ell)\,d\nu_\ell=0
        \qquad
        \text{for every }\ell\in[s,t].
    \end{equation} 
\end{lemma}

\begin{proof}
    See \S\ref{subsec:sec6-forward-heat-score-centering}.
\end{proof}

We shall also use the source-score covariance identity \eqref{eq:score-coordinate-cov-summary}
\begin{equation}\label{eq:sec6-source-score-covariance}
    \frac1{2\tau}
    \int_M z_V^\tau z_W^\tau\,d\nu_s
    =
    g_t(\Jcal_\tau V,W),
    \qquad V,W\in T_xM. 
\end{equation}
In particular, if \(|V|_{g_t}=1\), then Theorem \ref{thm:BamProp4.2} implies
\begin{equation}\label{eq:sec6-source-score-L2-bound}
    \tau^{-1}\int_M (z_V^\tau)^2\,d\nu_s
    =
    2g_t(\Jcal_\tau V,V)
    \le 2. 
\end{equation}

\subsection{Gradient identity}

We now prove an identity which allows us to use Fisher metric to control the gradient of the heat score maps; cf.~\cite[(2.8) in \S 2.2]{MR1800071}. 
\begin{lemma}
\label{lem:sec6-terminal-value-gradient}
    For all \(V\in T_xM\),
    \begin{equation}
        \nabla v_V(x,t)=\Jcal_\tau V.
    \end{equation}
    Equivalently, for every \(W\in T_xM\),
    \begin{equation}
        \nabla_Wv_V(x,t)
        =
        g_t(\Jcal_\tau V,W) =g_\tau^F(V,W). 
    \end{equation}
    Consequently, for all \(V,W\in T_xM\),
    \begin{equation}
        \left\langle
        \nabla v_V,\nabla v_W
        \right\rangle_{g_t}(x,t)
        =
        g_t(\Jcal_\tau V,\Jcal_\tau W).
    \end{equation}
\end{lemma}

\begin{proof}
    See \S \ref{subsec:sec6-terminal-value-gradient}.
\end{proof}

\subsection{Hessian identity}

The next identity is the analytic mechanism that converts Fisher information into Hessian control for the canonical heat-score functions.

\begin{lemma}[Weighted Dirichlet identity]
\label{lem:sec6-weighted-dirichlet-identity}
    Let \(u,w\) be smooth solutions of the forward heat equation \((\partial_\ell-\Delta_{g_\ell})u=0\) and \((\partial_\ell-\Delta_{g_\ell})w=0\) on \(M\times[s,t]\). Then
    \begin{equation}\label{eq:sec6-weighted-dirichlet-identity}
        \int_M
        \left\langle
        \nabla u,\nabla w
        \right\rangle_{g_s}
        \,d\nu_s
        -
        \left\langle
        \nabla u,\nabla w
        \right\rangle_{g_t}(x,t)
        =
        2\int_s^t\int_M
        \left\langle
        \nabla^2u,\nabla^2w
        \right\rangle_{g_\ell}
        \,d\nu_\ell\,d\ell. 
    \end{equation}
\end{lemma}

\begin{proof}
    It follows by integrating the polarized Bochner formula 
    \begin{equation}
        (\partial_\ell-\Delta_{g_\ell})
        \left\langle
        \nabla u,\nabla w
        \right\rangle_{g_\ell}
        =
        -2
        \left\langle
        \nabla^2u,\nabla^2w
        \right\rangle_{g_\ell}
    \end{equation} 
    against the conjugate heat-kernel measure and using
    \begin{align}
        \frac{d}{d\ell} \int_M
        \left\langle
        \nabla u,\nabla w
        \right\rangle_{g_\ell }
        \,d\nu_\ell= \int  (\partial_\ell-\Delta_{g_\ell})
        \left\langle
        \nabla u,\nabla w
        \right\rangle_{g_\ell}\,d\nu_\ell.
    \end{align}
    The Ricci terms in the ordinary Bochner formula cancel with the variation of the metric under the Ricci flow.
\end{proof}

As a consequence of \eqref{eq:target_Dirichlet_tensor}, Lemma \ref{lem:sec6-terminal-value-gradient} and Lemma \ref{lem:sec6-weighted-dirichlet-identity}, we obtain the following Hessian identity.

\begin{corollary}
\label{cor:sec6-operator-affineness-defect}
    For every \(V,W\in\Pcal\),
    \begin{align}
        G_\tau(V,W)
        -
        g_t(\Jcal_\tau V,\Jcal_\tau W)
        &=
        2\int_s^t\int_M
        \left\langle
        \nabla^2v_V,\nabla^2v_W
        \right\rangle_{g_\ell}
        \,d\nu_\ell\,d\ell.
    \end{align}
\end{corollary}

\begin{proof}
    See \S \ref{subsec:hessian_identity}.
\end{proof}

\subsection{Integral gradient estimate}

The following result converts Hessian control and a pointwise terminal gradient estimate into an integral gradient estimate.

\begin{proposition}
\label{prop:sec6-averaged-gradient-orthonormality}
    There exists a constant \(C=C(k)<\infty\) such that the following holds. Let
    \[
        v=(v_1,\ldots,v_k):M\times[s,t]\to\R^k
    \]
    be a \(k\)-tuple of forward heat solutions. Suppose that \(\delta_1,\delta_2\ge0\) and
    \begin{align}
    \label{eq:terminal-gradient-and-hessian-estimate-hypothesis}
        \max_{1\le i,j\le k}
        \left|
        \left\langle
        \nabla v_i,\nabla v_j
        \right\rangle_{g_t}(x,t)
        -
        \delta_{ij}
        \right|
        \le\delta_1, \qquad \text{and} \qquad
        \int_s^t\int_M
        \sum_{i=1}^k
        |\nabla^2v_i|_{g_\ell}^2
        \,d\nu_\ell\,d\ell
        \le\delta_2.
    \end{align}
    Then
    \begin{align}
        \tau^{-1}
        \int_s^t\int_M
        \sum_{i,j=1}^k
        \left|
        \left\langle
        \nabla v_i,\nabla v_j
        \right\rangle_{g_\ell}
        -
        \delta_{ij}
        \right|
        d\nu_\ell\,d\ell
        \le
        C(k)\left(
            \delta_1+\delta_2+
            \sqrt{\delta_2(1+\delta_1+2\delta_2)}
        \right).
    \end{align}
\end{proposition}

\begin{proof}
    See \S \ref{subsec:AvgGradOrthoFromHessEnergy}.
\end{proof}

\subsection{Fisher rank produces heat-kernel strong splitting}
\label{subsec:fisher-rank-produces-strong-splitting}

\begin{definition}[Heat-kernel strong splitting map]
\label{def:sec6-heat-kernel-strong-splitting-map}
    Let \(\eta>0\). A map
    \[
        v=(v_1,\ldots,v_k):M\times[t-\tau,t]\to\R^k
    \]
    is called a \emph{heat-kernel strong $(k,\eta,\tau)$-splitting map} at \(\mathbf x=(x,t)\) if the following conditions hold.
    \begin{enumerate}
        \item Each component \(v_i\) solves the heat equation \((\partial_\ell-\Delta_{g_\ell})v_i=0\) on \(M\times[t-\tau,t].\)
    
        \item Each component is centered:
        \[
            \int_M v_i(\cdot,\ell)\,d\nu_\ell=0
            \qquad
            \text{for every }\ell\in[t-\tau,t].
        \]
    
        \item The initial heat-kernel \(L^2\)-size is controlled:
        \[
            \tau^{-1}
            \int_M v_i(\cdot,t-\tau)^2\,d\nu_{t-\tau}
            \le
            2
            \qquad
            \text{for every }i.
        \]
    
        \item The three splitting errors satisfy
        \begin{align*}
            &\max_{1\le i,j\le k}\left|
            \left\langle
            \nabla v_i,\nabla v_j
            \right\rangle_{g_t}(x,t)
            -
            \delta_{ij}
            \right|+\tau^{-1}
            \int_{t-\tau}^t\int_M
            \sum_{i,j=1}^k
            \left|
            \left\langle
            \nabla v_i,\nabla v_j
            \right\rangle_{g_\ell}
            -
            \delta_{ij}
            \right|
            d\nu_\ell\,d\ell\\
            &\qquad +\int_{t-\tau}^t\int_M
            \sum_{i=1}^k
            |\nabla^2v_i|_{g_\ell}^2
            \,d\nu_\ell\,d\ell\leq \eta,
        \end{align*}
    \end{enumerate}
\end{definition}

We are now ready to prove closeness of the Fisher endomorphism to identity leads to pointwise gradient estimate of the heat score maps as well as Hessian estimates.

\begin{proposition}
\label{prop:sec6-terminal-orthonormality-from-fisher-pinching}
    Let \(\Pcal\subset T_xM\) be a \(k\)-plane, and suppose \(\tr_\Pcal(I-\Jcal_\tau)\le\delta.\) Let \(e_1,\ldots,e_k\) be any \(g_t\)-orthonormal basis of \(\Pcal\), and put \(v_i:=v_{e_i}\). Then
    \begin{align}
        \left|
        \left\langle
        \nabla v_i,\nabla v_j
        \right\rangle_{g_t}(x,t)
        -
        \delta_{ij}
        \right|
        \le
        2\delta
        \qquad
        \text{for all }i,j.
    \end{align}
    If, in addition, \(\tr_{\Pcal}(\Gcal_\tau-\Jcal_\tau)\le\delta_1,\) then
    \begin{equation}\label{eq:sec6-hessian-energy-small-bound}
        2\int_{t-\tau}^t\int_M
        \sum_{i=1}^k
        |\nabla^2v_i|_{g_\ell}^2
        \,d\nu_\ell\,d\ell
        \le
        \delta+\delta_1. 
    \end{equation}
\end{proposition}

\begin{proof}
    See \S \ref{subsec:canonical_estimates_from_fisher_production}.
\end{proof}

We now combine the good-scale selection Proposition \ref{prop:sec6-good-scale-fixed-plane}, the canonical heat-score identities in \S \ref{subsec:canonical-heat-score-maps}, and the averaged-gradient estimate Proposition \ref{prop:sec6-terminal-orthonormality-from-fisher-pinching} to obtain a canonical heat-kernel strong splitting map.

\begin{theorem}
\label{thm:sec6-fisher-rank-produces-strong-splitting-map}
    Fix \(k\in\{1,\ldots,n\}\), \(\theta\in(0,1)\), and \(\delta\in (0,1)\). Suppose that \(\dfplus{k}(\mathbf x;\tau_0)\le\delta,\)  and let \(\Pcal=\Pcal_k^F(\mathbf x;\tau_0)\)  be any top Fisher \(k\)-plane at scale \(\tau_0\). Then there exists \(\tau\in[\theta\tau_0,\tau_0]\) such that, for any \(g_t\)-orthonormal basis of \(\Pcal\), the associated canonical heat-score map is a heat-kernel strong $(k,C(\theta)\sqrt{\delta},\tau)$-splitting map at \(\mathbf x\).
\end{theorem}

\begin{proof}
    See \S \ref{subsec:proof_fisher_rank_strong_splitting}.
\end{proof}

\section{Top Fisher deficit and Nash entropy}
\label{sec:fisher-epsilon-regularity}

In this section we establish a quantitative equivalence, at comparable scales, between the smallness of the pointed Nash deficit \(-\Nash\) and that of the full Fisher deficit. We first prove that a small full Fisher deficit forces \(-\Nash\) to be small at a comparable scale. A qualitative form of this implication was communicated to us by Yongjia Zhang, who proved it using
\(\mathbb F\)-compactness \cite{MR4623543} and a contradiction argument. The quantitative statement and the proof below are inspired by Zhang's result and by the arguments in \cite[Propositions~14.1 and 14.2]{Bamler2020C}.

\begin{theorem}\label{thm:top-fisher-epsilon-regularity}
    There are dimensional constants \(c_n>0,\) \(\delta_n>0,\) \(C_n<\infty\) with the following property. Let $(M^n,g_\ell)_{\ell\in[t-\tau_0,t]}$ be a closed Ricci flow and let $\mathbf x=(x,t)$. If \(\delta_n^+(\mathbf x;\tau_0) \leq\delta\leq\delta_n,\) then
    \begin{equation}
    \label{eq:main-entropy-conclusion}
        -C_n\delta^{1/8}
        \leq
        \Nash_{\mathbf x}(c_n\tau_0)
        \leq0.
    \end{equation}
\end{theorem}

\begin{proof}
    See \S \ref{subsec:top-fisher-epsilon-regularity}.
\end{proof}

We also have two forms of converse to Theorem \ref{thm:top-fisher-epsilon-regularity}. We first prove an average converse.

\begin{proposition}
\label{prop:nash-gives-good-fisher-scale}
    Let \((M^n,(g_\ell)_{\ell\in I})\) be a closed Ricci flow, let
    \(\mathbf x=(x,t)\), and let \(\tau>0\) satisfy
    \([t-\tau,t]\subset I\). Put \(s:=t-\tau\). Then, for every
    \(\theta\in(0,1)\), there exists \(\rho\in(\theta\tau,\tau)\)
    such that
    \begin{equation}
    \label{eq:nash-gives-good-fisher-scale}
        \int_M
        \dfplus{n}\bigl((y,s+\rho);\rho\bigr)
        \,d\nu_{x,t;s+\rho}(y)
        \le
        \frac{2\bigl(-\Nash_{\mathbf x}(\tau)\bigr)}
        {\log(\theta^{-1})}.
    \end{equation}
\end{proposition}

\begin{proof}
    The proof follows from the following identity that relates Nash entropy with the top Fisher deficit. Let \(s<t\), let \(\mathbf x=(x,t)\), and put \(\tau=t-s\). Then
    \begin{equation}
    \label{eq:nash-integrated-fisher-deficit}
         -\Nash_{\mathbf x}(\tau)
         =
         \frac12\int_s^t\frac1{r-s}
         \int_M
         \delta_n^+\bigl((y,r);r-s\bigr)
         \,d\nu_{x,t;r}(y)\,dr.
    \end{equation}
    See \S \ref{subsec:nash-gives-good-fisher-scale}.
\end{proof}

When the Nash entropy is close to zero, we can convert the average smallness of the top Fisher deficit into a pointwise smallness using a heat-kernel lower bound and gradient estimates for the top Fisher deficit.

\begin{theorem}
\label{thm:quantitative-nash-to-fisher}
    There are dimensional constants \(c_n\in(0,1),\) \(\varepsilon_n>0,\) \(C_n<\infty\) with the following property. Let \((M^n,g_\ell)_{\ell\in[t-\tau_0,t]}\) be a closed Ricci flow and let \(\mathbf x=(x,t)\). If \(\Nash_{\mathbf x}(\tau_0)\ge-\varepsilon_n,\) then
    \begin{equation}
    \label{eq:quantitative-nash-to-fisher}
        \dfplus{n}(\mathbf x;c_n\tau_0)
        \le
        C_n\bigl(-\Nash_{\mathbf x}(\tau_0)\bigr)^{1/2}.
    \end{equation}
\end{theorem}

\begin{proof}
See \S\ref{subsec:quantitative-nash-to-fisher}.
\end{proof}

\begin{remark}
    The analogue of
    Theorem~\ref{thm:top-fisher-epsilon-regularity} fails for the top
    \((n-1)\)-Fisher deficit. Indeed, consider the static flat Ricci flow on
    \[
        M_{L,\epsilon}:=(S_L^1)^{n-1}\times S_\epsilon^1,
    \]
    where the subscripts denote the radii. For every fixed \(\tau>0\),
    the product structure implies
    \[
        0\le \dfplus{n-1}(\mathbf x;\tau)
        \le (n-1)\bigl(1-\lambda_L(\tau)\bigr)\longrightarrow0
        \qquad\text{as }L\to\infty,
    \]
    where \(\lambda_L(\tau)\) is the Fisher eigenvalue in an
    \(S_L^1\)-direction. On the other hand, pointed Nash entropy is additive
    under products, and the Nash entropy of \(S_L^1\) and \(S_{\epsilon}^1\) satisfy
    \[
        \Nash_{S_L^1}(\tau)\longrightarrow0
        \quad(L\to\infty),
        \qquad
        \Nash_{S_\epsilon^1}(\tau)
        =
        \log\!\left(\frac{\sqrt\pi\,\epsilon}{\sqrt{e\tau}}\right)+o(1)
        \longrightarrow-\infty
        \quad(\epsilon\to0).
    \]
    Thus one may choose \(L_j\to\infty\) and \(\epsilon_j\to0\) so that
    \(\dfplus{n-1}(\mathbf x;\tau)\to0\) while
    \(\Nash_{\mathbf x}(\tau)\to-\infty\).
\end{remark}

However, when we impose a lower bound on the Nash entropy, we do obtain a regularity consequence of the codimension-one Fisher splitting theorem. The small top \((n-1)\)-Fisher deficit is used to produce the heat-kernel strong splitting map. The resulting curvature-radius bound then follows from Bamler's \(\varepsilon\)-regularity theorem \cite[Proposition 16.1]{Bamler2020C}.

\begin{definition}[Curvature radius]
\label{def:sec6-curvature-radius}
    The \emph{parabolic curvature radius} at a spacetime point \(\mathbf x=(x,t)\) is
    \[
        r_{\Rm}(\mathbf x)
        :=
        \sup
        \left\{
        r>0:
        |\Rm|\le r^{-2}
        \text{ on }P(x,t,r,-r^2)
        \right\},
    \]
    where \(P(x,t,r,-r^2)\) denotes the backward parabolic neighborhood.
\end{definition}

\begin{theorem}[Fisher \(\varepsilon\)-regularity from \(n-1\) large eigenvalues]
\label{thm:sec6-fisher-epsilon-regularity-n-minus-one}
    Fix \(n\ge2\) and \(Y<\infty\). There exist constants \(\delta=\delta(n,Y)>0,\) \(\kappa=\kappa(n,Y)>0\) with the following property. Let \((M^n,(g_\ell)_{\ell\in I})\) be a closed Ricci flow, let \(\mathbf x=(x,t)\), and let \(\tau>0\) be such that \([t-\tau,t]\subset I. \) Suppose that \(\Nash_{\mathbf x}(\tau)\ge -Y\) and \(\dfplus{n-1}(\mathbf x;\tau)\le\delta. \) Then \(r_{\Rm}(\mathbf x)\ge \kappa\sqrt{\tau}.\)
\end{theorem}

\begin{proof}
    See \S\ref{subsec:FullRankFishEpsReg}.
\end{proof}

\section{Proofs}\label{sec7}

In this section we give proofs of the results in this paper.

\subsection{Proof of Lemma~\ref{lem:FisherMetricAsTauToZero}}
\label{subsec:FisherMetTauToZero}

We need the following short-time lower bound on the heat kernel. It is obtained from the local heat-kernel expansion for evolving metrics by the standard chain-of-balls argument; we include the details because the estimate is used with the sharp coefficient in the exponent.

\begin{lemma}[Short-time lower bound with sharp exponent]
\label{lem:short-time-sharp-lower-bound}
    Let \((M,g_s)\) be a smooth Ricci flow on a closed manifold. Fix \(t\in I\). For every \(\eta>0\) there are \(C_\eta<\infty\) and \(\varepsilon_\eta>0\) such that, for all \(x,y\in M\) and \(0<\varepsilon<\varepsilon_\eta\),
    \[
        K(x,t;y,t-\varepsilon)
        \ge
        C_\eta^{-1}\varepsilon^{-n/2}
        \exp\left(
        -\frac{d_t^2(x,y)+\eta}{4\varepsilon}
        \right).
    \]
    Moreover, \(\nu_{x,t;t-\varepsilon}\rightharpoonup\delta_x\) as \(\varepsilon\downarrow0.\)
\end{lemma}

\begin{proof}
    The local heat-kernel expansion for evolving metrics \cite[Theorem~24.21]{MSM163} gives the following local lower bound. For every \(\alpha>0\) there are \(r>0\), \(c_\alpha>0\), and \(\varepsilon_\alpha>0\) such that, whenever \(d_t(a,b)<r\), \(t_1\in[t-\varepsilon_\alpha,t]\), and \(0<h<\varepsilon_\alpha\),
    \[
        K(a,t_1;b,t_1-h)
        \ge
        c_\alpha h^{-n/2}
        \exp\left(
        -\frac{(1+\alpha)d_t^2(a,b)}{4h}
        \right).
    \]
    Here the smooth variation of \(g_s\) on \([t-\varepsilon_\alpha,t]\) and the zeroth-order term in the conjugate heat equation are absorbed into the factor \(1+\alpha\) and the constant \(c_\alpha\).
    
    We pass from this local estimate to the global one by a standard chaining argument. Let \(D:=\operatorname{diam}_{g_t}M. \) Choose \(\alpha>0\) so small that \(\alpha D^2<\eta/4\). Then choose \(N=N(\eta)\) so large that every minimizing \(g_t\)-geodesic can be divided into \(N\) subsegments of length \(<r/4\). Put
    \[
        h:=\frac{\varepsilon}{N},
        \qquad
        t_j:=t-jh,
        \qquad
        j=0,\ldots,N. 
    \]
    Let \(p_0=x,\ldots,p_N=y\) be an equally spaced subdivision of a minimizing \(g_t\)-geodesic from \(x\) to \(y\). By the semigroup property,
    \[
        K(x,t;y,t-\varepsilon)
        =
        \int_{M^{N-1}}
        \prod_{j=0}^{N-1}
        K(q_j,t_j;q_{j+1},t_{j+1})
        \prod_{j=1}^{N-1}dg_{t_j}(q_j),
    \]
    where \(q_0=x\) and \(q_N=y\). Restrict the integration to \(q_j\in B_{g_t}(p_j,\theta\sqrt h),\) \(j=1,\ldots,N-1,\) where \(\theta>0\) is fixed. For \(\varepsilon\) sufficiently small, the local lower bound applies to each factor. Moreover, since \(M\) is compact and \(g_s\to g_t\) smoothly,
    \[
        \operatorname{vol}_{g_{t_j}}
        B_{g_t}(p_j,\theta\sqrt h)
        \ge c h^{n/2}.
    \]
    Hence
    \[
        K(x,t;y,t-\varepsilon)
        \ge
        C^{-1}h^{-n/2}
        \exp\left(
        -\frac{1+\alpha}{4h}
        \sup
        \sum_{j=0}^{N-1}d_t^2(q_j,q_{j+1})
        \right),
    \]
    where the supremum is over the restricted integration region. For such \(q_j\),
    \[
        d_t(q_j,q_{j+1})
        \le
        \frac{d_t(x,y)}{N}+2\theta\sqrt h. 
    \]
    Therefore
    \[
        N\sum_{j=0}^{N-1}d_t^2(q_j,q_{j+1})
        \le
        d_t^2(x,y)+O_{N,\theta}(\sqrt\varepsilon).
    \]
    Since \(h=\varepsilon/N\), after decreasing \(\varepsilon_\eta\) we get
    \[
        \frac{1+\alpha}{4h}
        \sum_{j=0}^{N-1}d_t^2(q_j,q_{j+1})
        \le
        \frac{d_t^2(x,y)+\eta}{4\varepsilon}.
    \]
    Absorbing the fixed powers of \(N\), \(\theta\), and the local constants into \(C_\eta\), we obtain
    \[
        K(x,t;y,t-\varepsilon)
        \ge
        C_\eta^{-1}\varepsilon^{-n/2}
        \exp\left(
        -\frac{d_t^2(x,y)+\eta}{4\varepsilon}
        \right).
    \]
    
    The convergence \(\nu_{x,t;t-\varepsilon}\rightharpoonup\delta_x \) is standard.
\end{proof}

We next record the short-time estimates for the source scores.

\begin{lemma}[Short-time source-score estimates]
\label{lem:short-time-source-score-moments}
Fix \((x,t)\), and let \(V,W\in T_xM\). As \(\rho\downarrow0\),
\begin{subequations}
\begin{align}
    \|z_V^\rho\|_{L^\infty(M)}
        &=O(1),
    \label{eq:small-score-linfty}\\
    \int_M z_V^\rho z_W^\rho\,d\nu_\rho
        &=2\rho\,\langle V,W\rangle_{g_t}+o(\rho),
    \label{eq:small-score-second}\\
    \int_M |z_V^\rho|^p\,d\nu_\rho
        &=O(\rho^{p/2})
        \qquad\text{for every }p\ge3 .
    \label{eq:small-score-lp}
\end{align} 
\end{subequations}
\end{lemma}

\begin{proof}
Write \(g:=g_t\), and choose \(r_0>0\) smaller than the \(g\)-injectivity radius
at \(x\).

The estimate \eqref{eq:small-score-linfty} follows from the following claim.
\begin{claim}\label{claim:improved-pointwise-estimate-local}
    We have
    \begin{subequations}
    \begin{align}
    \label{eq:local-z-score-bound}
        |z_V^\rho(y)|
        &\le
        C_V\bigl(d_g(x,y)+\rho\bigr)
        \qquad
        \text{on }B_g(x,r_0),\\
        |z_V^\rho(y)|&\le C_V \qquad \text{on }M.\label{eq:off-diagonal-z-score-bound}
    \end{align}
    \end{subequations}
\end{claim}

\begin{proof}[Proof of Claim \ref{claim:improved-pointwise-estimate-local}]
On \(B_g(x,r_0)\), the short-time parametrix for the heat kernel on a smooth time-dependent background gives
\begin{align}\label{eq:local-heat-kernel-parametrix}
    K(x,t;y,t-\rho)
    =
    (4\pi\rho)^{-n/2}
    \exp\left(-\frac{d_g(x,y)^2}{4\rho}\right)
    \Psi(x,y,\rho),
\end{align}
where \(\Psi\) is smooth and positive, with \(\Psi(x,x,0)=1\); see
\cite[Theorem~24.21]{MSM163}. Differentiating in the source variable gives
\[
    \nabla_V^x\log K
    =
    -\frac1{4\rho}\nabla_V^x d_g(x,y)^2
    +
    O_V(1).
\]
Since \(\nabla_x\left(\frac12 d_g(x,y)^2\right) = -\exp_x^{-1}(y),\) we obtain, uniformly on \(B_g(x,r_0)\),
\begin{equation}
\label{eq:local-z-score-unscaled}
    z_V^\rho(y)
    =
    2\rho\,\nabla_V^x\log K(x,t;y,t-\rho)
    =
    \langle V,\exp_x^{-1}(y)\rangle_g
    +
    O_V(\rho).
\end{equation}
In particular, \eqref{eq:local-z-score-bound} holds.

We next prove a crude bound away from \(x\). Fix \(y\in M\), and set
\[
    u(w,q):=K(w,q;y,t-\rho),
    \qquad
    t-\rho<q\le t.
\]
Then \(u\) solves the forward heat equation in the source variables. We shall use
Hamilton's elementary gradient estimate: if \(u>0\) solves \(\partial_q u=\Delta_{g_q}u\) on \(M\times[q_0,q_1]\), and \(u\le A\), then
\begin{equation}
\label{eq:hamilton-gradient-estimate-short-time}
    (q-q_0)|\nabla\log u|_{g_q}^2
    \le
    \log\frac{A}{u}
    \qquad
    \text{for }q_0<q\le q_1 .
\end{equation}
Indeed, for \(f:=\log(A/u)\), one has \((\partial_q-\Delta)f=-|\nabla f|^2.\) Because \(g_q\) evolves by Ricci flow, the Ricci terms cancel in the Bochner
formula:
\[
    (\partial_q-\Delta)|\nabla f|^2
    =
    -2\langle\nabla f,\nabla|\nabla f|^2\rangle
    -
    2|\nabla^2f|^2 .
\]
Applying the maximum principle to \((q-q_0)|\nabla f|^2-f\) gives \eqref{eq:hamilton-gradient-estimate-short-time}.

Apply \eqref{eq:hamilton-gradient-estimate-short-time} on
\([t-\rho/2,t]\). The standard short-time upper bound gives
\[
    \sup_{w\in M,\ q\in[t-\rho/2,t]}
    K(w,q;y,t-\rho)
    \le
    C\rho^{-n/2}.
\]
On the other hand, Lemma~\ref{lem:short-time-sharp-lower-bound}, with one fixed
choice of \(\eta>0\), gives
\[
    K(x,t;y,t-\rho)
    \ge
    C^{-1}\rho^{-n/2}
    \exp\left(
        -\frac{d_g(x,y)^2+\eta}{4\rho}
    \right).
\]
Since \(M\) is compact,
\[
    \log
    \frac{
        \sup_{w,\ q\in[t-\rho/2,t]}K(w,q;y,t-\rho)
    }{
        K(x,t;y,t-\rho)
    }
    \le
    \frac{C}{\rho}.
\]
Hamilton's estimate \eqref{eq:hamilton-gradient-estimate-short-time} therefore implies
\[
    |\nabla_x\log K(x,t;y,t-\rho)|_{g_t}
    \le
    \frac{C}{\rho},
\]
and hence \eqref{eq:off-diagonal-z-score-bound} holds.
\end{proof}

We now prove the \(L^p\)-estimate. By the standard Gaussian upper bound for the
heat kernel on a smooth closed background,
\begin{align}\label{eq:gaussian-upper-bound}
    d\nu_\rho(y)
    \le
    C\rho^{-n/2}
    \exp\left(
        -c\frac{d_g(x,y)^2}{\rho}
    \right)dg(y);
\end{align}
see, for example, \cite[Theorem~26.25]{MSM163}. Using
\eqref{eq:local-z-score-bound}, for every \(p\ge3\),
\[
\begin{aligned}
    \int_{B_g(x,r_0)}
        |z_V^\rho|^p\,d\nu_\rho
    &\le
    C
    \int_{B_g(x,r_0)}
        \bigl(d_g(x,y)+\rho\bigr)^p
        \rho^{-n/2}
        \exp\left(
            -c\frac{d_g(x,y)^2}{\rho}
        \right)dg(y)                     
    \le
    C_p\rho^{p/2}.
\end{aligned}
\]
On \(M\setminus B_g(x,r_0)\), \eqref{eq:off-diagonal-z-score-bound} and the Gaussian upper bound \eqref{eq:gaussian-upper-bound} give
\[
    \int_{M\setminus B_g(x,r_0)}
        |z_V^\rho|^p\,d\nu_\rho
    \le
    C\exp\left(-\frac{c}{\rho}\right)
    =
    O(\rho^{p/2}).
\]
This proves \eqref{eq:small-score-lp}.

It remains to identify the leading second moment. Let
\[
    \gamma_n(\xi):=(4\pi)^{-n/2}e^{-|\xi|^2/4}.
\]
Fix \(R<\infty\), and use the rescaled \(g\)-normal coordinates
\[
    y=\exp_x^g(\sqrt\rho\,\xi),
    \qquad |\xi|\le R .
\]
By the local parametrix \eqref{eq:local-heat-kernel-parametrix}, after pulling back \(d\nu_\rho\) by
\(\xi\mapsto\exp_x^g(\sqrt\rho\,\xi)\), we have uniformly for \(|\xi|\le R\),
\begin{equation}
\label{eq:local-heat-measure-asymptotic}
    \bigl(\exp_x^g(\sqrt\rho\,\cdot)\bigr)^*d\nu_\rho
    =
    \gamma_n(\xi)(1+o_R(1))\,d\xi .
\end{equation}
Moreover, \eqref{eq:local-z-score-unscaled} gives
\begin{equation}
\label{eq:local-z-score-asymptotic}
    z_V^\rho(\exp_x^g(\sqrt\rho\,\xi))
    =
    \sqrt\rho\,\langle V,\xi\rangle_g
    +
    o_R(\sqrt\rho),
\end{equation}
again uniformly for \(|\xi|\le R\). Therefore
\[
    \lim_{\rho\downarrow0}
    \frac1\rho
    \int_{B_g(x,R\sqrt\rho)}
        z_V^\rho z_W^\rho\,d\nu_\rho
    =
    \int_{|\xi|\le R}
        \langle V,\xi\rangle_g
        \langle W,\xi\rangle_g
        \gamma_n(\xi)\,d\xi .
\]

We next estimate the part outside \(B_g(x,R\sqrt\rho)\). On
\(B_g(x,r_0)\setminus B_g(x,R\sqrt\rho)\), the local bound
\eqref{eq:local-z-score-bound} and the Gaussian upper bound \eqref{eq:gaussian-upper-bound} imply
\[
    \limsup_{\rho\downarrow0}
    \frac1\rho
    \left|
        \int_{B_g(x,r_0)\setminus B_g(x,R\sqrt\rho)}
            z_V^\rho z_W^\rho\,d\nu_\rho
    \right|
    \le
    C_{V,W}e^{-cR^2}.
\]
On \(M\setminus B_g(x,r_0)\), the bound
\eqref{eq:off-diagonal-z-score-bound} and the Gaussian upper bound \eqref{eq:gaussian-upper-bound} give an
exponentially small contribution. Hence
\[
    \limsup_{\rho\downarrow0}
    \left|
        \frac1\rho
        \int_M z_V^\rho z_W^\rho\,d\nu_\rho
        -
        \int_{|\xi|\le R}
            \langle V,\xi\rangle_g
            \langle W,\xi\rangle_g
            \gamma_n(\xi)\,d\xi
    \right|
    \le
    C_{V,W}e^{-cR^2}.
\]
Letting \(R\to\infty\) and using
\[
    \int_{\mathbb R^n}
        \langle V,\xi\rangle_g
        \langle W,\xi\rangle_g
        \gamma_n(\xi)\,d\xi
    =
    2\langle V,W\rangle_g,
\]
we obtain
\[
    \int_M z_V^\rho z_W^\rho\,d\nu_\rho
    =
    2\rho\,\langle V,W\rangle_g+o(\rho).
\]
Since \(g=g_t\), this proves \eqref{eq:small-score-second}.
\end{proof}

\begin{proof}[Proof of Lemma~\ref{lem:FisherMetricAsTauToZero}]
Fix \((x,t)\) and \(V,W\in T_xM\). By the source-score covariance identity
\eqref{eq:score-coordinate-cov-summary},
\[
    g_\tau^F(V,W)
    =
    \frac1{2\tau}
    \int_M z_V^\tau z_W^\tau\,d\nu_\tau .
\]
Applying Lemma~\ref{lem:short-time-source-score-moments} with \(\rho=\tau\), we get
\[
    \int_M z_V^\tau z_W^\tau\,d\nu_\tau
    =
    2\tau\,\langle V,W\rangle_{g_t}+o(\tau).
\]
Therefore
\[
    g_\tau^F(V,W)
    \longrightarrow
    \langle V,W\rangle_{g_t}
    \qquad
    \text{as }\tau\downarrow0.
\]
Since \(V,W\in T_xM\) were arbitrary,
\[
    g_\tau^F(x,t)\longrightarrow g_t(x)
\]
as symmetric \(2\)-tensors.
\end{proof}

\subsection{Proof of scale equation for the Fisher metric}
\label{subsec:ScaleEqFisherMetric}

We begin with a useful formula for solutions to conjugate heat equation, where
\[
    \Box^*:=-\partial_s-\Delta_{g(s)}+R
\]
is the conjugate heat operator.

\begin{lemma}\label{lem:quotient_identity_conjugate_heat}
    Let \((M,g(s))\) be a smooth time-dependent Riemannian manifold. Suppose \(u>0\), \(p\), and \(q\) satisfy \(\Box^*u=0,\) \(\Box^*p=0,\) \(\Box^*q=0.\) Then
    \begin{equation}\label{eq:general-quotient-identity}
        \Box^*\left(\frac{pq}{u}\right)
        =
        -2u
        \left\langle
        \nabla\left(\frac{p}{u}\right),
        \nabla\left(\frac{q}{u}\right)
        \right\rangle_{g(s)}. 
    \end{equation}
\end{lemma}

\begin{proof}
    Write
    \[
        f:=\frac pu,
        \qquad
        h:=\frac qu. 
    \]
    For any smooth function \(\phi\), the product rule gives
    \begin{align}\label{eq:box*uphi}
        \Box^*(u\phi)
        =
        u\left(
        -\partial_s\phi-\Delta\phi
        -2\langle\nabla\log u,\nabla\phi\rangle
        \right),
    \end{align}
    because \(\Box^*u=0\). Since \(p=uf\) and \(q=uh\) are conjugate heat solutions, \eqref{eq:box*uphi} implies
    \begin{align}\label{eq:vanishingof-f-h}
        -\partial_sf-\Delta f
        -2\langle\nabla\log u,\nabla f\rangle=0 \qquad  \text{and} \qquad -\partial_sh-\Delta h
        -2\langle\nabla\log u,\nabla h\rangle=0.
    \end{align}
    Applying \eqref{eq:box*uphi} again with \(\phi=fh\) and using \eqref{eq:vanishingof-f-h}, we obtain
    \begin{align*}
        \Box^*(ufh)&=
        u h
        \left(
        -\partial_s f-\Delta f
        -2\langle\nabla\log u,\nabla f\rangle
        \right) +
        u f
        \left(
        -\partial_s h-\Delta h
        -2\langle\nabla\log u,\nabla h\rangle
        \right)
        -2u\langle\nabla f,\nabla h\rangle           \\
        &=
        -2u\langle\nabla f,\nabla h\rangle. 
    \end{align*}
    Since \(ufh=pq/u\), this proves the claim.
\end{proof}

\begin{proof}[Proof of Proposition~\ref{prop:ScaleEqFisherMetric}]
    Fix \(V,W\in T_xM\), and write
    \[
        K=K(x,t;y,s),
        \qquad
        K_V=\nabla_V^xK(x,t;y,s),
        \qquad
        K_W=\nabla_W^xK(x,t;y,s).
    \]
    Then
    \[
        g^F_\tau(V,W)
        =
        2\tau
        \int_M
        \frac{K_VK_W}{K}\,dg_s,
        \qquad
        \tau=t-s.
    \]
    Since \(\Box_y^*\) acts only in the \((y,s)\)-variables, \(\Box_y^*K=\Box_y^*K_V=\Box_y^*K_W=0.\) Applying Lemma~\ref{lem:quotient_identity_conjugate_heat} with \(u=K\), \(p=K_V\), and \(q=K_W\), and using \(z_V^\tau=2\tau\,\frac{K_V}{K},\) \(z_W^\tau=2\tau\,\frac{K_W}{K},\) gives
    \begin{equation}\label{eq:quotient-score-version}
        \Box_y^*
        \left(
        \frac{K_VK_W}{K}
        \right)
        =
        -\frac{1}{2\tau^2}
        K
        \left\langle
        \nabla_y z_V^\tau,
        \nabla_y z_W^\tau
        \right\rangle_{g_s}.
    \end{equation}
    
    For any smooth \(F\), since \(M\) is closed and \(\partial_s dg_s=-R\,dg_s\),
    \begin{align*}
        \int_M\Box_y^*F\,dg_s
        =
        \int_M
        \bigl(-\partial_sF-\Delta F+RF\bigr)\,dg_s =
        -\frac{d}{ds}
        \int_MF\,dg_s. 
    \end{align*}
    As \(\partial_\tau=-\partial_s\), this implies
    \[
        \frac{d}{d\tau}
        \int_M
        \frac{K_VK_W}{K}\,dg_s
        =
        \int_M
        \Box_y^*
        \left(
        \frac{K_VK_W}{K}
        \right)
        dg_s. 
    \]
    Therefore, by \eqref{eq:quotient-score-version},
    \begin{align*}
        \partial_\tau g^F_\tau(V,W)
        &=
        \frac1{\tau}g^F_\tau(V,W)
        +
        2\tau
        \frac{d}{d\tau}
        \int_M
        \frac{K_VK_W}{K}\,dg_s                       =
        \frac1{\tau}g^F_\tau(V,W)
        -
        \frac1{\tau}
        \int_M
        \left\langle
        \nabla_y z_V^\tau,
        \nabla_y z_W^\tau
        \right\rangle_{g_s}
        d\nu_\tau                                      \\
        &=
        \frac1{\tau}
        \bigl(
        g^F_\tau(V,W)-G_\tau(V,W)
        \bigr).
    \end{align*}
    Multiplying by \(\tau\) proves \(\tau\,\partial_\tau g^F_\tau=g^F_\tau-G_\tau.\)
\end{proof}

\subsection{Proof of Fisher metric monotonicity}
\label{subsec:PfFisherMetricMonot}

\begin{proof}[Proof of Proposition~\ref{prop:FisherMetMonot}]
    For \(V\in T_xM\), recall
    \[
        z_V^\tau(y)
        =
        2\tau\,\nabla_V^x\log K(x,t;y,s),
        \qquad
        \int_M z_V^\tau\,d\nu_\tau=0. 
    \]
    The Hein--Naber Poincar\'e inequality \cite[Theorem~1.10]{HeinNaber2014CPAM} gives
    \[
        \int_M (z_V^\tau)^2\,d\nu_\tau
        \le
        2\tau
        \int_M |\nabla_y z_V^\tau|_{g_s}^2\,d\nu_\tau. 
    \]
    By the definitions of \(g^F_\tau\) and \(G_\tau\),
    \[
        \int_M (z_V^\tau)^2\,d\nu_\tau
        =
        2\tau\,g^F_\tau(V,V),
        \qquad
        \int_M |\nabla_y z_V^\tau|_{g_s}^2\,d\nu_\tau
        =
        G_\tau(V,V).
    \]
    Hence \(g^F_\tau(V,V)\le G_\tau(V,V)\) for all \(V\in T_xM,\) so \(g^F_\tau\le G_\tau\) as quadratic forms. The scale equation \(\tau\,\partial_\tau g^F_\tau = g^F_\tau-G_\tau\) therefore implies \(\partial_\tau g^F_\tau\le0.\)
\end{proof}

\subsection{Proof of hessian of source score and source shrinker identity}
\label{subsec:target-hessian-source-score}

\begin{proof}[Proof of Lemma \ref{lem:target-hessian-source-score}]
    Since the normalizing factor in \(K_\tau=(4\pi\tau)^{-n/2}e^{-f_\tau}\) is independent of the source point \(x\), we have
    \[
        z_V^\tau
        =
        2\tau\nabla_V^x\log K_\tau
        =
        -2\tau\nabla_V^x f_\tau.
    \]
    Hence
    \[
        \nabla_y^2 z_V^\tau
        =
        -2\tau\,\nabla_y^2(\nabla_V^x f_\tau)
        =
        -2\tau\,\nabla_V^x(\nabla_y^2 f_\tau),
    \]
    where source differentiation commutes with the target Hessian because the
    target metric \(g_{t-\tau}\) is independent of the source point \(x\).
    Since the only source-dependent term in
    \[
        S_\tau
        =
        \Ric_{g_{t-\tau}}
        +
        \nabla_y^2f_\tau
        -
        \frac1{2\tau}g_{t-\tau}
    \]
    is \(\nabla_y^2f_\tau\), we have
    \[
        \nabla_V^xS_\tau
        =
        \nabla_V^x(\nabla_y^2f_\tau).
    \]
    This proves \eqref{eq:hessian-score-source-shrinker}.

    Moreover, \(z_V^\tau=-2\tau\nabla_V^x f_\tau\), while
    \(g_{t-\tau}\), \(R_{g_{t-\tau}}\), and \(\tau\) are independent of the
    source point \(x\). Therefore
    \begin{align*}
        \nabla_V^x w_\tau
        &=
        2\tau\,
        \Delta_{f_\tau}\bigl(\nabla_V^x f_\tau\bigr)
        +
        \nabla_V^x f_\tau                   =
        -
        \left(
            \Delta_{f_\tau}z_V^\tau
            +
            \frac1{2\tau}z_V^\tau
        \right),
    \end{align*}
    which is \eqref{eq:drift-laplace-score-source-perelman-harnack}.
\end{proof}

\subsection{Proof of the second scale variation of the Fisher metric}
\label{subsec:second-scale-variation-of-the-fisher-metric}

Throughout this proof the source point \((x,t)\) is fixed and \(s=t-\tau\).
We keep the shorthand \(K_\tau,\nu_\tau,f_\tau,z_V^\tau\), but all target
differential operators and inner products are taken with respect to
\(g_{t-\tau}\).

We first record the following consequence of the heat-kernel gauge. If \(q_\tau\)
is any time-dependent function, then
\begin{align}\label{eq:time-derivative-in-moving-gauge}
    \frac{d}{d\tau}\int_M q_\tau\,d\nu_\tau
    =
    \int_M
    \bigl(\partial_\tau+\mathcal L_{\nabla f_\tau}\bigr)q_\tau
    \,d\nu_\tau .
\end{align}
Indeed, using \eqref{eq: target time Lie heat ker meas}, namely \(\bigl(\partial_\tau+\mathcal L_{\nabla f_\tau}\bigr)\nu_\tau=0,\) we obtain
\begin{align*}
    \frac{d}{d\tau}\int_M q_\tau\,d\nu_\tau
    =
    \int_M \partial_\tau q_\tau\,d\nu_\tau
    +
    \int_M q_\tau\,\partial_\tau d\nu_\tau  =
    \int_M \partial_\tau q_\tau\,d\nu_\tau
    -
    \int_M q_\tau\,\mathcal L_{\nabla f_\tau}d\nu_\tau     
    =
    \int_M
    \bigl(\partial_\tau q_\tau+\mathcal L_{\nabla f_\tau}q_\tau\bigr)
    \,d\nu_\tau .
\end{align*}
In the last step we used \(\int_M \mathcal L_{\nabla f_\tau}(q_\tau\,d\nu_\tau)=0.\)

Write
\[
    \Delta_{f_\tau}u
    :=
    \Delta_{g_{t-\tau}}u
    -
    \left\langle
        \nabla_y f_\tau,\nabla_yu
    \right\rangle_{g_{t-\tau}}
\]
for the weighted target Laplacian. The weighted Bochner formula gives
\begin{align*}
    \Delta_{f_\tau}
    \left\langle\nabla_yu,\nabla_yv\right\rangle_{g_{t-\tau}}
    &=
    \left\langle
        \nabla_y\Delta_{f_\tau}u,\nabla_yv
    \right\rangle_{g_{t-\tau}}
    +
    \left\langle
        \nabla_yu,\nabla_y\Delta_{f_\tau}v
    \right\rangle_{g_{t-\tau}}              +
    2\left\langle
        \nabla_y^2u,\nabla_y^2v
    \right\rangle_{g_{t-\tau}}                                  \\
    &\quad+
    2
    \bigl(
        \Ric_{g_{t-\tau}}+\nabla_y^2f_\tau
    \bigr)
    (\nabla_yu,\nabla_yv).
\end{align*}
Combining this with weighted integration by parts gives
\begin{align}\label{eq:integrated-weighted-Bochner}
    \int_M
    (\Delta_{f_\tau}u)(\Delta_{f_\tau}v)\,d\nu_\tau
    =
    \int_M
    \left\langle
        \nabla_y^2u,\nabla_y^2v
    \right\rangle_{g_{t-\tau}}
    \,d\nu_\tau                             +
    \int_M
    \bigl(
        \Ric_{g_{t-\tau}}+\nabla_y^2f_\tau
    \bigr)
    (\nabla_yu,\nabla_yv)\,d\nu_\tau .
\end{align}

\begin{proof}[Proof of Proposition~\ref{prop:second-scale-variation-of-the-fisher-metric}]
    We differentiate the target Dirichlet tensor
    \[
        G_\tau(V,W)
        =
        \int_M
        \left\langle
            \nabla_y z_V^\tau,\nabla_y z_W^\tau
        \right\rangle_{g_{t-\tau}}
        \,d\nu_\tau .
    \]
    Using \eqref{eq:time-derivative-in-moving-gauge},
    \eqref{eq:freezing-of-conjugate-heat-kernel}, and the source-score evolution
    \[
        \bigl(\partial_\tau+\mathcal L_{\nabla f_\tau}\bigr)z_V^\tau
        =
        \Delta_{f_\tau}z_V^\tau+\frac1\tau z_V^\tau,
    \]
    we compute
    \begin{align*}
        \frac{d}{d\tau} G_\tau(V,W)
        &=
        \int_M
        \bigl(\partial_\tau+\mathcal L_{\nabla f_\tau}\bigr)
        \left\langle
            \nabla_y z_V^\tau,
            \nabla_y z_W^\tau
        \right\rangle_{g_{t-\tau}}
        \,d\nu_\tau                                             \\
        &=
        \frac{2}{\tau} G_\tau(V,W)
        -
        2
        \int_M
        \bigl(\Ric_{g_{t-\tau}}+\nabla_y^2f_\tau\bigr)
        (\nabla_y z_V^\tau,\nabla_y z_W^\tau)
        \,d\nu_\tau                                             \\
        &\quad+
        \int_M
        \left\langle
            \nabla_y\Delta_{f_\tau}z_V^\tau,
            \nabla_yz_W^\tau
        \right\rangle_{g_{t-\tau}}
        \,d\nu_\tau                         +
        \int_M
        \left\langle
            \nabla_yz_V^\tau,
            \nabla_y\Delta_{f_\tau}z_W^\tau
        \right\rangle_{g_{t-\tau}}
        \,d\nu_\tau                                             \\
        &=
        \frac{2}{\tau} G_\tau(V,W)
        -
        2
        \int_M
        \bigl(\Ric_{g_{t-\tau}}+\nabla_y^2f_\tau\bigr)
        (\nabla_y z_V^\tau,\nabla_y z_W^\tau)
        \,d\nu_\tau                         -
        2
        \int_M
        (\Delta_{f_\tau} z_V^\tau)
        (\Delta_{f_\tau} z_W^\tau)
        \,d\nu_\tau                                             \\
        &=
        -2
        \int_M
        \left\langle
            \nabla_y^2 z_V^\tau,
            \nabla_y^2 z_W^\tau
        \right\rangle_{g_{t-\tau}}
        \,d\nu_\tau                         -
        4
        \int_M
        S_\tau
        (\nabla_y z_V^\tau,\nabla_y z_W^\tau)
        \,d\nu_\tau .
    \end{align*}
    In the third equality we integrated by parts. In the fourth equality we used
    \eqref{eq:integrated-weighted-Bochner} and \(\Ric_{g_{t-\tau}}+\nabla_y^2f_\tau = S_\tau+\frac1{2\tau}g_{t-\tau}.\) The contribution of \(\frac1{2\tau}g_{t-\tau}\) cancels the term
    \(\frac2\tau G_\tau(V,W)\). We now apply this formula to the derivative of the scale identity \(\tau\partial_\tau g_\tau^F=g_\tau^F-G_\tau\) to obtain
    \[
        \frac{d^2}{d\tau^2}g_\tau^F(V,W)
        =
        -\frac1\tau\frac{d}{d\tau}G_\tau(V,W)=\frac{2}{\tau}
        \int_M
        \left\langle
            \nabla_y^2 z_V^\tau,
            \nabla_y^2 z_W^\tau
        \right\rangle_{g_{t-\tau}}
        \,d\nu_\tau              +
        \frac{4}{\tau}
        \int_M
        S_\tau
        \bigl(
            \nabla_y z_V^\tau,
            \nabla_y z_W^\tau
        \bigr)
        \,d\nu_\tau.
    \]
\end{proof}

\subsection{Proof of mean source hessian identity}\label{subsec:mean-Hessian-identity}

\begin{proof}[Proof of Lemma \ref{lem:mean-Hessian-identity}]
    The heat-kernel normalization gives
    \[
        \int_M W^x\log K_\tau\,d\nu_\tau=0.
    \]
    Differentiate this identity in the source direction $V$. Since
    \[
        V(d\nu_\tau)=V^x\log K_\tau\,d\nu_\tau,
    \]
    we obtain
    \[
        0=\int_M V(W\log K_\tau)\,d\nu_\tau
        +\int_M V\log K_\tau\,W\log K_\tau\,d\nu_\tau.
    \]
    Writing
    \[
        V(W\log K_\tau)
        =\nabla^x_{V,W}\log K_\tau+(\nabla_VW)\log K_\tau
    \]
    and using again the first derivative normalization for $\nabla_VW$, we get
    \[
        \int_M \nabla^x_{V,W}\log K_\tau\,d\nu_\tau
        =-\int_M V\log K_\tau\,W\log K_\tau\,d\nu_\tau
        =-\frac1{2\tau}g^F_\tau(V,W).
    \]
    Adding $(2\tau)^{-1}g_t(V,W)$ gives \eqref{eq:mean-Hessian-identity}.
\end{proof}

\subsection{Proof of the matrix square identity refining the pointed Nash identity}
\label{subsec:matrix-square-refines-Nash}

Let \(\Box:=\partial_t-\Delta_{g(t)}\). We use Uhlenbeck's trick (\cite[\S 2]{Hamilton1986}), so the pulled-back metric components  $g_{ab}$ are time-independent $\partial_t g_{ab}=0$ and \(\Box\) commutes with spatial covariant differentiation on scalar functions.

\begin{lemma}\label{lem:BasicBochnerRF}
    Let \(u>0\) solve \(\Box u=0\). Then, in an Uhlenbeck frame, the symmetric $2$-tensor $\frac{\nabla u \otimes \nabla u}{u}$ satisfies
    \[
        \Box
        \left(
        \frac{\nabla_a u\,\nabla_b u}{u}
        \right)
        =
        -2u\,
        \nabla_a\nabla_c\log u\,
        \nabla_b\nabla_c\log u. 
    \]
\end{lemma}

\begin{proof}
    Since \(\Box\nabla_a u=\nabla_a\Box u=0\), the product rule and \(\Box(u^{-1})=-2u^{-3}|\nabla u|^2\) give
    \begin{align*}
        \Box
        \left(
        \frac{\nabla_a u\,\nabla_b u}{u}
        \right)
        &=
        -\frac{2}{u}
        \left(
        \nabla_c\nabla_a u
        -
        \frac{\nabla_a u\,\nabla_c u}{u}
        \right)
        \left(
        \nabla_c\nabla_b u
        -
        \frac{\nabla_b u\,\nabla_c u}{u}
        \right)  \\
        &=
        -2u\,
        \nabla_c\nabla_a\log u\,
        \nabla_c\nabla_b\log u. 
    \end{align*}
    This is the claimed identity.
\end{proof}

\begin{proof}[Proof of Proposition~\ref{prop:matrix-square-refines-Nash}]
    Fix \((y,s)\), write
    \[
        K=K(x,t;y,s),
        \qquad
        \tau=t-s,
    \]
    and take all derivatives in the source variables \((x,t)\). By Lemma~\ref{lem:BasicBochnerRF},
    \[
        \Box_{x,t}
        \left(
        \frac{\nabla_aK\,\nabla_bK}{K}
        \right)
        =
        -2K\,
        \nabla_a\nabla_c\log K\,
        \nabla_b\nabla_c\log K. 
    \]
    After integrating in \(y\), this becomes
    \begin{equation}\label{eq:Box-A-Nash-section}
        \Box_{x,t}\int_M \frac{\nabla_a K \nabla_b K}{K} dg_s 
        =  
        -2
        \int_M
        \nabla_c\nabla_a\log K\,
        \nabla_c\nabla_b\log K
        \,d\nu_{x,t;s}.
    \end{equation}
    Also, differentiating \(\int_M K(x,t;y,s)\,dg_s(y)=1\) twice in \(x\) gives
    \begin{equation}\label{eq:A-average-Hess-Nash-section}
        \int_M  \frac{\nabla_a K \nabla_b K}{K} dg_s 
        =
        -\int_M\nabla_a\nabla_b\log K\,d\nu_{x,t;s}.
    \end{equation}
    Since \(g_{ab}(t)\) is constant in an Uhlenbeck frame and \(\Box\tau=1\), and by substituting \eqref{eq:Box-A-Nash-section} and \eqref{eq:A-average-Hess-Nash-section}, and completing the square, we have 
    \begin{align*}
        \Box_{x,t} \left(  \frac{\tau}{2}D^F_{\tau,ab} \right)  
        & = \Box_{x,t} \left(  \frac{\tau}{2} ( g_t-g^F_\tau )_{ab}  \right)  \\
        & = \frac{1}{2}g_{t,ab} - \Box_{x,t} \left( \tau^2  \int_M 
        \frac{\nabla_a K \nabla_b K}{K} dg_s  \right) \\
        & = \frac{1}{2}g_{t,ab} + 2\tau \int_M\nabla_a\nabla_b\log K\,d\nu_{x,t;s} 
        + 2\tau^2  \int_M
        \nabla_c\nabla_a\log K\,
        \nabla_c\nabla_b\log K
        \,d\nu_{x,t;s} \\
        & = 2\tau^2
        \int_M
        \left(
        \nabla_c\nabla_a\log K+\frac1{2\tau}g_{ca}
        \right)
        \left(
        \nabla_c\nabla_b\log K+\frac1{2\tau}g_{cb}
        \right)
        d\nu_{x,t;s}, 
    \end{align*}
    which proves the first identity in \eqref{eq:deficit-square-refines-Nash}.

    By Lemma~\ref{lem:mean-Hessian-identity},
    \[
        \int_M \left(\mathsf H^{(x)}-\frac1{2\tau}D^F_\tau\right)\,d\nu_\tau=0.
    \]
    Hence the cross terms vanish, and
    \begin{align*}
        2\tau^2 \int_M
        \mathsf H^{(x)}\circ\mathsf H^{(x)}
        \,d\nu_\tau
        &=
        \frac1{2}(D^F_\tau)^2
        +
        2\tau^2\int_M
        \left(\mathsf H^{(x)}-\frac1{2\tau}D^F_\tau\right)
        \circ\left(\mathsf H^{(x)}-\frac1{2\tau}D^F_\tau\right)
        \,d\nu_\tau
    \end{align*}
    Substituting this into the first identity in \eqref{eq:deficit-square-refines-Nash} proves the second identity.
\end{proof}

\subsection{Proof of scalar square identity for pointed Nash entropy}
\label{subsec:scalar-Nash-square-identity}

\begin{proof}[Proof of Corollary~\ref{cor:scalar-Nash-square-identity}]
    By the definition of \(g^F_\tau\),
    \[
        \tr_{g_t}D^F_\tau
        =
        n
        -
        2\tau
        \int_M|\nabla_x\log K|^2\,d\nu_{x,t;s}.
    \]
    Thus, using \eqref{eq:E-Fisher-information}, \(\frac{\tau}{2}\tr_{g_t}D^F_\tau = -\tau^2E.\) Therefore, taking the trace of the first identity in \eqref{eq:deficit-square-refines-Nash} gives
    \[
        \Box_{x,t}
        \left(
        \frac{\tau}{2}\tr_{g_t}D^F_\tau
        \right)=-\Box_{x,t}(\tau^2E)
        =
        2\tau^2
        \int_M
        \left|
        \nabla_x^2\log K
        +
        \frac1{2\tau}g_t
        \right|^2
        d\nu_{x,t;s},
    \]
    which proves the first identity in \eqref{eq:tau2E-square}. Taking the trace of second identity in \eqref{eq:deficit-square-refines-Nash} and using
    \[
        \operatorname{tr}(D^F_\tau)^2
        \ge
        \frac1n(\operatorname{tr}D^F_\tau)^2
    \]
    gives
    \[
        \Box_{x,t}
        \left(
        \tau\operatorname{tr}D^F_\tau
        \right)
        \ge
        \frac1n
        \left(
        \operatorname{tr}D^F_\tau
        \right)^2,
    \]
    which proves \eqref{eq:tau2E-square}.
    
    Since \(\Box_{x,t}\tau=1\) and \(\nabla_x\tau=0\),
    \[
        \Box_{x,t}(\tau^2E)
        =
        2\tau E+\tau^2\Box_{x,t}E. 
    \]
    Expanding the square in \eqref{eq:tau2E-square} and using
    \[
        \int_M\Delta_x\log K\,d\nu_{x,t;s}
        =
        -\int_M|\nabla_x\log K|^2\,d\nu_{x,t;s},
    \]
    which follows by differentiating \(\int_M K(x,t;y,s)\,dg_s(y)=1\) twice in \(x\) and tracing, gives
    \[
        \Box_{x,t}E
        =
        -2\int_M|\nabla_x^2\log K|^2\,d\nu_{x,t;s}
        +
        \frac n{2\tau^2}.
    \]
    This proves \eqref{eq:BoxE-raw-Nash-section}.
\end{proof}

\subsection{Fisher equality rigidity}
\label{subsec:PfFisherEqualityRigidity}

\begin{proof}[Proof of Theorem~\ref{thm:fisher-equality-rigidity}]
    We work in an Uhlenbeck frame in the source \((x,t)\)-variables, so the metric components are time-independent. The Fisher defect identity \eqref{eq:deficit-square-refines-Nash} gives
    \[
        \Box_{x,t}
        \left(
        \frac{\tau}{2}D^F_{\tau,ab}
        \right)
        =
        2\tau^2
        \int_M
        \mathsf H^{(x)}_{ca}\mathsf H^{(x)}_{cb}
        \,d\nu_{x,t;s}.
    \]
    The right-hand side is nonnegative as a symmetric \(2\)-tensor, since its quadratic form in a vector \(V\) is
    \[
        2\tau^2
        \int_M
        |\mathsf H^{(x)}(\cdot,V)|^2
        \,d\nu_{x,t;s}.
    \]
    Thus the tensor equation for \((\tau/2)D^F_\tau\) satisfies the null-eigenvector condition for the cone of nonnegative symmetric tensors. Hamilton's strong maximum principle for tensor systems \cite[\S 8]{Hamilton1986}, applied locally and then propagated along paths using connectedness of \(M\), implies that the nullspaces
    \[
        \mathcal K_x:=\ker D^F_{\tau_0}(x,t_0;s)
    \]
    form a nontrivial parallel distribution on \((M,g(t_0))\).
    
    The same maximum-principle argument gives vanishing of the square term on null vectors. Hence, if \(V\in\mathcal K_x\), then
    \[
        0=
        2\tau_0^2
        \int_M
        |\mathsf H^{(x)}(\cdot,V)|^2
        \,d\nu_{x,t_0;s}.
    \]
    Since \(K(x,t_0;\cdot,s)>0\), the measure \(d\nu_{x,t_0;s}\) has positive smooth density. Therefore
    \[
        \mathsf H^{(x)}(x,t_0;y,s)(\cdot,V)=0
        \qquad
        \text{for every }y\in M,
    \]
    which proves \eqref{eq:null-H-vanishing}.
    
    We next prove the curvature-vanishing statement. Fix \(y\in M\), let \(P:TM\to\mathcal K\) be the \(g(t_0)\)-orthogonal projection onto the parallel distribution \(\mathcal K\), and set
    \[
        X_y:=P\nabla_x\log K(x,t_0;y,s).
    \]
    Since \(PX_y=X_y\) and \(\nabla P=0\), differentiating gives
    \[
        P(\nabla_WX_y)=\nabla_WX_y.
    \]
    Thus \(\nabla_WX_y\in\mathcal K\). For \(V\in\mathcal K\), using \eqref{eq:null-H-vanishing}, we have
    \begin{align*}
        g_{t_0}(\nabla_WX_y,V)
        =
        g_{t_0}(P\nabla_W\nabla_x\log K,V)
        =
        \nabla_x^2\log K(W,V)  =
        -\frac1{2\tau_0}g_{t_0}(W,V)
        =
        -\frac1{2\tau_0}g_{t_0}(PW,V).
    \end{align*}
    Both \(\nabla_WX_y\) and \(-\frac1{2\tau_0}PW\) lie in \(\mathcal K\), so the last identity for all \(V\in\mathcal K\) implies
    \[
        \nabla_WX_y=-\frac1{2\tau_0}PW. 
    \]
    Extending \(W,Z\) locally, this identity, together with \(\nabla P=0\), gives
    \begin{align*}
        \Rm(W,Z)X_y
        =
        \nabla_W\nabla_ZX_y
        -
        \nabla_Z\nabla_WX_y
        -
        \nabla_{[W,Z]}X_y      =
        -\frac1{2\tau_0}
        P\bigl(\nabla_WZ-\nabla_ZW-[W,Z]\bigr)
        =
        0,
    \end{align*}
    because the Levi-Civita connection is torsion-free. It remains to show that the vectors \(X_y(x)\), as \(y\) varies, span \(\mathcal K_x\). If \(U\in\mathcal K_x\) is orthogonal to all \(X_y(x)\), then
    \[
        0
        =
        g_{t_0}(U,X_y(x))
        =
        g_{t_0}(U,P\nabla_x\log K)
        =
        g_{t_0}(U,\nabla_x\log K)
        =
        \nabla_U^x\log K(x,t_0;y,s)
    \]
    for every \(y\in M\). Hence
    \[
        g^F_{\tau_0}(U,U)
        =
        2\tau_0
        \int_M
        \bigl(\nabla_U^x\log K\bigr)^2
        \,d\nu_{x,t_0;s}
        =
        0.
    \]
    Since \(U\in\ker D^F_{\tau_0}\),
    \[
        0=D^F_{\tau_0}(U,U)
        =
        |U|_{g(t_0)}^2-g^F_{\tau_0}(U,U),
    \]
    so \(U=0\). Thus the vectors \(X_y(x)\) span \(\mathcal K_x\), and therefore
    \[
        \Rm_{g(t_0)}(W,Z)V=0
    \]
    for all \(V\in\mathcal K_x\) and all \(W,Z\in T_xM\). This proves \eqref{eq:curvature-annihilates-K}.
    
    Assume now that \((M,g(t_0))\) is complete. On the universal cover, the lifted distribution \(\widetilde{\mathcal K}\) is parallel. The de Rham splitting theorem gives a product factor tangent to \(\widetilde{\mathcal K}\), and \eqref{eq:curvature-annihilates-K} shows that this factor is flat. Hence
    \[
        (\widetilde M,\widetilde g(t_0))
        \cong
        (\mathbb R^k,g_{\mathrm{Euc}})\times (N,h(t_0)),
        \qquad
        k=\dim\mathcal K,
    \]
    with \(\widetilde{\mathcal K}\) tangent to the Euclidean factor.
    
    Finally, the spacetime version of Hamilton's strong maximum principle gives invariance of the null distribution on \((s,t_0]\). Repeating the preceding argument at each time gives a parallel flat distribution, and
    \[
        \Rm(W,Z)V=0
        \qquad
        \text{for }V\text{ in the null distribution}.
    \]
    Moreover, \(\Rm(W,Z)V=0\) for \(V\) in the null distribution implies \(\Ric_{g(t)}(V,\cdot)=0\). Indeed, for any \(U\in T_xM\) and any \(g(t)\)-orthonormal basis \(e_1,\ldots,e_n\), the first Bianchi identity gives
    \begin{align*}
        \Ric_{g(t)}(V,U)
        =
        \sum_i \Rm(e_i,V,U,e_i)       =
        -\sum_i \Rm(U,e_i,V,e_i)
        =
        0,
    \end{align*}
    where the omitted Bianchi term is \(\Rm(V,U,e_i,e_i)=0\), and the last equality uses \(\Rm(U,e_i)V=0\). Thus
    \[
        \partial_tg(t)(V,\cdot)
        =
        -2\Ric_{g(t)}(V,\cdot)
        =
        0.
    \]
    If \(g(t)\) is complete for all \(s<t\le t_0\), the lifted flow splits consistently as
    \[
        (\widetilde M,\widetilde g(t))
        \cong
        (\mathbb R^k,g_{\mathrm{Euc}})\times (N,h(t)),
        \qquad
        s<t\le t_0,
    \]
    where the Euclidean factor is static and \(h(t)\) evolves by Ricci flow.
\end{proof}

\subsection{Proof of strict Fisher deficit on closed flows}
\label{subsec:PfStrictFisherDeficitClosedFlows}

\begin{proof}[Proof of Corollary~\ref{cor:strict-closed-fisher}]
    Set \(\tau:=t-s>0\). We first prove \(g^F_\tau(x,t)>0\). Suppose that \(g^F_\tau(V,V)=0\) for some nonzero \(V\in T_xM\). Then
    \[
        0
        =
        2\tau
        \int_M
        \bigl(\nabla_V^x\log K(x,t;y,s)\bigr)^2
        d\nu_{x,t;s}(y).
    \]
    Since the heat kernel is positive, this implies
    \[
        \nabla_V^x K(x,t;y,s)=0
        \qquad
        \text{for all }y\in M. 
    \]
    
    Fix \(\ell\in(s,t)\). Differentiating the reproduction formula
    \[
        K(x,t;y,s)
        =
        \int_M K(x,t;z,\ell)K(z,\ell;y,s)\,dg_\ell(z)
    \]
    in the source direction \(V\), we obtain
    \[
        \int_M
        \nabla_V^xK(x,t;z,\ell)K(z,\ell;y,s)\,dg_\ell(z)
        =
        0
        \qquad
        \text{for all }y\in M. 
    \]
    Equivalently, \(\nabla_V^xK(x,t;\cdot,\ell)\) is orthogonal to the range of the forward heat operator \(P_{s,\ell}\). On a closed smooth Ricci-flow background this range is dense, by backward uniqueness. Hence
    \[
        \nabla_V^xK(x,t;z,\ell)=0
        \qquad
        \text{for all }z\in M,\quad s<\ell<t. 
    \]
    This contradicts the short-time heat-kernel asymptotics; see \cite[Theorem~24.21]{MSM163}: if
    \[
        z_\varepsilon:=\exp_x^{g_t}(\sqrt\varepsilon\,V),
    \]
    then, as \(\varepsilon\downarrow0\),
    \[
        \nabla_V^x\log K(x,t;z_\varepsilon,t-\varepsilon)
        =
        \frac{|V|_{g_t}^2}{2\sqrt\varepsilon}
        +
        O(1),
    \]
    which is nonzero for all sufficiently small \(\varepsilon\). Therefore \(g^F_\tau(x,t)\) is positive definite.
    
    It remains to prove the strict upper bound. Suppose that
    \[
        D^F_\tau(V,V)=0,
        \qquad
        D^F_\tau:=g_t-g^F_\tau,
    \]
    for some nonzero \(V\in T_xM\). Since \(M\) is closed, the matrix gradient estimate gives \(D^F_{\tau'}\ge0\) at every positive scale \(\tau'>0\), so Theorem~\ref{thm:fisher-equality-rigidity} applies. Thus the universal cover of \((M,g_t)\) splits off a Euclidean factor tangent to the Fisher-null distribution, and for every null vector,
    \[
        \mathsf H^{(x)}(\cdot,V)=0.
    \]
    
    Fix \(y\in M\), and lift \(\log K(\cdot,t;y,s)\) to the universal cover. Since \(M\) is closed and the heat kernel is positive and smooth, this lifted function is bounded. Let \(\gamma:\mathbb R\to \widetilde M\) be a unit-speed geodesic tangent to the Euclidean factor. The identity \(\mathsf H^{(x)}(\gamma',\gamma')=0\) gives
    \[
        \frac{d^2}{dr^2}
        \log K(\pi(\gamma(r)),t;y,s)
        =
        -\frac1{2\tau}.
    \]
    Hence
    \[
        \log K(\pi(\gamma(r)),t;y,s)
        =
        -\frac{r^2}{4\tau}+ar+b,
    \]
    which is unbounded below as \(|r|\to\infty\). This contradicts boundedness of the lifted heat kernel logarithm. Therefore \(D^F_\tau(x,t;s)\) is positive definite.
    
    We have shown \(0<g^F_\tau(x,t)<g_t(x).\) Equivalently, the eigenvalues of \(\Jcal_\tau=g_t^{-1}g^F_\tau\) satisfy \(0<\lambda_i(x,t;t-s)<1\) for every \(i=1,\ldots,n.\)
\end{proof}

\subsection{Proof of sharp reverse Poincar\'e inequality}
\label{subsec:SharpReversePoincare}

\begin{proof}[Proof of Corollary~\ref{cor:sharp-reverse-poincare-hellinger}]
    Differentiating \(\int_M K(x,t;y,s)\,dg_s(y)=1\) in the source variable gives
    \[
        \int_M \nabla_x\log K(x,t;y,s)\,d\nu_{x,t;s}(y)=0. 
    \]
    Hence, for \(V\in T_xM\),
    \begin{align*}
        d_x(P_{s,t}\phi)(V)
        &=
        \int_M
        \phi(y)\,\nabla_V^x\log K(x,t;y,s)
        \,d\nu_{x,t;s}(y)                                  \\
        &=
        \int_M
        \bigl(\phi-(P_{s,t}\phi)(x)\bigr)
        \nabla_V^x\log K(x,t;y,s)
        \,d\nu_{x,t;s}(y).
    \end{align*}
    By Cauchy--Schwarz,
    \begin{align*}
        \bigl(d_x(P_{s,t}\phi)(V)\bigr)^2
        &\le
        \operatorname{Var}_{x,t;s}(\phi)
        \int_M
        \bigl(\nabla_V^x\log K\bigr)^2
        \,d\nu_{x,t;s}         =
        \frac{\operatorname{Var}_{x,t;s}(\phi)}{2\tau}
        g^F_\tau(V,V).
    \end{align*}
    This proves the tensor inequality in \eqref{eq:tensor-reverse-poincare}. Using \(g^F_\tau\le g_t\), we obtain
    \[
        \bigl(d_x(P_{s,t}\phi)(V)\bigr)^2
        \le
        \frac{\operatorname{Var}_{x,t;s}(\phi)}{2\tau}
        |V|_{g_t}^2. 
    \]
    Taking the supremum over all \(g_t\)-unit vectors \(V\) gives
    \[
        |\nabla_x(P_{s,t}\phi)|_{g_t}^2
        \le
        \frac{\operatorname{Var}_{x,t;s}(\phi)}{2\tau},
    \]
    which is \eqref{eq:sharp-reverse-poincare}.
\end{proof}

\subsection{Proofs of the \texorpdfstring{$\varphi$}{varphi}-divergence estimates}
\label{subsec:proofs-phi-divergences}

\begin{proof}[Proof of Proposition~\ref{prop:divergence-contraction}]
    Set \(\lambda=\frac{v_0}{v_1}\) and \(\Phi(a,b)=b\varphi(a/b)\). Then
    \[
        \Phi_a=\varphi'(\lambda),
        \qquad
        \Phi_b=\varphi(\lambda)-\lambda\varphi'(\lambda),
        \qquad
        a\Phi_a+b\Phi_b=\Phi. 
    \]
    Since \(\partial_s v_i=-\Delta v_i+Rv_i\) and \(\partial_s dg_s=-R\,dg_s,\) the curvature terms cancel by homogeneity. Hence, after integrating by parts,
    \begin{align*}
        \frac{d}{ds}D_\varphi
        &=
        -\int_M
        (\Phi_a\Delta v_0+\Phi_b\Delta v_1)\,dg_s
        =
        \int_M
        \left(
        \langle\nabla\Phi_a,\nabla v_0\rangle
        +
        \langle\nabla\Phi_b,\nabla v_1\rangle
        \right)dg_s                                      \\
        &=
        \int_M
        \varphi''(\lambda)
        \langle\nabla\lambda,\nabla v_0-\lambda\nabla v_1\rangle
        dg_s
        =
        \int_M v_1\varphi''(\lambda)|\nabla\lambda|^2\,dg_s. 
    \end{align*}
    If \(\varphi''>0\), equality forces \(\nabla\lambda=0\). Since \(M\) is connected and both measures have mass one, \(\lambda\equiv1\).
\end{proof}

\subsection{Proof of infinitesimal \texorpdfstring{$\varphi$}{varphi}-divergence}\label{subsec:infinitesimal-phi-divergence}

\begin{proof}[Proof of Proposition~\ref{prop:infinitesimal-phi-divergence}]
    Let
    \[
        K_r(y)=K(\exp_x(rV),t;y,s),
        \qquad
        K_0(y)=K(x,t;y,s),
        \qquad
        \lambda_r=\frac{K_r}{K_0}.
    \]
    Then
    \[
        \lambda_r
        =
        1+r\nabla_V^x\log K(x,t;\cdot,s)+O(r^2),
        \qquad
        \int_M(\lambda_r-1)\,d\nu_{x,t;s}=0. 
    \]
    Taylor expansion at \(1\) gives
    \[
        D_\varphi
        \left(
        \nu_{\exp_x(rV),t;s}
        \mid
        \nu_{x,t;s}
        \right)
        =
        \frac12\varphi''(1)
        \int_M(\lambda_r-1)^2\,d\nu_{x,t;s}
        +o(r^2).
    \]
    Therefore
    \begin{align*}
        \lim_{r\to0}
        \frac{1}{r^2}
        D_\varphi
        \left(
        \nu_{\exp_x(rV),t;s}
        \mid
        \nu_{x,t;s}
        \right)
        =
        \frac12\varphi''(1)
        \int_M
        \left(\nabla_V^x\log K(x,t;y,s)\right)^2
        d\nu_{x,t;s}(y)          =
        \frac{\varphi''(1)}{4\tau}g_\tau^F(V,V),
    \end{align*}
    which proves \eqref{eq:infinitesimal-divergence-is-fisher}. 
\end{proof}

\subsection{Proof of sharp \texorpdfstring{$\varphi$}{varphi}-Sobolev inequality}\label{subsec:sharp-phi-sobolev-heat-kernel}

\begin{proof}[Proof of Proposition~\ref{prop:sharp-phi-sobolev-heat-kernel}]
    Write \(s=t-\tau\), set \(d\nu_\ell=d\nu_{x,t;\ell}\), and let \(u_\ell=P_{s,\ell}u\). Since \(\int_M u\,d\nu_s=1\), we have \(u_t(x)=1\). Thus
    \[
        D_\varphi(u\nu_s\mid\nu_s)
        =
        \int_s^t
        \int_M
        \varphi''(u_\ell)|\nabla u_\ell|^2
        d\nu_\ell\,d\ell. 
    \]
    Put \(a=\varphi''\). The Ricci-flow Bochner formula gives
    \begin{align*}
        (\partial_\ell-\Delta)
        \left(a(u_\ell)|\nabla u_\ell|^2\right)
        &=
        -2a(u_\ell)|\nabla^2u_\ell|^2
        -4a'(u_\ell)\nabla^2u_\ell(\nabla u_\ell,\nabla u_\ell) 
        -a''(u_\ell)|\nabla u_\ell|^4. 
    \end{align*}
    The condition \((1/a)''\le0\) is equivalent to \(aa''\ge2(a')^2\), and hence the right-hand side is nonpositive. Pairing with the conjugate heat-kernel measure yields
    \[
        \frac{d}{d\ell}
        \int_M
        a(u_\ell)|\nabla u_\ell|^2d\nu_\ell
        \le0. 
    \]
    Therefore
    \[
        D_\varphi(u\nu_s\mid\nu_s)
        \le
        (t-s)
        \int_M
        \varphi''(u)|\nabla u|^2d\nu_s,
    \]
    which is \eqref{eq:sharp-phi-sobolev-heat-kernel}.
\end{proof}

\subsection{Proof of \texorpdfstring{$\varphi$}{varphi}-strong data processing inequality}\label{subsec:scale-sharp-phi-SDPI}

\begin{proof}[Proof of Corollary~\ref{cor:scale-sharp-phi-SDPI}]
    Let
    \[
        w_\tau
        :=
        \frac{d\eta_\tau}{d\nu_{x,t;t-\tau}}.
    \]
    The backward-time form of Proposition~\ref{prop:divergence-contraction} gives
    \[
        \frac{d}{d\tau}
        D_\varphi
        \left(
        \eta_\tau
        \mid
        \nu_{x,t;t-\tau}
        \right)
        =
        -
        \int_M
        \varphi''(w_\tau)|\nabla w_\tau|^2
        d\nu_{x,t;t-\tau}.
    \]
    By Proposition~\ref{prop:sharp-phi-sobolev-heat-kernel},
    \[
        D_\varphi
        \left(
        \eta_\tau
        \mid
        \nu_{x,t;t-\tau}
        \right)
        \le
        \tau
        \int_M
        \varphi''(w_\tau)|\nabla w_\tau|^2
        d\nu_{x,t;t-\tau}.
    \]
    Hence
    \[
        \frac{d}{d\tau}
        \left[
        \tau
        D_\varphi
        \left(
        \eta_\tau
        \mid
        \nu_{x,t;t-\tau}
        \right)
        \right]
        \le0,
    \]
    and the finite-scale estimate follows by integration.
\end{proof}

\subsection{Proof of dual monotonicity}
\label{subsec:dual-nonlinear-fisher-monotonicity}

For \(0<\rho\le\tau\), write
\[
    z_\rho:=z_V^\rho .
\]

We first record the source-score chain rule.

\begin{lemma}[Source-score chain rule]
\label{lem:source-score-chain-rule}
For every smooth function \(\chi:\mathbb R\to\mathbb R\),
\begin{equation}
\label{eq:source-score-chain-rule}
    \frac{d}{d\rho}
    \int_M \chi(z_\rho)\,d\nu_\rho
    =
    \frac1\rho
    \int_M z_\rho\chi'(z_\rho)\,d\nu_\rho
    -
    \int_M
    \chi''(z_\rho)
    |\nabla z_\rho|_{g_{t-\rho}}^2
    \,d\nu_\rho .
\end{equation}
\end{lemma}

\begin{proof}
Let
\[
    K=K(x,t;\cdot,t-\rho),
\]
and let all gradients and Laplacians be taken in the target variable with respect
to \(g_{t-\rho}\). Since \(K\) solves the conjugate heat equation and
\(\partial_\rho dg_{t-\rho}=R\,dg_{t-\rho}\),
\[
    \partial_\rho K=\Delta K-RK,
    \qquad
    \partial_\rho(K\,dg_{t-\rho})=(\Delta K)\,dg_{t-\rho}.
\]
Moreover, \(\nabla_V^xK\) satisfies the same conjugate heat equation in the target
variables. Hence
\[
\begin{aligned}
    \partial_\rho(\nabla_V^x\log K)
    &=
    \frac{\Delta(\nabla_V^xK)-R\nabla_V^xK}{K}
    -
    \frac{\nabla_V^xK}{K^2}(\Delta K-RK)      \\
    &=
    \Delta(\nabla_V^x\log K)
    +
    2\left\langle
        \nabla\log K,
        \nabla(\nabla_V^x\log K)
    \right\rangle .
\end{aligned}
\]
Since \(z_\rho=2\rho\,\nabla_V^x\log K\),
\begin{equation}
\label{eq:source-score-pointwise-evolution}
    \partial_\rho z_\rho
    =
    \frac1\rho z_\rho
    +
    \Delta z_\rho
    +
    2\langle\nabla\log K,\nabla z_\rho\rangle .
\end{equation}

Differentiating the integral and using that \(M\) is closed,
\begin{align*}
    \frac{d}{d\rho}\int_M\chi(z_\rho)\,d\nu_\rho
    &=
    \int_M\chi'(z_\rho)\partial_\rho z_\rho\,d\nu_\rho
    +
    \int_M\chi(z_\rho)\Delta K\,dg_{t-\rho}             \\
    &=
    \int_M
    \left[
        \chi'(z_\rho)\partial_\rho z_\rho
        +
        \Delta\chi(z_\rho)
    \right]d\nu_\rho\\
    &=\frac1\rho
    \int_Mz_\rho\chi'(z_\rho)\,d\nu_\rho                 +
    2\int_M
    \chi'(z_\rho)
    \left(
        \Delta z_\rho
        +
        \langle\nabla\log K,\nabla z_\rho\rangle
    \right)d\nu_\rho                                    \\
    &\quad+
    \int_M
    \chi''(z_\rho)|\nabla z_\rho|^2
    \,d\nu_\rho\\
    &=\frac1\rho
    \int_Mz_\rho\chi'(z_\rho)\,d\nu_\rho                 +
    2\int_M
    \chi'(z_\rho)
    \left(
        \operatorname{div}(K\nabla z_\rho)
    \right)dg_{t-\rho} +
    \int_M
    \chi''(z_\rho)|\nabla z_\rho|^2
    \,d\nu_\rho\\
    &=\frac1\rho
    \int_Mz_\rho\chi'(z_\rho)\,d\nu_\rho                 
    -
    \int_M
    \chi''(z_\rho)|\nabla z_\rho|^2
    \,d\nu_\rho.
\end{align*}                                                    
We integrated by parts in the second equality. We substituted \eqref{eq:source-score-pointwise-evolution} and used
\[
    \Delta\chi(z_\rho)
    =
    \chi'(z_\rho)\Delta z_\rho
    +
    \chi''(z_\rho)|\nabla z_\rho|^2
\]
in the third equality. We used 
\[
    K\left(
        \Delta z_\rho
        +
        \langle\nabla\log K,\nabla z_\rho\rangle
    \right)
    =
    \operatorname{div}(K\nabla z_\rho),
\]
in the fourth equality. Finally, we integrated by parts in the fifth equality.
This proves \eqref{eq:source-score-chain-rule}.
\end{proof}

We next identify the optimizer in the dual problem. Let
\begin{equation}
\label{eq:legendre-transform-of-varphi}
    \varphi^*(a):=\sup_{r>0}\{ar-\varphi(r)\}
\end{equation}
denote the Legendre transform.

\begin{lemma}[Dual optimizer]
\label{lem:dual-optimizer-fenchel}
For every \(\rho>0\), the supremum defining
\(D_{\varphi,\rho}^*(z_\rho)\) in \eqref{eq:dual-varphi-divergence} is achieved by a unique positive density
\(u_\rho\). Moreover, there is a unique number \(\lambda_\rho\) such that
\begin{subequations}
\begin{align}
    \label{eq:dual-optimizer-euler}
    z_\rho-\lambda_\rho
    &=
    \varphi'(u_\rho),
    \qquad
    \int_Mu_\rho\,d\nu_\rho=1,\\
    \label{eq:dual-fenchel-formula}
    D_{\varphi,\rho}^*(z_\rho)
    &=
    \lambda_\rho+
    \int_M\varphi^*(z_\rho-\lambda_\rho)\,d\nu_\rho =
    \int_Mz_\rho u_\rho\,d\nu_\rho
    -
    D_\varphi(u_\rho\nu_\rho\mid\nu_\rho).
\end{align}
\end{subequations}
The maps \(\rho\mapsto\lambda_\rho\) and \(\rho\mapsto u_\rho\) are \(C^1\) for
\(\rho>0\).
\end{lemma}

\begin{proof}
By \eqref{eq:legendre-type-condition} and \(\varphi''>0\), the function \(\varphi':(0,\infty)\to\mathbb R\) is a \(C^3\) diffeomorphism. Since \(M\) is closed and \(z_\rho\) is smooth,
\[
    u_\lambda:=(\varphi')^{-1}(z_\rho-\lambda)
\]
is a positive \(C^3\) function for every \(\lambda\in\mathbb R\). The map
\[
    \lambda\longmapsto\int_Mu_\lambda\,d\nu_\rho
\]
is continuous and strictly decreasing. Because \(z_\rho\) is bounded, the limits in
\eqref{eq:legendre-type-condition} imply that this integral tends to \(+\infty\) as
\(\lambda\to-\infty\) and to \(0\) as \(\lambda\to+\infty\). Thus there is a unique
\(\lambda_\rho\) such that
\[
    \int_Mu_{\lambda_\rho}\,d\nu_\rho=1.
\]
Set \(u_\rho:=u_{\lambda_\rho}\). This proves \eqref{eq:dual-optimizer-euler}.

Fenchel's inequality gives, for every probability density \(u\),
\[
    (z_\rho-\lambda_\rho)u-\varphi(u)
    \le
    \varphi^*(z_\rho-\lambda_\rho).
\]
After integration and using \(\int_Mu\,d\nu_\rho=1\), we get
\[
    \int_Mz_\rho u\,d\nu_\rho
    -
    \int_M\varphi(u)\,d\nu_\rho
    \le
    \lambda_\rho+
    \int_M\varphi^*(z_\rho-\lambda_\rho)\,d\nu_\rho .
\]
Equality holds precisely when \(z_\rho-\lambda_\rho=\varphi'(u)\), hence
precisely for \(u=u_\rho\). This proves existence and uniqueness of the maximizer
and the first equality in \eqref{eq:dual-fenchel-formula}. The second equality
follows from Fenchel's identity
\[
    \varphi(u_\rho)+\varphi^*(z_\rho-\lambda_\rho)
    =
    (z_\rho-\lambda_\rho)u_\rho
\]
and \(\int_Mu_\rho\,d\nu_\rho=1\).

The \(C^1\)-dependence on \(\rho\) follows from the implicit function theorem
applied to
\[
    (\rho,\lambda)
    \longmapsto
    F(\rho,\lambda)
    :=
    \int_M
    (\varphi')^{-1}(z_\rho-\lambda)
    \,d\nu_\rho
    -1,
\]
since, at \(\lambda=\lambda_\rho\),
\[
    \partial_\lambda
    \left[
        \int_M(\varphi')^{-1}(z_\rho-\lambda)\,d\nu_\rho-1
    \right]
    =
    -
    \int_M\frac1{\varphi''(u_\rho)}\,d\nu_\rho
    <0.
\]
Indeed, for each fixed \(\rho_0>0\), the implicit function theorem gives a
neighborhood of \(\rho_0\) and a unique \(C^1\) function \(\lambda(\rho)\) such that
\[
    F(\rho,\lambda(\rho))=0,
    \qquad
    \lambda(\rho_0)=\lambda_{\rho_0}.
\]
By uniqueness of the normalizing constant, \(\lambda(\rho)=\lambda_\rho\) on this
neighborhood. Hence \(\rho\mapsto\lambda_\rho\) as well as \(u_\rho=(\varphi')^{-1}(z_\rho-\lambda_\rho)\) are \(C^1\) for \(\rho>0\).
\end{proof}

\begin{lemma}
Assume that \(\varphi''>0\) and that \(\varphi\) satisfies
\eqref{eq:legendre-type-condition}. If \(z_\rho-\lambda_\rho=\varphi'(u_\rho)\), then
\begin{equation}\label{eq:two-derivative-of-varphi-star}
    (\varphi^*)''(z_\rho-\lambda_\rho)
    |\nabla z_\rho|^2
    =
    \varphi''(u_\rho)|\nabla u_\rho|^2 .
\end{equation}
Furthermore, for \(a\in \R\) near \(0\),
\begin{equation}
\label{eq:varphi-star-local-expansion}
    \varphi^*(a)
    =
    a+\frac{a^2}{2{\varphi''(1)}}+O(a^3),
    \qquad
    (\varphi^*)'(a)
    =
    1+\frac {a}{\varphi''(1)}+O(a^2).
\end{equation}
\end{lemma}

\begin{proof}
Since \(\varphi''>0\) and \(\varphi\) satisfies
\eqref{eq:legendre-type-condition}, the supremum defining \(\varphi^*(a)=\sup_{r>0}\{ar-\varphi(r)\}\) is attained at the unique point \(r=(\varphi')^{-1}(a).\) Thus
\[
    \varphi^*(a)=ar-\varphi(r),
    \qquad
    \varphi'(r)=a.
\]
Differentiating with respect to \(a\) gives
\begin{align}\label{eq:legendre-identities}
    (\varphi^*)'(a)=(\varphi')^{-1}(a),
    \qquad
    (\varphi^*)''(a)
    =
    \frac1{\varphi''((\varphi')^{-1}(a))}.
\end{align}
Finally, using
\(z_\rho-\lambda_\rho=\varphi'(u_\rho)\), we get
\[
    (\varphi^*)''(z_\rho-\lambda_\rho)
    =
    \frac1{\varphi''(u_\rho)},
    \qquad
    \nabla z_\rho=\varphi''(u_\rho)\nabla u_\rho,
\]
which proves \eqref{eq:two-derivative-of-varphi-star}.  

By \eqref{eq:legendre-identities},
\[
    (\varphi^*)'(0)=1,
    \qquad
    (\varphi^*)''(0)=\frac1{\varphi''(1)}.
\]
Since \(\varphi(1)=\varphi'(1)=0\), we also have \(\varphi^*(0)=0\). Therefore,
near \(a=0\), we obtain the expansion \eqref{eq:varphi-star-local-expansion}.
\end{proof}

\begin{lemma}
\label{lem:lambda-rho-small}
Let \(\lambda_\rho\) be as in Lemma~\ref{lem:dual-optimizer-fenchel}. As
\(\rho\downarrow0\),
\[
    \lambda_\rho=O(\rho).
\]
\end{lemma}

\begin{proof}
Using \eqref{eq:legendre-identities}, the normalization equation for \(u_\rho\) in \eqref{eq:dual-optimizer-euler} becomes
\begin{equation}
\label{eq:lambda-rho-normalization}
    \int_M
    (\varphi^*)'(z_\rho-\lambda_\rho)\,d\nu_\rho
    =
    1.
\end{equation}
By Lemma~\ref{lem:short-time-source-score-moments},
\[
    \int_M z_\rho^2\,d\nu_\rho=O(\rho),
    \qquad
    \|z_\rho\|_{L^\infty(M)}=O(1), \qquad \text{as }\rho\to 0.
\]
In particular, \(z_\rho\to0\) in \(L^1(d\nu_\rho)\).

We first prove that \(\lambda_\rho\to0\). Choose \(C<\infty\) such that \(|z_\rho|\le C\) for all sufficiently small \(\rho\). By \eqref{eq:legendre-identities}, \((\varphi^*)'=(\varphi')^{-1}.\) Thus \((\varphi^*)'\) is strictly increasing,
\[
    \lim_{a\to-\infty}(\varphi^*)'(a)=0,
    \qquad
    \lim_{a\to+\infty}(\varphi^*)'(a)=+\infty,
    \qquad
    (\varphi^*)'(0)=1.
\]
Choose \(A>C\). If \(\lambda_\rho\ge A\), then \(z_\rho-\lambda_\rho\le C-A<0,\) so
\[
    \int_M(\varphi^*)'(z_\rho-\lambda_\rho)\,d\nu_\rho
    \le
    (\varphi^*)'(C-A)
    <1,
\]
contradicting \eqref{eq:lambda-rho-normalization}. Similarly, we can rule out
\(\lambda_\rho\le -A\). Hence \(\lambda_\rho\)
is bounded.

Let \(\rho_j\downarrow0\) be any sequence. Passing to a subsequence, we may assume \(\lambda_{\rho_j}\to\lambda_0.\) The quantities \(z_{\rho_j}-\lambda_{\rho_j}\) and \(-\lambda_0\) lie in a fixed
compact interval, and \((\varphi^*)'\) is Lipschitz on this interval. Therefore
\[
\begin{aligned}
    \left|
        \int_M
        (\varphi^*)'(z_{\rho_j}-\lambda_{\rho_j})\,d\nu_{\rho_j}
        -
        (\varphi^*)'(-\lambda_0)
    \right|
    &=
    \left|
        \int_M
        \left[
            (\varphi^*)'(z_{\rho_j}-\lambda_{\rho_j})
            -
            (\varphi^*)'(-\lambda_0)
        \right]d\nu_{\rho_j}
    \right|                                      \\
    &\le
    C\int_M
    \left|
        z_{\rho_j}-\lambda_{\rho_j}+\lambda_0
    \right|d\nu_{\rho_j}                         \\
    &\le
    C\int_M |z_{\rho_j}|\,d\nu_{\rho_j}
    +
    C|\lambda_{\rho_j}-\lambda_0|.
\end{aligned}
\]
Passing to the limit in \eqref{eq:lambda-rho-normalization} gives \((\varphi^*)'(-\lambda_0)=1.\) Using again \((\varphi^*)'=(\varphi')^{-1}\) and \(\varphi'(1)=0\), we get
\(\lambda_0=0\). Thus every subsequential limit of \(\lambda_\rho\) is zero, and
therefore \(\lambda_\rho\to0.\)

It remains to get the rate. Fix \(\delta>0\) small enough that the Taylor expansion \eqref{eq:varphi-star-local-expansion} holds for \(|a|\le\delta\). Since \(\lambda_\rho\to0\), for all sufficiently small
\(\rho\), we have \(|\lambda_\rho|\le \frac{\delta}{2}.\) On the set \(\{|z_\rho|\le\delta/2\}\), we may apply the Taylor expansion to
\(a=z_\rho-\lambda_\rho\). On the complementary set, the integrand is uniformly
bounded, while Chebyshev's inequality gives
\[
    \nu_\rho\bigl(\{|z_\rho|>\delta/2\}\bigr)
    \le
    \frac{4}{\delta^2}
    \int_M z_\rho^2\,d\nu_\rho
    =
    O(\rho).
\]
Hence integrating the Taylor expansion in
\eqref{eq:lambda-rho-normalization} gives
\[
\begin{aligned}
    1
    &=
    \int_M
    \left[
        1+\frac{z_\rho-\lambda_\rho}{\varphi''(1)}
        +
        O\bigl((z_\rho-\lambda_\rho)^2\bigr)
    \right]d\nu_\rho
    +
    O(\rho)             =
    1-\frac{\lambda_\rho}{\varphi''(1)}
    +
    O(\rho+\lambda_\rho^2),
\end{aligned}
\]
where we used
\[
    \int_M z_\rho\,d\nu_\rho=0,
    \qquad
    \int_M z_\rho^2\,d\nu_\rho=O(\rho).
\]
Thus
\[
    |\lambda_\rho|
    \le
    C(\rho+\lambda_\rho^2).
\]
Since \(\lambda_\rho\to0\), the quadratic term can be absorbed for all sufficiently
small \(\rho\). Therefore \(|\lambda_\rho|\le C\rho,\) which proves the lemma.
\end{proof}

\begin{proof}[Proof of Proposition~\ref{prop:dual-nonlinear-fisher-monotonicity}]
Let \(u_\rho\) and \(\lambda_\rho\) be given by
Lemma~\ref{lem:dual-optimizer-fenchel}. Differentiating
\eqref{eq:dual-fenchel-formula}, the resulting \(\lambda_\rho'\)-terms cancel because
\[
    \int_M(\varphi^*)'(z_\rho-\lambda_\rho)\,d\nu_\rho
    =
    \int_Mu_\rho\,d\nu_\rho
    =
    1.
\]
Hence, applying Lemma~\ref{lem:source-score-chain-rule} to
\[
    a\longmapsto \varphi^*(a-\lambda_\rho),
\]
with \(\lambda_\rho\) fixed inside the chain rule, and then using
\eqref{eq:two-derivative-of-varphi-star}, we obtain
\[
\begin{aligned}
    \frac{d}{d\rho}D_{\varphi,\rho}^*(z_\rho)
    &=
    \frac1\rho
    \int_M z_\rho u_\rho\,d\nu_\rho
    -
    \int_M
    (\varphi^*)''(z_\rho-\lambda_\rho)
    |\nabla z_\rho|^2
    \,d\nu_\rho                                      \\
    &=
    \frac1\rho
    \int_M z_\rho u_\rho\,d\nu_\rho
    -
    \int_M
    \varphi''(u_\rho)|\nabla u_\rho|^2
    \,d\nu_\rho .
\end{aligned}
\]
Using the second identity in \eqref{eq:dual-fenchel-formula}, this gives
\[
\begin{aligned}
    \frac{d}{d\rho}
    \left(
        \frac{1}{2\rho}D_{\varphi,\rho}^*(z_\rho)
    \right)
    &=
    \frac1{2\rho}
    \left[
        \frac1\rho
        D_\varphi(u_\rho\nu_\rho\mid\nu_\rho)
        -
        \int_M
        \varphi''(u_\rho)|\nabla u_\rho|^2
        \,d\nu_\rho
    \right]
    \le 0.
\end{aligned}
\]
The inequality is precisely the \(\varphi\)-Sobolev inequality (Proposition~\ref{prop:sharp-phi-sobolev-heat-kernel}). This proves
\eqref{eq:dual-nonlinear-fisher-dissipation}.

\begin{claim}[Small-scale endpoint]
\label{claim:dual-small-scale-endpoint}
As \(\rho\downarrow0\),
\begin{equation}
\label{eq:dual-small-scale-endpoint}
    D_{\varphi,\rho}^*(z_V^\rho)
    =
    \frac{\rho}{\varphi''(1)}
    |V|_{g_t}^2
    +
    o(\rho).
\end{equation}
\end{claim}

\begin{proof}
By \eqref{eq:dual-fenchel-formula} and
\eqref{eq:varphi-star-local-expansion},
\[
\begin{aligned}
    D_{\varphi,\rho}^*(z_\rho)
    &=
    \lambda_\rho+
    \int_M
    \left[
        z_\rho-\lambda_\rho
        +
        \frac{(z_\rho-\lambda_\rho)^2}{2{\varphi''(1)}}
    \right]d\nu_\rho
    +
    O\left(
        \int_M |z_\rho-\lambda_\rho|^3\,d\nu_\rho
    \right).
\end{aligned}
\]
The linear terms cancel by \eqref{eq:score-coordinate-mean-summary}. Moreover,
Lemma~\ref{lem:lambda-rho-small} gives \(\lambda_\rho^2=O(\rho^2)\), while
\eqref{eq:small-score-lp} with \(p=3\) gives
\[
    \int_M |z_\rho-\lambda_\rho|^3\,d\nu_\rho
    =
    O(\rho^{3/2})=o(\rho).
\]
Therefore
\[
    D_{\varphi,\rho}^*(z_\rho)
    =
    \frac1{2\varphi''(1)}
    \int_M z_\rho^2\,d\nu_\rho
    +
    o(\rho).
\]
Using \eqref{eq:small-score-second}, the equation \eqref{eq:dual-small-scale-endpoint} follows.
\end{proof}

Claim~\ref{claim:dual-small-scale-endpoint} gives
\[
    \lim_{\rho\downarrow0}
    \frac{\varphi''(1)}{2\rho}
    D_{\varphi,\rho}^*(z_V^\rho)
    =
    \frac12 |V|_{g_t}^2.
\]
Since
\[
    \rho\longmapsto
    \frac{1}{2\rho}D_{\varphi,\rho}^*(z_V^\rho)
\]
is nonincreasing, we get
\[
    \frac{\varphi''(1)}{2\tau}
    D_{\varphi,\tau}^*(z_V^\tau)
    \le
    \frac12 |V|_{g_t}^2,
\]
which is \eqref{eq:dual-nonlinear-fisher-endpoint}.
\end{proof}

\subsection{Proof of the spatial estimate for KL divergence}\label{subsec:KL-spatial-estimate}

\begin{proof}[Proof of Corollary~\ref{cor:KL-spatial-estimate}]
    By Corollary~\ref{cor:scale-sharp-phi-SDPI}, applied with reference heat kernel \(\nu_{x_0,t;t-\sigma}\) and evolving measure \(\nu_{x_1,t;t-\sigma}\), for \(0<\varepsilon<\tau\),
    \[
        D_{\varphi_{\rm KL}}
        \left(
        \nu_{x_1,t;t-\tau}
        \mid
        \nu_{x_0,t;t-\tau}
        \right)
        \le
        \frac{\varepsilon}{\tau}
        D_{\varphi_{\rm KL}}
        \left(
        \nu_{x_1,t;t-\varepsilon}
        \mid
        \nu_{x_0,t;t-\varepsilon}
        \right).
    \]
    Since both heat-kernel measures have mass one,
    \begin{align*}
        \varepsilon
        D_{\varphi_{\rm KL}}
        \left(
        \nu_{x_1,t;t-\varepsilon}
        \mid
        \nu_{x_0,t;t-\varepsilon}
        \right)
        &=
        \varepsilon
        \int_M
        \log K(x_1,t;y,t-\varepsilon)
        \,d\nu_{x_1,t;t-\varepsilon}(y)                         \\
        &\quad-
        \varepsilon
        \int_M
        \log K(x_0,t;y,t-\varepsilon)
        \,d\nu_{x_1,t;t-\varepsilon}(y).
    \end{align*}
    
    The standard short-time upper bound
    \[
        K(x_1,t;y,t-\varepsilon)\le C\varepsilon^{-n/2}
    \]
    implies
    \[
        \limsup_{\varepsilon\downarrow0}
        \varepsilon
        \int_M
        \log K(x_1,t;y,t-\varepsilon)
        \,d\nu_{x_1,t;t-\varepsilon}(y)
        \le0. 
    \]
    On the other hand, Lemma~\ref{lem:short-time-sharp-lower-bound} gives, for every \(\eta>0\),
    \[
        -\varepsilon\log K(x_0,t;y,t-\varepsilon)
        \le
        \frac14d_t^2(x_0,y)
        +
        \frac{\eta}{4}
        +
        O_\eta(\varepsilon|\log\varepsilon|).
    \]
    Using \(\nu_{x_1,t;t-\varepsilon}\rightharpoonup\delta_{x_1},\) we get
    \begin{align*}
        \limsup_{\varepsilon\downarrow0}
        \varepsilon
        D_{\varphi_{\rm KL}}
        \left(
        \nu_{x_1,t;t-\varepsilon}
        \mid
        \nu_{x_0,t;t-\varepsilon}
        \right)
        &\le
        \frac14d_t^2(x_0,x_1)+\frac{\eta}{4}.
    \end{align*}
    Letting \(\eta\downarrow0\),
    \[
        \limsup_{\varepsilon\downarrow0}
        \varepsilon
        D_{\varphi_{\rm KL}}
        \left(
        \nu_{x_1,t;t-\varepsilon}
        \mid
        \nu_{x_0,t;t-\varepsilon}
        \right)
        \le
        \frac14d_t^2(x_0,x_1).
    \]
    Letting \(\varepsilon\downarrow0\) in the strong data-processing inequality proves
    \[
        D_{\varphi_{\rm KL}}
        \left(
        \nu_{x_1,t;t-\tau}
        \mid
        \nu_{x_0,t;t-\tau}
        \right)
        \le
        \frac{d_t^2(x_0,x_1)}{4\tau}.
    \]
\end{proof}

\subsection{Proof of Wasserstein spatial estimate and monotonicity}\label{subsec:wasserstein-from-KL}

\begin{proof}[Proof of Corollary \ref{cor:wasserstein-from-KL}]
    Applying Corollary~\ref{cor:KL-spatial-estimate} on the time interval \([s_1,s_0]\) gives, for all \(z,w\in M\),
    \[
        D_{\varphi_{\rm KL}}
        \left(
        \nu_{z,s_0;s_1}
        \mid
        \nu_{w,s_0;s_1}
        \right)
        \le
        \frac{d_{s_0}^2(z,w)}{4(s_0-s_1)}.
    \]
    The sharp weighted log-Sobolev inequality (LSI) for \(\nu_{w,s_0;s_1}\) \cite[Theorem~1.10]{HeinNaber2014CPAM}, together with the Otto--Villani Theorem (see \cite{MR1760620} or Bakry--Gentil--Ledoux \cite[Theorem~9.6.1]{MR3155209}), gives the Talagrand inequality (see \cite{MR1392331} for the Euclidean Gaussian version)
    \[
        W_{2,g_{s_1}}^2
        \left(
        \eta,
        \nu_{w,s_0;s_1}
        \right)
        \le
        4(s_0-s_1)
        D_{\varphi_{\rm KL}}
        \left(
        \eta
        \mid
        \nu_{w,s_0;s_1}
        \right).
    \]
    Taking \(\eta=\nu_{z,s_0;s_1}\), we obtain
    \begin{align}\label{eq:w2-estimate}
        W_{2,g_{s_1}}
        \left(
        \nu_{z,s_0;s_1},
        \nu_{w,s_0;s_1}
        \right)
        \le
        d_{s_0}(z,w).
    \end{align}
    This proves the first assertion by taking \(s_0=t\), \(s_1=t-\tau\), \(z=x_1\), and \(w=x_0\), and using \(W_1\le W_2\).
    
    We now prove monotonicity. Let \(\pi\) be any coupling of \(\mu_0(s_0)\) and \(\mu_1(s_0)\). For each \((z,w)\), choose a coupling \(\pi_{z,w}\) of \(\nu_{z,s_0;s_1}\) and \(\nu_{w,s_0;s_1}\) such that, for a fixed \(\delta>0\),
    \[
        \int_{M\times M}
        d_{s_1}^p\,d\pi_{z,w}
        \le
        d_{s_0}^p(z,w)+\delta. 
    \]
    This is possible because \(1\le p\le2\) and the \(W_2\)-estimate \eqref{eq:w2-estimate} implies
    \[
        W_{p,g_{s_1}}
        \left(
        \nu_{z,s_0;s_1},
        \nu_{w,s_0;s_1}
        \right)
        \le
        d_{s_0}(z,w).
    \]
    By the reproduction formula for conjugate heat flows,
    \[
        \mu_i(s_1)
        =
        \int_M
        \nu_{z,s_0;s_1}\,d\mu_i(s_0)(z),
        \qquad i=0,1.
    \]
    Thus
    \[
        \Pi
        :=
        \int_{M\times M}
        \pi_{z,w}\,d\pi(z,w)
    \]
    is a coupling of \(\mu_0(s_1)\) and \(\mu_1(s_1)\). Therefore
    \begin{align*}
        W_{p,g_{s_1}}^p
        \left(
        \mu_0(s_1),
        \mu_1(s_1)
        \right)
        \le
        \int_{M\times M}
        \int_{M\times M}
        d_{s_1}^p\,d\pi_{z,w}\,d\pi(z,w)
        \le
        \int_{M\times M}
        d_{s_0}^p(z,w)\,d\pi(z,w)
        +\delta. 
    \end{align*}
    Letting \(\delta\downarrow0\) and then taking the infimum over all couplings \(\pi\) of \(\mu_0(s_0)\) and \(\mu_1(s_0)\) gives
    \[
        W_{p,g_{s_1}}
        \left(
        \mu_0(s_1),
        \mu_1(s_1)
        \right)
        \le
        W_{p,g_{s_0}}
        \left(
        \mu_0(s_0),
        \mu_1(s_0)
        \right).
    \]
\end{proof}

\subsection{Proof of spatial \texorpdfstring{$\varphi$}{varphi}-divergence estimate}\label{subsec:spatial-phi-divergence-estimate}

\begin{proof}[Proof of Proposition~\ref{prop:spatial-phi-divergence-estimate}]
    Set \(s=t-\tau\), and let \(\gamma:[0,L]\to M\) be a smooth curve from \(x_0\) to \(x_1\), parametrized by \(g_t\)-arclength. Put
    \[
        K_a=K(\gamma(a),t;\cdot,s),
        \qquad
        K_0=K(x_0,t;\cdot,s),
        \qquad
        \lambda_a=\frac{K_a}{K_0},
        \qquad
        h=\varphi^{1/p}.
    \]
    Define
    \[
        \Psi(a):=K_0^{1/p}h(\lambda_a)\in L^p(M,dg_s).
    \]
    Then
    \[
        \|\Psi(L)\|_{L^p(dg_s)}^p
        =
        D_\varphi
        \left(
        \nu_{x_1,t;s}
        \mid
        \nu_{x_0,t;s}
        \right),
        \qquad
        \Psi(0)=0. 
    \]
    For a.e. \(a\),
    \[
        \partial_a\Psi
        =
        K_0^{1/p}
        h'(\lambda_a)\lambda_a
        \nabla_{\dot\gamma(a)}^x\log K_a. 
    \]
    Thus
    \begin{align*}
        \|\partial_a\Psi\|_{L^p(dg_s)}^p
        &=
        \int_M
        \lambda_a^{p-1}|h'(\lambda_a)|^p
        \left|
        \nabla_{\dot\gamma(a)}^x\log K_a
        \right|^p d\nu_{\gamma(a),t;s}     \le
        A_{\varphi,p}
        \int_M
        \left|
        \nabla_{\dot\gamma(a)}^x\log K_a
        \right|^p d\nu_{\gamma(a),t;s}.
    \end{align*}
    Since \(1\le p\le2\) and \(d\nu_{\gamma(a),t;s}\) has unit mass,
    \[
        \left(
        \int_M
        \left|
        \nabla_{\dot\gamma(a)}^x\log K_a
        \right|^p d\nu_{\gamma(a),t;s}
        \right)^{1/p}
        \le
        \left(
        \int_M
        \left|
        \nabla_{\dot\gamma(a)}^x\log K_a
        \right|^2 d\nu_{\gamma(a),t;s}
        \right)^{1/2}.
    \]
    By the definition of \(g_\tau^F\) and the estimate \(g_\tau^F\le g_t\),
    \[
        \|\partial_a\Psi\|_{L^p(dg_s)}
        \le
        A_{\varphi,p}^{1/p}
        \left(
        \frac{1}{2\tau}g_\tau^F(\dot\gamma(a),\dot\gamma(a))
        \right)^{1/2}
        \le
        \frac{A_{\varphi,p}^{1/p}}{\sqrt{2\tau}}. 
    \]
    Therefore
    \[
        D_\varphi
        \left(
        \nu_{x_1,t;s}
        \mid
        \nu_{x_0,t;s}
        \right)^{1/p}
        =
        \|\Psi(L)-\Psi(0)\|_{L^p}
        \le 
        \int_0^L \|\partial_a\Psi\|_{L^p(dg_s)}\,da  \leq            
        A_{\varphi,p}^{1/p}\frac{L}{\sqrt{2\tau}}. 
    \]
    Taking the infimum over \(\gamma\) proves the estimate. Reversing the roles of \(x_0\) and \(x_1\) gives the reversed estimate.
\end{proof}

\subsection{Proof of Hellinger and total variation estimate}\label{subsec:hellinger-tv-from-phi}

\begin{proof}[Proof of Corollary~\ref{cor:hellinger-tv-from-phi}]
    For
    \[
        \varphi_{\rm H}(r)=\frac12(\sqrt r-1)^2
    \]
    one has \(D_{\varphi_{\rm H}}=d_{\rm H}^2\) and \(\varphi_{\rm H}''(r)=1/(4r^{3/2})\). For the spatial estimate, \(\varphi_{\rm H}^{1/2}\) satisfies
    \[
        r\left|\frac{d}{dr}\varphi_{\rm H}(r)^{1/2}\right|^2=\frac18
        \qquad\text{for a.e. }r>0.
    \]
    Thus \(A_{\varphi_{\rm H},2}=1/8\), and Proposition~\ref{prop:spatial-phi-divergence-estimate} with \(p=2\) gives the Hellinger estimate. For total variation,
    \[
        \varphi_{\rm TV}(r)=\frac12|r-1|,
        \qquad
        A_{\varphi_{\rm TV},1}=\frac12,
    \]
    and the same proposition with \(p=1\) gives the total variation estimate.
\end{proof}

\subsection{Proof of good comparable Fisher scale}
\label{subsec:GoodCompScaleFisher}

\begin{proof}[Proof of Proposition~\ref{prop:sec6-good-scale-fixed-plane}]
    By Fisher monotonicity, \(\Jcal_\sigma\) is nonincreasing in the scale \(\sigma\). Hence, for \(\theta\tau_0\le\sigma\le\tau_0\),
    \[
        \Jcal_\sigma\ge \Jcal_{\tau_0}
        \qquad \text{and}\qquad
        \tr_{\Pcal}(I-\Jcal_\sigma)
        \le
        \tr_{\Pcal}(I-\Jcal_{\tau_0})
        \le \delta. 
    \]
    This proves the first estimate in \eqref{eq:sec6-good-scale-estimates} at every scale in the interval.
    
    Integrating \(\Gcal_\sigma-\Jcal_\sigma = -\partial_{\log\sigma}\Jcal_\sigma\) gives
    \begin{align*}
        \int_{\theta\tau_0}^{\tau_0}
        \tr_{\Pcal}(\Gcal_\sigma-\Jcal_\sigma)\,d\log\sigma
        =
        \tr_{\Pcal}(\Jcal_{\theta\tau_0})
        -
        \tr_{\Pcal}(\Jcal_{\tau_0})       \le
        k-\tr_{\Pcal}(\Jcal_{\tau_0})       =
        \tr_{\Pcal}(I-\Jcal_{\tau_0})
        \le \delta. 
    \end{align*}
    Since the logarithmic length of the interval is \(\log(\theta^{-1})\), some \(\tau\in[\theta\tau_0,\tau_0]\) satisfies the second estimate in \eqref{eq:sec6-good-scale-estimates}:
    \[
        \tr_{\Pcal}(\Gcal_\tau-\Jcal_\tau)
        \le
        \frac{\delta}{\log(\theta^{-1})}.
    \]
    
    Finally, if \(\Pcal\) is a top Fisher \(k\)-plane at scale \(\tau_0\), then \(\tr_{\Pcal}(I-\Jcal_{\tau_0}) = \dfplus{k}(\mathbf x;\tau_0) \le\delta\) by \eqref{eq:sec6-top-deficit-Ky-Fan}, so the above argument applies.
\end{proof}

\subsection{Proof of centeredness of heat-score maps}
\label{subsec:sec6-forward-heat-score-centering}

\begin{proof}[Proof of Lemma~\ref{lem:sec6-forward-heat-score-centering}]
    For \(s\le\ell<t\), the heat kernel reproduction formula gives 
    \begin{align*}
        \int_M v_V(q,\ell)\,d\nu_\ell(q)
        &=
        \int_M
        \left[
        \int_M K(q,\ell;y,s)z_V^\tau(y)\,dg_s(y)
        \right]
        K(x,t;q,\ell)\,dg_\ell(q)                         \\
        &=
        \int_M z_V^\tau(y)K(x,t;y,s)\,dg_s(y)
        =
        \int_M z_V^\tau\,d\nu_s
        =
        0
    \end{align*}
    by the source-score centering identity \eqref{eq:score-coordinate-mean-summary}. At \(\ell=t\), since \(d\nu_t=\delta_x\),
    \[
        \int_M v_V(\cdot,t)\,d\nu_t
        =
        v_V(x,t)
        =
        \int_M K(x,t;y,s)z_V^\tau(y)\,dg_s(y)
        =
        \int_M z_V^\tau\,d\nu_s
        =
        0. 
    \]
\end{proof}

\subsection{Proof of terminal gradient formula}
\label{subsec:sec6-terminal-value-gradient}

\begin{proof}[Proof of Lemma~\ref{lem:sec6-terminal-value-gradient}]
    For \(W\in T_xM\), differentiating the terminal heat evolution in the source variable gives
    \begin{align*}
        \nabla_Wv_V(x,t)
        &=
        \int_M \nabla_W^xK(x,t;y,s)z_V^\tau(y)\,dg_s(y)      =
        \int_M
        \nabla_W^x\log K(x,t;y,s)\,
        z_V^\tau(y)\,d\nu_s(y) \\
        &=
        2\tau
        \int_M
        \nabla_W^x\log K\,
        \nabla_V^x\log K
        \,d\nu_s               =
        g_t(\Jcal_\tau V,W).
    \end{align*}
    Thus \(\nabla v_V(x,t)=\Jcal_\tau V\). Applying this to \(V\) and \(W\) gives
    \[
        \left\langle
        \nabla v_V,\nabla v_W
        \right\rangle_{g_t}(x,t)
        =
        g_t(\Jcal_\tau V,\Jcal_\tau W).
    \]
\end{proof}

\subsection{Proof of Hessian identity}
\label{subsec:hessian_identity}

\begin{proof}[Proof of Corollary~\ref{cor:sec6-operator-affineness-defect}]
    Since \(v_V(\cdot,s)=z_V^\tau\), the definition of \(G_\tau\) gives
    \[
        \int_M
        \left\langle
        \nabla v_V,\nabla v_W
        \right\rangle_{g_s}
        d\nu_s
        =
        G_\tau(V,W).
    \]
    At the terminal time, \(d\nu_t=\delta_x\), and Lemma~\ref{lem:sec6-terminal-value-gradient} gives
    \[
        \int_M
        \left\langle
        \nabla v_V,\nabla v_W
        \right\rangle_{g_t}
        d\nu_t
        =
        g_t(\Jcal_\tau V,\Jcal_\tau W).
    \]
    Applying the weighted Dirichlet identity, Lemma~\ref{lem:sec6-weighted-dirichlet-identity}, with \(u=v_V\) and \(w=v_W\), we obtain
    \[
        G_\tau(V,W)
        -
        g_t(\Jcal_\tau V,\Jcal_\tau W)
        =
        2\int_s^t\int_M
        \left\langle
        \nabla^2v_V,\nabla^2v_W
        \right\rangle_{g_\ell}
        \,d\nu_\ell\,d\ell. 
    \]
\end{proof}

\subsection{Proof of averaged gradient orthonormality from Hessian energy}
\label{subsec:AvgGradOrthoFromHessEnergy}

We use the following \(L^1\)-Poincaré inequality to pass from terminal gradient control and Hessian energy to averaged gradient orthonormality.

\begin{lemma}[\(L^1\)-Poincaré inequality]
\label{lem:sec6-intermediate-L1-poincare}
    There exists a universal constant \(C_P<\infty\) with the following property. For every \(\ell\in[s,t)\) and every smooth function \(\phi\) on \((M,g_\ell)\),
    \[
        \int_M
        \left|
        \phi-\int_M\phi\,d\nu_\ell
        \right|
        d\nu_\ell
        \le
        C_P\sqrt{t-\ell}
        \int_M|\nabla\phi|_{g_\ell}\,d\nu_\ell. 
    \]
\end{lemma}

\begin{proof}
    This is the \(p=1\) case of the conjugate heat-kernel \(L^p\)-Poincaré inequality, applied to the conjugate heat-kernel measure based at \((x,t)\) and evaluated at the intermediate time \(\ell\); see \cite[Section~11]{Bamler2020A}.
\end{proof}

\begin{proof}[Proof of Proposition~\ref{prop:sec6-averaged-gradient-orthonormality}]
    For \(1\le i,j\le k\), define
    \begin{align*}
        h_{ij}(\cdot,\ell)
        :=
        \left\langle
        \nabla v_i,\nabla v_j
        \right\rangle_{g_\ell}
        -
        \delta_{ij}, \qquad \text{and}\qquad 
        m_{ij}(\ell)
        :=
        \int_M h_{ij}(\cdot,\ell)\,d\nu_\ell. 
    \end{align*}
    We estimate the mean and the fluctuation separately.
    
    \begin{claim}\label{claim:mean-estimate}
        We have
        \begin{equation}\label{eq:sec6-mean-gradient-error}
            \tau^{-1}
            \int_s^t
            \sum_{i,j=1}^k
            |m_{ij}(\ell)|\,d\ell
            \le
            k^2\delta_1+2k\delta_2. 
        \end{equation}
    \end{claim}
    
    \begin{proof}[Proof of Claim \ref{claim:mean-estimate}]
        First apply Lemma~\ref{lem:sec6-weighted-dirichlet-identity} on the interval
        \([\ell,t]\) with \(u=v_i\) and \(w=v_j\). Since \(d\nu_t=\delta_x\), we obtain
        \begin{align*}
            m_{ij}(\ell)
            =
            \left\langle
            \nabla v_i,\nabla v_j
            \right\rangle_{g_t}(x,t)
            -
            \delta_{ij} +
            2\int_\ell^t\int_M
            \left\langle
            \nabla^2v_i,\nabla^2v_j
            \right\rangle_{g_q}
            \,d\nu_q\,dq. 
        \end{align*}
        Therefore, by \eqref{eq:terminal-gradient-and-hessian-estimate-hypothesis} and \(2|\langle S,T\rangle|\le |S|^2+|T|^2\),
        \[
            |m_{ij}(\ell)|
            \le
            \delta_1
            +
            \int_\ell^t\int_M
            \left(
            |\nabla^2v_i|_{g_q}^2
            +
            |\nabla^2v_j|_{g_q}^2
            \right)
            d\nu_q\,dq. 
        \]
        Integrating in \(\ell\), dividing by \(\tau=t-s\), and using
        \[
            \tau^{-1}
            \int_s^t
            \int_\ell^t F(q)\,dq\,d\ell
            =
            \tau^{-1}
            \int_s^t (q-s)F(q)\,dq
            \le
            \int_s^t F(q)\,dq,
        \]
        we get
        \begin{align*}
            \tau^{-1}\int_s^t |m_{ij}(\ell)|\,d\ell
            &\le
            \delta_1
            +
            \int_s^t\int_M
            \left(
            |\nabla^2v_i|_{g_q}^2
            +
            |\nabla^2v_j|_{g_q}^2
            \right)
            d\nu_q\,dq. 
        \end{align*}
        Summing over \(i,j\), and using \eqref{eq:terminal-gradient-and-hessian-estimate-hypothesis}, gives the desired estimate. 
    \end{proof}
    
    We next estimate the fluctuation.
    \begin{claim}\label{claim:fluctuation}
        We have
        \begin{equation}\label{eq:sec6-fluctuation-gradient-error}
        \begin{aligned}
            \tau^{-1}
            \int_s^t\int_M
            \sum_{i,j=1}^k
            |h_{ij}-m_{ij}(\ell)|
            \,d\nu_\ell\,d\ell           \le
            C(k)
            \sqrt{\delta_2(1+\delta_1+2\delta_2)}. 
        \end{aligned}
        \end{equation}
    \end{claim}
    
    \begin{proof}[Proof of Claim \ref{claim:fluctuation}]
        By the \(L^1\)-Poincaré inequality Lemma~\ref{lem:sec6-intermediate-L1-poincare},
        \[
            \int_M
            |h_{ij}-m_{ij}(\ell)|\,d\nu_\ell
            \le
            C_P\sqrt{t-\ell}
            \int_M|\nabla h_{ij}|_{g_\ell}\,d\nu_\ell. 
        \]
        Moreover,
        \[
            |\nabla h_{ij}|_{g_\ell}
            \le
            |\nabla^2v_i|_{g_\ell}|\nabla v_j|_{g_\ell}
            +
            |\nabla v_i|_{g_\ell}|\nabla^2v_j|_{g_\ell}.
        \]
        Thus
        \begin{align*}
            \tau^{-1}
            \int_s^t\int_M
            |h_{ij}-m_{ij}(\ell)|
            \,d\nu_\ell\,d\ell  
            \le
            C_P\tau^{-1}
            \int_s^t
            \sqrt{t-\ell}
            \int_M
            \left(
            |\nabla^2v_i|\,|\nabla v_j|
            +
            |\nabla v_i|\,|\nabla^2v_j|
            \right)
            d\nu_\ell\,d\ell. 
        \end{align*}
        
        We estimate the first term on the right. By Cauchy--Schwarz in spacetime,
        \begin{align}\label{eq:intermediate-step-for-fluctuation}
        \begin{aligned}
            &\tau^{-1}
            \int_s^t
            \sqrt{t-\ell}
            \int_M
            |\nabla^2v_i|\,|\nabla v_j|
            \,d\nu_\ell\,d\ell                                      \\
            &\qquad\le
            \tau^{-1}
            \left(
            \int_s^t
            (t-\ell)
            \int_M|\nabla^2v_i|^2\,d\nu_\ell\,d\ell
            \right)^{1/2}
            \left(
            \int_s^t
            \int_M|\nabla v_j|^2\,d\nu_\ell\,d\ell
            \right)^{1/2}\\
            &\qquad \le \left(
            \int_s^t\int_M|\nabla^2v_i|^2\,d\nu_\ell\,d\ell
            \right)^{1/2}
            \left(
            \tau^{-1}
            \int_s^t\int_M|\nabla v_j|^2\,d\nu_\ell\,d\ell
            \right)^{1/2},
        \end{aligned}
        \end{align}
        where we used \(t-\ell\le\tau\) in the second inequality.
        
        For the gradient term, apply Lemma~\ref{lem:sec6-weighted-dirichlet-identity} on \([\ell,t]\) with \(u=w=v_j\):
        \[
            \int_M|\nabla v_j|_{g_\ell}^2\,d\nu_\ell
            =
            |\nabla v_j|_{g_t}^2(x,t)
            +
            2\int_\ell^t\int_M
            |\nabla^2v_j|_{g_q}^2
            \,d\nu_q\,dq\leq 1+\delta_1+2\delta_2
        \]
        for every \(\ell\in[s,t]\), where we used \eqref{eq:terminal-gradient-and-hessian-estimate-hypothesis} in the last inequality.
        
        Summing over \(i,j\) in \eqref{eq:intermediate-step-for-fluctuation}, using Cauchy--Schwarz in the finite sum, and then using \eqref{eq:terminal-gradient-and-hessian-estimate-hypothesis}, we obtain the desired estimate.
    \end{proof}
    Using
    \[
        |h_{ij}|
        \le
        |m_{ij}(\ell)|
        +
        |h_{ij}-m_{ij}(\ell)|,
    \]
    together with \eqref{eq:sec6-mean-gradient-error} and
    \eqref{eq:sec6-fluctuation-gradient-error}, yields
    \begin{align*}
        \tau^{-1}
        \int_s^t\int_M
        \sum_{i,j=1}^k |h_{ij}|
        \,d\nu_\ell\,d\ell
        \le
        C(k)\left(
            \delta_1+\delta_2+
            \sqrt{\delta_2(1+\delta_1+2\delta_2)}
        \right).
    \end{align*}
\end{proof}

\subsection{Proof of canonical estimates from Fisher pinching and production}
\label{subsec:canonical_estimates_from_fisher_production}

\begin{proof}[Proof of Proposition \ref{prop:sec6-terminal-orthonormality-from-fisher-pinching}]
    By Lemma \ref{lem:sec6-terminal-value-gradient},
    \[
        \left\langle
        \nabla v_i,\nabla v_j
        \right\rangle_{g_t}(x,t)
        =
        g_t(\Jcal_\tau e_i,\Jcal_\tau e_j)
        =
        g_t(\Jcal_\tau^2e_i,e_j).
    \]
    Since \(0\le \Jcal_\tau\le I\), we have
    \[
        0\le I-\Jcal_\tau^2 =(I+\Jcal_\tau)(I-\Jcal_\tau)\le 2(I-\Jcal_\tau).
    \]
    Taking the trace over \(\Pcal\) gives
    \[
        \tr_\Pcal(I-\Jcal_\tau^2)
        \le
        2\tr_\Pcal(I-\Jcal_\tau)
        \le
        2\delta.
    \]
    The restriction of \(I-\Jcal_\tau^2\) to \(\Pcal\) is nonnegative definite, so every matrix coefficient in an orthonormal basis is bounded in absolute value by its trace. Therefore
    \begin{align*}
        \left|
        \left\langle
        \nabla v_i,\nabla v_j
        \right\rangle_{g_t}(x,t)
        -
        \delta_{ij}
        \right|
        =
        \left|
        g_t((\Jcal_\tau^2-I)e_i,e_j)
        \right| \le
        \tr_\Pcal(I-\Jcal_\tau^2) \le
        2\delta.
    \end{align*}
    
    We now use Corollary \ref{cor:sec6-operator-affineness-defect} and our hypothesis to obtain
    \begin{align*}
        2\int_{t-\tau}^t\int_M
        \sum_{i=1}^k
        |\nabla^2v_i|_{g_\ell}^2
        \,d\nu_\ell\,d\ell
        =
        \tr_{\Pcal}(\Gcal_\tau)
        -
        \tr_{\Pcal}(\Jcal_\tau^2)= \tr_{\Pcal}(\Gcal_\tau-\Jcal_\tau)
        +
        \tr_{\Pcal}(\Jcal_\tau-\Jcal_\tau^2)\leq \delta +\delta_1,
    \end{align*}
    where we used \(0\le \Jcal_\tau\le I\) to get \(0\le \Jcal_\tau-\Jcal_\tau^2\le I-\Jcal_\tau\) in the last inequality.
\end{proof}

\subsection{Proof that Fisher rank produces heat-kernel strong splitting}
\label{subsec:proof_fisher_rank_strong_splitting}

\begin{proof}[Proof of Theorem~\ref{thm:sec6-fisher-rank-produces-strong-splitting-map}]
    By Proposition~\ref{prop:sec6-good-scale-fixed-plane}, there exists \(\tau\in[\theta\tau_0,\tau_0]\) such that
    \[
        \tr_{\Pcal}(I-\Jcal_\tau)\le\delta
        \qquad \text{and} \qquad \tr_{\Pcal}(\Gcal_\tau-\Jcal_\tau)
        \le
        \frac{\delta}{\log(\theta^{-1})}.
    \]
    Let \(e_1,\ldots,e_k\) be any \(g_t\)-orthonormal basis of \(\Pcal\), and let \(v=(v_1,\ldots,v_k)\)  be the canonical heat-score map associated to \((\mathbf x,\tau,\Pcal,e_1,\ldots,e_k)\). The heat equation and centering for \(v_i\) follow from Lemma~\ref{lem:sec6-forward-heat-score-centering}. The initial \(L^2\)-size estimate follows from \eqref{eq:sec6-source-score-L2-bound}. Finally, Propositions~\ref{prop:sec6-terminal-orthonormality-from-fisher-pinching} and~\ref{prop:sec6-averaged-gradient-orthonormality} imply that
    \begin{align*}
        &\max_{1\le i,j\le k}\left|
        \left\langle
        \nabla v_i,\nabla v_j
        \right\rangle_{g_t}(x,t)
        -
        \delta_{ij}
        \right|+\tau^{-1}
        \int_s^t\int_M
        \sum_{i,j=1}^k
        \left|
        \left\langle
        \nabla v_i,\nabla v_j
        \right\rangle_{g_\ell}
        -
        \delta_{ij}
        \right|
        d\nu_\ell\,d\ell\\
        &\qquad +\int_s^t\int_M
        \sum_{i=1}^k
        |\nabla^2v_i|_{g_\ell}^2
        \,d\nu_\ell\,d\ell\leq C(\theta) \sqrt{\delta}.
    \end{align*}
\end{proof}

\subsection{Proof of top Fisher \texorpdfstring{\(\varepsilon\)}{epsilon}-regularity}\label{subsec:top-fisher-epsilon-regularity}

Let $K(x,t;\,\cdot,s)$ denote the conjugate heat kernel based at $(x,t)$, normalized to have unit mass at time $s<t$. If $\tau=t-s$ and
\[
    K(x,t;y,s)=(4\pi\tau)^{-n/2}e^{-f(y,s)},
    \qquad d\nu_s=K(x,t;\,\cdot,s)\,dg_s,
\]
then
\[
    \Nash_{x,t}(\tau):=\int_M f(\cdot,s)\,d\nu_s-\frac n2
\]
denotes the pointed Nash entropy. At time \(s=t-\tau\), we write the \textit{renormalized potential function}:
\begin{align}
    F:=f-\Nash_{x,t}(\tau).
\end{align}

We need the following result due to Yongjia Zhang on the exponential integrability of the renormalized potential function. See also \cite[Proposition 6.5]{Bamler2020C}. For completeness, we include an alternate proof.

\begin{proposition}\label{prop:zhang-moments}
    Let $(M^n,g_t)_{t\in I}$ be a Ricci flow on a connected closed manifold. Suppose that $(x,t)\in M\times I$, $\tau>0$, and \([t-3\tau,t]\subset I.\) Then, for every $0\le\beta<\tfrac12$,
    \begin{equation}
    \label{eq:zhang-moments}
        F\ge-C(n)\quad\text{on }M,
        \qquad \text{and} \qquad
        \int_M e^{\beta F}\,d\nu_s
        \le
        \exp\!\left\{
        \frac{n\beta}{2}
        +\frac{13n\beta}{8}\log\frac1{1-2\beta}
        \right\}.
    \end{equation}
\end{proposition}

\begin{proof}
    Since \((\partial_\sigma-\Delta)R=2|\operatorname{Ric}|^2\ge \frac2nR^2,\) the scalar parabolic maximum principle, applied from time $t-3\tau$, gives
    \[
        R(\cdot,\sigma)\ge-\frac{n}{2(\sigma-t+3\tau)}
        \qquad (t-3\tau<\sigma\le t).
    \]
    Hence
    \begin{equation}
    \label{eq:R-lower}
        R\ge-\frac{n}{4\tau}
        \qquad\text{on }M\times[s,t].
    \end{equation}
    Bamler's normalized heat-kernel bound \cite[Theorem~7.1]{Bamler2020A}, applied with \(R_{\min}=-n/(4\tau)\), yields
    \[
        K(x,t;y,s)\le C(n)\tau^{-n/2}
        e^{-\Nash_{x,t}(\tau)}.
    \]
    Comparing this with \(K(x,t;y,s)=(4\pi\tau)^{-n/2}e^{-f(y)}\) gives \(F\ge-C(n)\).
    
    We now prove the exponential integrability. Translate $t-3\tau$ to time $0$. The resulting flow is defined on $[0,3\tau]$, the conjugate heat kernel is a positive solution on $M\times[0,3\tau)$, and $s$ corresponds to time $2\tau\ge(3\tau)/2$, where the backward time is $3\tau-2\tau=\tau$. Therefore Kuang--Zhang's estimate \cite[Theorem~2.1(ii)]{MR2433960} gives, at time $s$,
    \begin{equation}
    \label{eq:harnack}
        2\Delta f-|\nabla f|^2+R\le\frac{3n}{\tau}.
    \end{equation}
    
    Fix $0\le\beta<\tfrac12$ and let \(d\nu:=d\nu_s\) and
    \[
        Z(\beta):=\int_M e^{\beta F}\,d\nu.
    \]
    Since $M$ is closed, $d\nu=(4\pi\tau)^{-n/2}e^{-f}dg_s$, and $\nabla F=\nabla f$, integration by parts gives
    \[
        \int_M (\Delta f)e^{\beta F}\,d\nu
        =(1-\beta)\int_M |\nabla f|^2e^{\beta F}\,d\nu.
    \]
    Consequently, \eqref{eq:R-lower} and \eqref{eq:harnack} imply
    \[
        (1-2\beta)\int_M |\nabla f|^2e^{\beta F}\,d\nu
        +\int_M R e^{\beta F}\,d\nu
        \le\frac{3n}{\tau}Z(\beta),
    \]
    and hence
    \begin{equation}
    \label{eq:weighted-energy}
        \tau\int_M |\nabla f|^2e^{\beta F}\,d\nu
        \le \frac{13n}{4(1-2\beta)}Z(\beta).
    \end{equation}
    
    The log-Sobolev inequality of Hein--Naber \cite[Theorem~1.10(2)]{HeinNaber2014CPAM}, applied in homogeneous form (to obtain it, apply the normalized inequality to $\sqrt{w/\int_Mw\,d\nu}$), states that
    \[
        \int_M w\log w\,d\nu
        -\left(\int_Mw\,d\nu\right)\log\left(\int_Mw\,d\nu\right)
        \le \tau\int_M\frac{|\nabla w|^2}{w}\,d\nu
    \]
    for every positive smooth $w$. Applying this to $w=e^{\beta F}$ and then using \eqref{eq:weighted-energy} yields
    \[
        \beta Z'(\beta)-Z(\beta)\log Z(\beta)
        \le \frac{13n}{4}\frac{\beta^2}{1-2\beta}Z(\beta).
    \]
    Thus, for $\beta>0$,
    \[
        \left(\frac{\log Z(\beta)}{\beta}\right)'
        \le\frac{13n}{4(1-2\beta)}.
    \]
    Moreover, since $Z(0)=1$,
    \[
        \lim_{\beta\downarrow0}\frac{\log Z(\beta)}{\beta}
        =\int_MF\,d\nu=\frac n2,
    \]
    where the last equality is the definition of $\Nash_{x,t}(\tau)$. Integration from $0$ to $\beta$ proves the exponential integrability.
\end{proof}

We also need the following lemma to control the Nash entropy of a closed Ricci flow from an almost Euclidean map in the spirit of \cite[Lemma~14.2]{Bamler2020C}. In contrast to \cite[Lemma~14.1]{Bamler2020C}, we neither assume static condition nor impose a lower bound on the Nash entropy. A lower bound on Nash entropy allows to obtain a scale where the flow is almost self-similar in which case good analytical estimates for the potential function are available, see \cite[\S 7]{Bamler2020C}, and the hypothesis of \cite[Lemma~14.2]{Bamler2020C} assumes such estimate. Such estimates are unavailable in our setup.

We write
\[
    d\gamma_n=(4\pi)^{-n/2}e^{-|z|^2/4}\,dz
\]
for the Gaussian probability measure on $\R^n$.

\begin{lemma}\label{lem:quantitative-area-estimate}
    Fix $A<\infty$ and $\beta>0$. There are constants $\eta_0=\eta_0(n,A,\beta)>0$ and $C=C(n,A,\beta)<\infty$ with the following property. Let $(M^n,g)$ be closed, let
    \[
        d\nu=(4\pi)^{-n/2}e^{-N-F}\,dg
    \] 
    be a probability measure, where $F\in C^\infty(M)$ and $N\in\R$, and let $y=(y_1,\ldots,y_n)\in C^\infty(M;\R^n)$. Write $G_{ij}=\langle\nabla y_i,\nabla y_j\rangle$. Suppose that $0<\eta\le\eta_0$ and
    \begin{subequations}
    \begin{align}
    \label{eq:quantitative-area-potential}
        &F\ge-A,
        \qquad
        \int_M\bigl(e^{\beta F}+|\nabla F|^2\bigr)\,d\nu\le A,
        \qquad
        \int_MF\,d\nu=\frac n2,\\
        \label{eq:quantitative-area-map}
        &\int_M|y|^4\,d\nu\le A,
        \qquad
        \left|\int_M|y|^2\,d\nu-2n\right|
        +\int_M|G-I|\,d\nu
        +\int_M\sum_i|\nabla^2y_i|^2\,d\nu
        \le\eta,\\
        \label{eq:quantitative-area-transport}
        &W_2(y_\#\nu,\gamma_n)\le A\eta^{1/3}.
        \end{align}
    \end{subequations}
    Then
    \begin{equation}
    \label{eq:quantitative-area-conclusion}
        N\ge-C\eta^{1/4}.
    \end{equation}
\end{lemma}

\begin{proof}
    All constants below depend only on $n,A,\beta$. We may assume $\eta_0\le1$. Fix a dimensional constant $\kappa>0$ so small that
    \[
        |G-I|\le2\kappa
        \quad\Longrightarrow\quad
        \frac12I\le G\le2I,
    \]
    and choose $\chi\in C^\infty([0,\infty),[0,1])$ such that
    \[
        \chi=1\ \text{on }[0,\kappa^2],
        \qquad
        \chi=0\ \text{on }[4\kappa^2,\infty),
        \qquad
        |\chi'|\le C\kappa^{-2}.
    \]
    Set
    \[
        \psi=\chi(|G-I|^2),
        \qquad
        m=\int_M\psi^2\,d\nu.
    \]
    
    \begin{claim}\label{claim:m-is-controlled}
        After decreasing $\eta_0$, we have
        \begin{equation}
        \label{eq:quantitative-area-cutoff}
            m\ge\frac12,
            \qquad
            1-m\le C\eta,
            \qquad
            \int_M|\nabla\psi|^2\,d\nu\le C\eta.
        \end{equation}
    \end{claim}
    
    \begin{proof}[Proof of Claim \ref{claim:m-is-controlled}]
        The lower bound on \(m\) follows from the second inequality in \eqref{eq:quantitative-area-cutoff} after decreasing \(\eta_0\). Since $1-\psi^2$ is supported on $\{|G-I|\ge\kappa\}$, Chebyshev's inequality gives
        \[
            1-m
            \le \nu\{|G-I|\ge\kappa\}
            \le \kappa^{-1}\int_M|G-I|\,d\nu
            \le C\eta.
        \]
        On $\operatorname{supp}\nabla\psi$ one has $\kappa\le|G-I|\le2\kappa$, and hence \(|\nabla\psi|\le C\kappa^{-1}|\nabla G|.\) Moreover,
        \[
            \nabla G_{ij}
            =\nabla^2y_i(\,\cdot\,,\nabla y_j)
             +\nabla^2y_j(\,\cdot\,,\nabla y_i),
        \]
        so, since $\operatorname{tr}G\le C_n$ on $\operatorname{supp}\nabla\psi$,
        \[
            |\nabla G|^2\le C\tr(G) \sum_i |\nabla^2 y_i|^2\le C\sum_i|\nabla^2y_i|^2.
        \]
        This proves the last estimate in \eqref{eq:quantitative-area-cutoff}.
    \end{proof}
    
    On $\operatorname{supp}\psi$ we have $\frac12I\le G\le2I$. Since \(G=Dy(Dy)^T\), the differential $Dy$ has full rank on $\operatorname{supp}\psi$, and compactness shows that $y$ is a local diffeomorphism on a neighborhood of $\operatorname{supp}\psi$. Put
    \[
        d\widetilde\nu=m^{-1}\psi^2\,d\nu,
        \qquad
        \widetilde\mu=y_\#\widetilde\nu.
    \]
    
    \begin{claim}\label{claim:density-of-push-forward}
        The measure $\widetilde\mu$ has a density $p$ with respect to Lebesgue measure, given for almost every $z\in\R^n$ by
        \begin{equation}
        \label{eq:quantitative-area-density}
            p(z)=\frac{(4\pi)^{-n/2}e^{-N}}{m}
            \sum_{q\in y^{-1}(z)\cap\{\psi>0\}}
            \frac{\psi(q)^2e^{-F(q)}}{J_y(q)},
            \qquad
            J_y=(\det G)^{1/2}.
        \end{equation}
        Each sum in \eqref{eq:quantitative-area-density} is finite.
    \end{claim}
    
    \begin{proof}[Proof of Claim \ref{claim:density-of-push-forward}]
        By the area formula for maps between Riemannian manifolds (see Federer \cite[Theorem 3.2.3 and \S 3.2.46]{MR257325} or Cibotaru and de Lira \cite[Corollary~4.6]{MR3458179}), for every nonnegative Borel function $\phi$ on $\R^n$,
        \begin{align*}
            \int_{\R^n}\phi\,d\widetilde\mu =
            \frac{(4\pi)^{-n/2}e^{-N}}{m}
            \int_M\phi(y)\psi^2e^{-F}\,dg=
            \frac{(4\pi)^{-n/2}e^{-N}}{m}
            \int_{\R^n}\phi(z)
            \sum_{q\in y^{-1}(z)\cap\{\psi>0\}}
            \frac{\psi(q)^2e^{-F(q)}}{J_y(q)}\,dz.
        \end{align*}
        Thus $\widetilde\mu=p\,dz$ with $p$ as in \eqref{eq:quantitative-area-density}. The sums are finite: an infinite fiber in $\operatorname{supp}\psi$ would have an accumulation point, contradicting the local injectivity of $y$.  
    \end{proof}
    
    \begin{claim}\label{claim:Fisher-information-is-controlled}
        The Fisher information of \(\widetilde{\mu}\) with respect to \(\gamma_n\) is bounded:
        \begin{equation}
        \label{eq:quantitative-area-fisher}
            I(\widetilde\mu\mid\gamma_n)
            :=\int_{\R^n}
            \left|\nabla\log\frac p{\gamma_n}\right|^2p\,dz
            \le C.
        \end{equation}
    \end{claim}
    
    \begin{proof}[Proof of Claim \ref{claim:Fisher-information-is-controlled}]
        For $q\in\{\psi>0\}$, set
        \[
            b(q)=(Dy_q)^{-T}
            \bigl(2\nabla\log\psi-\nabla F-\nabla\log J_y\bigr).
        \]
        For almost every $z$ with $p(z)>0$, the finite fiber $y^{-1}(z)\cap\operatorname{supp}\psi$ extends to finitely many local inverse branches $q_\alpha=q_\alpha(z)$. If
        \[
            a_\alpha(z)
            =\frac{\psi(q_\alpha(z))^2e^{-F(q_\alpha(z))}}
                {J_y(q_\alpha(z))},
            \qquad
            c=\frac{(4\pi)^{-n/2}e^{-N}}m,
        \]
        then $p=c\sum_\alpha a_\alpha$. Since $Dq_\alpha=(Dy_{q_\alpha})^{-1}$, the chain rule gives \(\nabla a_\alpha(z)=a_\alpha(z)b(q_\alpha(z)).\) Consequently,
        \[
            \nabla\log p(z)
            =\sum_{q\in y^{-1}(z)\cap\{\psi>0\}}\lambda_q(z)b(q),
            \qquad
            \lambda_q(z)
            =\frac{\psi(q)^2e^{-F(q)}/J_y(q)}
            {\displaystyle
             \sum_{q'}\psi(q')^2e^{-F(q')}/J_y(q')}.
        \]
        The weights $\lambda_q\geq 0$ and \(\sum_q\lambda_q=1\). Since $y(q)=z$ on each fiber and $\nabla\log\gamma_n(z)=-z/2$, convexity of the squared norm yields
        \[
            \left|\nabla\log\frac p{\gamma_n}(z)\right|^2
            =
            \left|
            \sum_q\lambda_q(z)
            \left(b(q)+\frac{y(q)}2\right)
            \right|^2\le
            \sum_q\lambda_q(z)\left|b(q)+\frac{y(q)}2\right|^2.
        \]
        Multiplying by $p(z)$, integrating, and using the area formula once more, we obtain
        \begin{align*}
            I(\widetilde\mu\mid\gamma_n)
            &\le
            c\int_{\R^n}
            \sum_{q\in y^{-1}(z)\cap\{\psi>0\}}
            \frac{\psi(q)^2e^{-F(q)}}{J_y(q)}
            \left|b(q)+\frac{y(q)}2\right|^2\,dz=
            c\int_M
            \psi^2e^{-F}
            \left|b+\frac y2\right|^2\,dg\\
            &=\int_M\left|b+\frac y2\right|^2d\widetilde\nu.
        \end{align*}
        
        On $\operatorname{supp}\psi$, the singular values of $Dy$ lie in $[2^{-1/2},2^{1/2}]$, because $G=Dy(Dy)^*$. Hence
        \[
            \|(Dy)^{-1}\|+\|(Dy)^{-T}\|\le C.
        \]
        Furthermore,
        \[
            \nabla\log J_y=\frac12\operatorname{tr}(G^{-1}\nabla G),
            \qquad
            |\nabla\log J_y|^2\le C\sum_i|\nabla^2y_i|^2,
        \]
        where the second estimate follows from the formula for $\nabla G_{ij}$ above and $\operatorname{tr}G\le2n$. Therefore, on $\{\psi>0\}$,
        \[
            \psi^2\left|b+\frac y2\right|^2
            \le C\left(
            |\nabla\psi|^2
            +\psi^2|\nabla F|^2
            +\psi^2\sum_i|\nabla^2y_i|^2
            +\psi^2|y|^2
            \right).
        \]
        Using $m\ge\frac12$, \eqref{eq:quantitative-area-cutoff}, and the hypotheses of the lemma, we conclude that
        \begin{align*}
            I(\widetilde\mu\mid\gamma_n)
            &\le
            Cm^{-1}\int_M\left(
            |\nabla\psi|^2
            +\psi^2|\nabla F|^2
            +\psi^2\sum_i|\nabla^2y_i|^2
            +\psi^2|y|^2
            \right)d\nu
            \le C.
        \end{align*}
        Here we also used $\int|y|^2d\nu\le(\int|y|^4d\nu)^{1/2}$.
    \end{proof}
    
    \begin{claim}\label{claim:KL-upperbound}
        We have
        \begin{equation}
        \label{eq:quantitative-area-KL}
            D_{\varphi_{\rm KL}}(\widetilde\mu\mid\gamma_n)
            \le C\eta^{1/4}.
        \end{equation}
    \end{claim}
    
    \begin{proof}[Proof of Claim \ref{claim:KL-upperbound}]
        Let $\mu=y_\#\nu$ and $\alpha=1-m$. If $\alpha=0$, then $\mu=\widetilde\mu$. Otherwise, define
        \[
            \mu_{\rm b}
            =y_\#\bigl(\alpha^{-1}(1-\psi^2)\nu\bigr),
            \qquad
            \mu=m\widetilde\mu+\alpha\mu_{\rm b}.
        \]
        The measure
        \[
            \pi
            =m(\operatorname{id},\operatorname{id})_\#\widetilde\mu
             +\alpha\,\mu_{\rm b}\otimes\widetilde\mu
        \]
        is a coupling of $\mu$ and $\widetilde\mu$. Its diagonal part has zero cost, and therefore using \(|z-z'|^2\leq 2|z|^2+2|z'|^2\)
        \begin{align*}
            W_2(\mu,\widetilde\mu)^2
            &\le
            \alpha\int_{\R^n\times\R^n}|z-z'|^2
            \,d\mu_{\rm b}(z)\,d\widetilde\mu(z')=2\alpha\int_{\R^n}|z|^2\,d\mu_{\rm b}(z)
            +
            2\alpha\int_{\R^n}|z|^2\,d\widetilde\mu(z)\\
            &=
            2\int_M(1-\psi^2)|y|^2\,d\nu
            +2\alpha\int_{\R^n}|z|^2\,d\widetilde\mu(z).
        \end{align*}
        By Cauchy--Schwarz, \(0\leq 1-\psi^2\leq 1\), \eqref{eq:quantitative-area-map}, and \eqref{eq:quantitative-area-cutoff},
        \[
            \int_M(1-\psi^2)|y|^2\,d\nu
            \le
            \left(\int_M(1-\psi^2)^2\,d\nu\right)^{1/2}
            \left(\int_M|y|^4\,d\nu\right)^{1/2}\le C\alpha^{1/2},
        \]
        and
        \[
            \alpha\int_{\R^n}|z|^2\,d\widetilde\mu(z)
            =\frac\alpha m\int_M\psi^2|y|^2\,d\nu
            \le C\alpha.
        \]
        Thus
        \[
            W_2(\mu,\widetilde\mu)\le C\eta^{1/4}.
        \]
        The triangle inequality and \eqref{eq:quantitative-area-transport} give
        \[
            W_2(\widetilde\mu,\gamma_n)
            \le C\eta^{1/4}+A\eta^{1/3}
            \le C\eta^{1/4}.
        \]
        The Gaussian HWI inequality \cite[Corollary~9.3.3]{MR3155209}, applied by approximation if necessary, yields
        \[
            D_{\varphi_{\rm KL}}(\widetilde\mu\mid\gamma_n)
            \le
            W_2(\widetilde\mu,\gamma_n)
            I(\widetilde\mu\mid\gamma_n)^{1/2}
            -\frac14W_2(\widetilde\mu,\gamma_n)^2.
        \]
        Dropping the last term and using \eqref{eq:quantitative-area-fisher} proves \eqref{eq:quantitative-area-KL}.
    \end{proof}
    
    We first record the errors introduced by the cutoff. 
    \begin{claim}\label{claim:cutoff-comparison}
        The cutoff measure satisfies
        \begin{equation}
        \label{eq:quantitative-area-moments}
            \left|\int_MF\,d\widetilde\nu-\frac n2\right|
            +\left|\int_{\R^n}|z|^2\,d\widetilde\mu-2n\right|
            \le C\eta^{1/2}.
        \end{equation}
        Further,
        \begin{equation}
        \label{eq:quantitative-area-log-errors}
            -\int_M\log\psi\,d\widetilde\nu\le C\eta,
            \qquad
            \int_M|\log J_y|\,d\widetilde\nu\le C\eta.
        \end{equation}
        Here \(\int_M\log\psi\,d\widetilde\nu := \frac1m\int_{\{\psi>0\}}\psi^2\log\psi\,d\nu.\)
    \end{claim}
    
    \begin{proof}[Proof of Claim \ref{claim:cutoff-comparison}]
        Since $F\ge-A$,
        \[
            F^2\le C(1+e^{\beta F}),
            \qquad
            \int_MF^2\,d\nu\le C.
        \]
        Moreover,
        \[
            \|\psi^2-m\|_{L^2(\nu)}^2
            =\operatorname{Var}_\nu(\psi^2)
            \le\int_M(1-\psi^2)^2\,d\nu
            \le1-m.
        \]
        Thus, for $H=F$ and $H=|y|^2$,
        \[
            \left|\int_MH\,d\widetilde\nu-\int_MH\,d\nu\right|
            =\frac1m\left|\int_M(\psi^2-m)H\,d\nu\right|
            \le C\eta^{1/2}\left(\int_MH^2\,d\nu\right)^{1/2}.
        \]
        The estimates in \eqref{eq:quantitative-area-moments} then follow from our hypothesis \eqref{eq:quantitative-area-potential} and \eqref{eq:quantitative-area-map} and 
        \[
            \int_{\R^n}|z|^2\,d\widetilde\mu(z)
            =
            \int_M|y|^2\,d\widetilde\nu.
        \]
        
        Finally, with the convention $0\log0=0$,
        \[
            -\int_M\log\psi\,d\widetilde\nu
            =-\frac1m\int_M\psi^2\log\psi\,d\nu
            \le\frac C m\int_M(1-\psi^2)\,d\nu
            \le C\eta,
        \]
        where we used $r^2|\log r|\le C(1-r^2)$ for $0\le r\le1$ and \(m\geq 1/2\). Furthermore, on $\{\psi>0\}$,
        \[
            |\log J_y|=\frac12|\log\det G|\le C|G-I|,
        \]
        Therefore, using our hypothesis \eqref{eq:quantitative-area-map},
        \[
            \int_M|\log J_y|\,d\widetilde\nu
            =
            \frac1m\int_M\psi^2|\log J_y|\,d\nu
            \le
            C\int_M|G-I|\,d\nu
            \le C\eta,
        \]
        which gives the second estimate in \eqref{eq:quantitative-area-log-errors}.
    \end{proof}
    
    \begin{claim}\label{claim:KL-lowerbound}
        We have
        \begin{equation}
        \label{eq:quantitative-area-entropy-lower}
            D_{\varphi_{\rm KL}}(\widetilde\mu\mid\gamma_n)
            \ge-N-C\eta^{1/2}.
        \end{equation}
    \end{claim}
    
    \begin{proof}[Proof of Claim \ref{claim:KL-lowerbound}]
        For almost every $z$, set
        \[
            a_q=\frac{\psi(q)^2e^{-F(q)}}{J_y(q)},
            \qquad
            S(z)=\sum_{q\in y^{-1}(z)\cap\{\psi>0\}}a_q.
        \]
        By \eqref{eq:quantitative-area-density},
        \[
            p(z)=\frac{(4\pi)^{-n/2}e^{-N}}mS(z),
        \]
        and hence, on $\{p>0\}=\{S>0\}$,
        \[
            \log\frac{p(z)}{\gamma_n(z)}
            =-N-\log m+\log S(z)+\frac{|z|^2}{4}.
        \]
        Write $\lambda_q=a_q/S$. Then \(\lambda_q\ge0\), \(\sum_q\lambda_q=1\) so that
        \[
            \log S-\sum_q\lambda_q\log a_q
            =
            -\sum_q\lambda_q\log\lambda_q
            \ge0.
        \]
        Therefore
        \begin{align*}
            D_{\varphi_{\rm KL}}(\widetilde\mu\mid\gamma_n)&=
            -N-\log m
            +\int_{\R^n}p\log S\,dz
            +\frac14\int_{\R^n}|z|^2\,d\widetilde\mu(z)\\
            &\ge
            -N-\log m
            +\int_{\R^n}p(z)\sum_q\lambda_q(z)\log a_q\,dz
            +\frac14\int_{\R^n}|z|^2\,d\widetilde\mu(z).
        \end{align*}
        If $c=(4\pi)^{-n/2}e^{-N}/m$, then $p(z)\lambda_q(z)=ca_q$. The area formula, applied first to bounded truncations if necessary, yields
        \begin{align*}
            \int_{\R^n}
            p(z)\sum_q\lambda_q(z)\log a_q\,dz
            &=
            c\int_{\R^n}
            \sum_{q\in y^{-1}(z)\cap\{\psi>0\}}
            \frac{\psi(q)^2e^{-F(q)}}{J_y(q)}
            \log\frac{\psi(q)^2e^{-F(q)}}{J_y(q)}
            \,dz\\
            &=
            c\int_M
            \psi^2e^{-F}
            \log\frac{\psi^2e^{-F}}{J_y}\,dg=
            2\int_M\log\psi\,d\widetilde\nu
            -\int_MF\,d\widetilde\nu
            -\int_M\log J_y\,d\widetilde\nu.
        \end{align*}
        Consequently, 
        \begin{align}
        \label{eq:quantitative-area-log-sum}
            D_{\varphi_{\rm KL}}(\widetilde\mu\mid\gamma_n)
            \ge{}&
            -N-\log m
            +2\int_M\log\psi\,d\widetilde\nu
            -\int_MF\,d\widetilde\nu-\int_M\log J_y\,d\widetilde\nu
            +\frac14\int_{\R^n}|z|^2\,d\widetilde\mu(z).
        \end{align}
        Since $m\le1$, the term $-\log m$ is nonnegative. Combining \eqref{eq:quantitative-area-moments}, \eqref{eq:quantitative-area-log-errors}, and \eqref{eq:quantitative-area-log-sum} proves \eqref{eq:quantitative-area-entropy-lower}.
    \end{proof}
    
    Finally, \eqref{eq:quantitative-area-KL} and \eqref{eq:quantitative-area-entropy-lower} imply
    \[
        N\ge-C\eta^{1/4}-C\eta^{1/2}\ge-C\eta^{1/4},
    \]
    which proves \eqref{eq:quantitative-area-conclusion}.
\end{proof}

\begin{proof}[Proof of Theorem~\ref{thm:top-fisher-epsilon-regularity}]
    By parabolic rescaling and time translation, assume that \(\tau_0=1\) and \(t=0\). We may assume \(\delta>0\); the case \(\delta=0\) follows by approximation. Constants denoted by \(C_n\) may change from line to line.
    
    Apply Theorem~\ref{thm:sec6-fisher-rank-produces-strong-splitting-map} with \(k=n\) and \(\theta=\frac12\). There are \(\sigma\in[\frac12,1]\) and a canonical heat-score map
    \[
        v=(v_1,\ldots,v_n):M\times[-\sigma,0]\longrightarrow\R^n
    \]
    such that, writing
    \[
        G_{ij}(\ell)=
        \langle\nabla v_i,\nabla v_j\rangle_{g_\ell},
        \qquad
        \eta=C_n\sqrt\delta,
    \]
    we have
    \begin{equation}
    \label{eq:top-fisher-map-golfed}
        \sigma^{-1}\int_{-\sigma}^0\int_M|G-I|\,d\nu_\ell\,d\ell
        +
        \int_{-\sigma}^0\int_M\sum_i|\nabla^2v_i|^2\,d\nu_\ell\,d\ell
        \le\eta,
    \end{equation}
    and
    \begin{equation}
    \label{eq:top-fisher-initial-L2-golfed}
        \int_Mv_i(\cdot,-\sigma)^2\,d\nu_{-\sigma}\le2\sigma.
    \end{equation}
    Each \(v_i\) solves the heat equation and is centered with respect to \(\nu_\ell\). Since \(\nu_0=\delta_x\), this gives \(v(x,0)=0\).
    
    \begin{claim}\label{claim:splitting-map-is-controlled}
        There is \(a\in[\sigma/4,\sigma/3]\) such that, setting \(\bar g=a^{-1}g_{-a},\) \(y=a^{-1/2}v(\cdot,-a),\) \(d\nu=d\nu_{-a},\) and \((G_y)_{ij} = \langle\bar\nabla y_i,\bar\nabla y_j\rangle_{\bar g},\) we have
        \begin{subequations}
        \begin{align}
        \label{eq:top-fisher-static-errors-golfed}
            &\int_M|G_y-I|\,d\nu
            +
            \int_M\sum_i|\bar\nabla^2y_i|^2\,d\nu
            +\left|\int_M|y|^2\,d\nu-2n\right| \le C_n\eta,\\
            \label{eq:top-fisher-fourth-moment-golfed}
            &\int_M|y|^4\,d\nu\le C_n.
        \end{align}
        \end{subequations}
    \end{claim}
    
    \begin{proof}[Proof of Claim \ref{claim:splitting-map-is-controlled}]
        Averaging \eqref{eq:top-fisher-map-golfed} over \(a\in[\sigma/4,\sigma/3]\), we may choose \(a\) so that
        \begin{equation}
        \label{eq:top-fisher-good-time-golfed}
            \int_M|G(-a)-I|\,d\nu_{-a}
            +
            \sigma\int_M\sum_i|\nabla^2v_i(-a)|^2\,d\nu_{-a}
            \le C_n\eta.
        \end{equation}
        Under the above rescaling,
        \[
            G_y=G(-a),
            \qquad
            |\bar\nabla^2y_i|_{\bar g}^2
            =
            a|\nabla^2v_i|_{g_{-a}}^2,
        \]
        so the estimate for the first two integrals in \eqref{eq:top-fisher-static-errors-golfed} follows from \eqref{eq:top-fisher-good-time-golfed}.
        
        The matrix Bochner formula gives
        \[
            (\partial_\ell-\Delta)G_{ij}
            =
            -2\langle\nabla^2v_i,\nabla^2v_j\rangle.
        \]
        Applying the chain rule to smooth convex approximations of \(|G-I|\) yields
        \[
            (\partial_\ell-\Delta)|G-I|
            \le2\sum_i|\nabla^2v_i|^2.
        \]
        Since conjugacy implies
        \[
            \frac d{d\ell}\int_Mu\,d\nu_\ell
            =
            \int_M(\partial_\ell-\Delta)u\,d\nu_\ell,
        \]
        we obtain, for \(-a\le\ell\le0\),
        \[
            \int_M|G(\ell)-I|\,d\nu_\ell
            \le
            \int_M|G(-a)-I|\,d\nu_{-a}
            +
            2\int_{-a}^{\ell}\int_M
            \sum_i|\nabla^2v_i|^2\,d\nu_s\,ds.
        \]
        Thus
        \begin{equation}
        \label{eq:top-fisher-uniform-gram-golfed}
            \sup_{-a\le\ell\le0}
            \int_M|G(\ell)-I|\,d\nu_\ell
            \le C_n\eta.
        \end{equation}
        Finally,
        \[
            \frac d{d\ell}\int_Mv_iv_j\,d\nu_\ell
            =
            -2\int_MG_{ij}(\ell)\,d\nu_\ell.
        \]
        The terminal integral vanishes because \(v(x,0)=0\); hence
        \[
            \int_My_iy_j\,d\nu
            =
            \frac2a\int_{-a}^0\int_MG_{ij}(\ell)\,d\nu_\ell\,d\ell.
        \]
        Taking the trace and using \eqref{eq:top-fisher-uniform-gram-golfed} proves the estimate for the third integral in \eqref{eq:top-fisher-static-errors-golfed}.
        
        Since \(\sigma/a\ge3=(4-1)/(2-1)\), hypercontractivity \cite[Theorem~12.1]{Bamler2020A}, with \(q=2\) and \(p=4\), and \eqref{eq:top-fisher-initial-L2-golfed} give
        \[
            \left(\int_M|v_i(\cdot,-a)|^4\,d\nu_{-a}\right)^{1/4}
            \le \left(\int_Mv_i(\cdot,-\sigma)^2\,d\nu_{-\sigma}\right)^{1/2}\le(2\sigma)^{1/2}.
        \]
        As \(\sigma/a\le4\),
        \[
            \int_M|y_i|^4\,d\nu
            \le4\left(\frac{\sigma}{a}\right)^2\le C.
        \]
        Summing over \(i\) proves \eqref{eq:top-fisher-fourth-moment-golfed}.
    \end{proof}
    
    \begin{claim}\label{claim:w2-is-controlled}
        For \(\mu=y_\#\nu\),
        \begin{equation}
        \label{eq:top-fisher-W2-golfed}
            W_2(\mu,\gamma_n)\le C_n\eta^{1/3}.
        \end{equation}
    \end{claim}
    
    \begin{proof}[Proof of Claim \ref{claim:w2-is-controlled}]
        Let \(P_r^{\R^n}\) be the Euclidean heat semigroup normalized by
        \[
            P_r^{\R^n}\phi(z)
            =
            \int_{\R^n}\phi(z+\sqrt r\,\xi)\,d\gamma_n(\xi), \qquad \partial_r P_r^{\R^n}\phi=\Delta_{\R^n} P_r^{\R^n}\phi.
        \]
        For a smooth \(1\)-Lipschitz function \(\phi\), put
        \[
            \ell_\theta=-a\theta,
            \qquad
            w_\theta=a^{-1/2}v(\cdot,\ell_\theta), \qquad
            \mathcal A(\theta)
            =
            \int_M P_{1-\theta}^{\R^n}\phi(w_\theta)\,d\nu_{\ell_\theta}.
        \]
        Conjugacy gives
        \[
            \frac d{d\theta}\int_MU_\theta\,d\nu_{\ell_\theta}
            =
            \int_M(\partial_\theta U_\theta
                +a\Delta_{g_{\ell_\theta}}U_\theta)\,d\nu_{\ell_\theta}.
        \]
        Using \(\partial_\ell v=\Delta v\), the first-order terms cancel and
        \begin{equation}
        \label{eq:top-fisher-interpolation-derivative}
            \mathcal A'(\theta)
            =
            \int_M
            \left\langle
            G(\ell_\theta)-I,
            D^2P_{1-\theta}^{\R^n}\phi(w_\theta)
            \right\rangle_{\rm HS}
            d\nu_{\ell_\theta},
        \end{equation}
        where \( \langle A,B\rangle_{\rm HS} = \sum_{i,j}A_{ij}B_{ij}.\) Gaussian integration by parts gives
        \[
            \|D^2P_r^{\R^n}\phi\|_\infty
            \le C_nr^{-1/2}\|\nabla\phi\|_\infty,
        \]
        and therefore, by \eqref{eq:top-fisher-uniform-gram-golfed},
        \[
            |\mathcal A'(\theta)|
            \le C_n\eta(1-\theta)^{-1/2}.
        \]
        Moreover,
        \[
            \mathcal A(0)=\int_{\R^n}\phi\,d\gamma_n,
            \qquad
            \mathcal A(1)=\int_{\R^n}\phi\,d\mu.
        \]
        Integration and Kantorovich--Rubinstein duality yield
        \[
            W_1(\mu,\gamma_n)\le C_n\eta.
        \]
        
        Let \(\pi\) be a \(W_1\)-optimal coupling of \(\mu\) and \(\gamma_n\). Hölder's inequality gives
        \[
            \int|z-z'|^2\,d\pi
            \le
            \left(\int|z-z'|\,d\pi\right)^{2/3}
            \left(\int|z-z'|^4\,d\pi\right)^{1/3}.
        \]
        Using \(|z-z'|^4\leq C|z|^4+C|z'|^4\), the last factor is bounded by \eqref{eq:top-fisher-fourth-moment-golfed} and the Gaussian fourth moment. Hence
        \[
            W_2(\mu,\gamma_n)^2
            \le C_nW_1(\mu,\gamma_n)^{2/3},
        \]
        which proves \eqref{eq:top-fisher-W2-golfed}.
    \end{proof}
    
    \begin{claim}\label{claim:potential-is-controlled}
        Let \(f_a\) be the heat-kernel potential at time \(-a\), and set \(N=\Nash_{\mathbf x}(a),\) \(F=f_a-N.\) Then \(d\nu=(4\pi)^{-n/2}e^{-N-F}\,d\bar g\) is a probability measure and for some fixed \(\beta\in(0,\frac12)\),
        \begin{equation}
        \label{eq:top-fisher-potential-golfed}
            F\ge-C_n \, ,
            \qquad
            \int_M
            \bigl(e^{\beta F}+|\bar\nabla F|^2\bigr)\,d\nu
            \le C_n \, , 
            \qquad 
            \int_MF\,d\nu=\frac n2 \, .
        \end{equation}
    \end{claim}
    
    \begin{proof}[Proof of Claim \ref{claim:potential-is-controlled}]
        The density formula follows from
        \[
            d\nu_{-a}
            =
            (4\pi a)^{-n/2}e^{-f_a}\,dg_{-a},
            \qquad
            d\bar g=a^{-n/2}dg_{-a},
        \]
        while the definition of the pointed Nash entropy gives
        \[
            \int_MF\,d\nu
            =
            \int_Mf_a\,d\nu-N
            =
            \frac n2.
        \]
        Since \(3a\le\sigma\), Proposition~\ref{prop:zhang-moments} yields
        \[
            F\ge-C_n,
            \qquad
            \int_Me^{\beta F}\,d\nu\le C_n.
        \]
        
        Perelman's inequalities \(\mathcal W_{\mathbf x}(a)\le\Nash_{\mathbf x}(a)\le0\) and  the identity
        \[
            \mathcal W_{\mathbf x}(a)
            =
            a\int_M\bigl(|\nabla f_a|^2+R\bigr)\,d\nu
            +
            \Nash_{\mathbf x}(a)-\frac n2
        \]
        imply
        \[
            a\int_M\bigl(|\nabla f_a|^2+R\bigr)\,d\nu
            \le\frac n2.
        \]
        The scalar maximum principle, applied from time \(-1\), gives \(R(\cdot,-a)\ge-\frac{n}{2(1-a)}.\) Since \(a\le1/3\),
        \[
            \int_M|\bar\nabla F|^2\,d\nu
            =
            a\int_M|\nabla f_a|^2\,d\nu
            \le
            \frac n2+\frac{na}{2(1-a)}
            \le C_n.
        \]
    \end{proof}
    
    Claim \ref{claim:splitting-map-is-controlled}, Claim \ref{claim:w2-is-controlled} and Claim \ref{claim:potential-is-controlled} verify the hypotheses of Lemma~\ref{lem:quantitative-area-estimate}, with \(C_n\eta\) in place of \(\eta\). Choosing \(\delta_n\) sufficiently small, the lemma gives
    \[
        \Nash_{\mathbf x}(a)
        =
        N
        \ge
        -C_n\eta^{1/4}
        \ge
        -C_n\delta^{1/8}.
    \]
    Since \(a\ge\sigma/4\ge1/8\) and the pointed Nash entropy is nonincreasing in the backward scale,
    \[
        \Nash_{\mathbf x}(1/8)
        \ge
        \Nash_{\mathbf x}(a)
        \ge
        -C_n\delta^{1/8}.
    \]
    Together with \(\Nash_{\mathbf x}(1/8)\le0\), this proves the theorem after restoring the original scale and taking \(c_n=1/8\).
\end{proof}

\subsection{Proof of small Nash implies small average top Fisher deficit}\label{subsec:nash-gives-good-fisher-scale}

\begin{proof}[Proof of Proposition \ref{prop:nash-gives-good-fisher-scale}]
For fixed \(s\), set
\[
    \mathcal N_s^*(y,r)
    :=
    \Nash_{(y,r)}(r-s),
    \qquad r>s.
\]
Then \(\mathcal N_s^*(\cdot,r)\to0\) as \(r\downarrow s\), and the trace identity \eqref{eq:trace-deficit-Nash} gives
\[
    (\partial_r-\Delta_{g_r})\mathcal N_s^*
    =
    -\frac{1}{2(r-s)}
    \delta_n^+\bigl((\,\cdot\,,r);r-s\bigr).
\]
Applying Duhamel's formula for the forward heat equation yields
\[
    \mathcal N_s^*(x,t)
    =
    -\frac12
    \int_s^t\frac1{r-s}
    \int_M
    \delta_n^+\bigl((y,r);r-s\bigr)
    K(x,t;y,r)\,dg_r(y)\,dr.
\]
Since \(K(x,t;\cdot,r)\,dg_r=d\nu_{x,t;r}\), this is \eqref{eq:nash-integrated-fisher-deficit}. 
\end{proof}

\subsection{Proof of the quantitative pointwise converse}
\label{subsec:quantitative-nash-to-fisher}

We first record the local estimate used in the proof.

\begin{lemma}\label{lem:interior-full-fisher-regularity}
There are dimensional constants \(\varepsilon_*>0,\) \(c_*\in(0,1),\) \(L<\infty,\) \(\Lambda<\infty,\) \(m_0>0\) with the following property. Suppose that \(\Nash_{\mathbf x}(\tau_0)\ge-\varepsilon_*\). Set \(\sigma=c_*\tau_0,\) \(s=t-\sigma,\) and, for \(s<r\le t\), define
\[
    d_s(y,r)
    :=
    \dfplus{n}\bigl((y,r);r-s\bigr)
    =
    \tr_{g_r}D^F_{r-s}(y,r;s).
\]
Then
\begin{equation}
\label{eq:interior-full-fisher-derivatives}
    \sigma|\partial_r d_s|
    +
    \sigma|\nabla^2d_s|
    \le L \qquad \text{on} \qquad \bigcup_{r\in[t-\sigma/4,t]}
    B_{g_r}(x,2\sqrt\sigma)\times\{r\}.
\end{equation}
Moreover, for every \(0<u\le\sigma/4\),
\begin{equation}
\label{eq:short-heat-kernel-concentration}
    \nu_{x,t;t-u}
    \bigl(B_{g_{t-u}}(x,\Lambda\sqrt u)\bigr)
    \ge m_0.
\end{equation}
\end{lemma}

\begin{proof}
We may restrict to the connected component containing \(x\). By parabolic
rescaling and time translation, assume first that \(\tau_0=1,\) \(t=0.\) Fix a sufficiently small dimensional number \(\eta\in(0,1)\). After decreasing
\(\varepsilon_*\), Bamler's quantitative pointed-Nash regularity theorem
\cite[Theorem~10.3]{Bamler2020A} gives
\begin{equation}
\label{eq:large-conventional-regular-neighborhood}
    |\Rm|\le\eta,
    \qquad
    \Nash_{(w,q)}(q+1)\ge-\eta \qquad \text{on} \qquad B_{g_0}(x,\eta^{-1})\times[-(1-\eta),0].
\end{equation}
Choose \(\sigma=c_*=c_*(n)>0\) so small that
\begin{equation}
\label{eq:c-star-choice}
    c_*\le\frac14,
    \qquad
    \frac\sigma2\le1-\eta,
    \qquad
    8e^{C_n\eta\sigma}\sqrt\sigma<\eta^{-1}.
\end{equation}
On the fixed neighborhood in
\eqref{eq:large-conventional-regular-neighborhood}, the Ricci flow equation
implies \(e^{-C_n\eta|q|}g_0 \le g_q \le e^{C_n\eta|q|}g_0.\) A first-exit argument therefore shows that
\begin{equation}
\label{eq:moving-cylinder-contained}
    B_{g_q}(x,8\sqrt\sigma)
    \subset B_{g_0}(x,\eta^{-1})
    \qquad
    \left(-\frac\sigma2\le q\le0\right).
\end{equation}
Indeed, if a \(g_q\)-minimizing geodesic from \(x\) first met
\(\partial B_{g_0}(x,\eta^{-1})\), then the portion preceding the first exit
would have \(g_q\)-length at least
\(e^{-C_n\eta\sigma}\eta^{-1}>8\sqrt\sigma\), contradicting the choice of
its endpoint. Hence, on the cylinder in \eqref{eq:moving-cylinder-contained}, we have \(|\Rm|\le\eta\) and \(\Nash_{(w,q)}(q+\sigma) \ge \Nash_{(w,q)}(q+1) \ge-\eta,\) where we used the monotonicity of pointed Nash entropy in the
backward scale.

Rescale once more so that \(\sigma=1\) and \(s=-1\), and suppress the
rescaling notation. The flow is then defined on \([-c_*^{-1},0]\), and
\begin{equation}
\label{eq:rescaled-regular-cylinder}
    |\Rm|\le1,
    \qquad
    \Nash_{(w,q)}(q+1)\ge-\eta \qquad \text{on} \qquad \bigcup_{-1/2\le q\le0}B_{g_q}(x,8)\times\{q\}.
\end{equation}
The scalar maximum principle, applied from time \(-c_*^{-1}\), also gives
\begin{equation}
\label{eq:global-scalar-lower-for-source-estimates}
    R(\cdot,q)
    \ge
    -\frac{n}{2(q+c_*^{-1})}
    \ge-C_n
    \qquad
    (-1\le q\le0).
\end{equation}

Fix
\begin{align}\label{eq:y-under-consideration}
    y\in B_{g_r}(x,2),
    \qquad
    -\frac14\le r\le0,
    \qquad
    \rho:=r+1\in\left[\frac34,1\right],
\end{align}
and write
\[
    K(z):=K(y,r;z,-1),
    \qquad
    d\nu(z):=K(z)\,dg_{-1}(z).
\]
All derivatives of \(K\) below are taken in the source variable \(y\); the
variable \(z\) is only the target variable of integration.

\begin{claim}\label{claim:lp-of-higher-derivative-of-logk}
For every \(1\le p<\infty\) and \(j\in\{1,2,3\}\),
\begin{equation}
\label{eq:weighted-source-score-estimates}
    \int_M
    \left|\nabla_y^j\log K(y,r;z,-1)\right|^p
    d\nu(z)
    \le C(n,p,j).
\end{equation}
\end{claim}

\begin{proof}[Proof of Claim \ref{claim:lp-of-higher-derivative-of-logk}]
Choose a dimensional number \(\vartheta\in(0,1/8]\) sufficiently small. By
\eqref{eq:rescaled-regular-cylinder}, metric distortion, and the same first-exit
argument as above, the cylinders
\[
    Q_2(y,r)
    :=
    \bigcup_{r-2\vartheta\le q\le r}
    B_{g_q}(y,2)\times\{q\},
    \qquad
    Q_1(y,r)
    :=
    \bigcup_{r-\vartheta\le q\le r}
    B_{g_q}(y,1)\times\{q\}
\]
are contained in the region in \eqref{eq:rescaled-regular-cylinder}, uniformly
for all \((y,r)\) as in \eqref{eq:y-under-consideration}; in particular, the closure of
\(Q_1(y,r)\) is contained in \(Q_2(y,r)\). Indeed, after
decreasing \(\vartheta\), the same first-exit argument gives
\(d_{g_q}(x,y)\le3\) whenever \(|q-r|\le2\vartheta\); hence
\(B_{g_q}(y,2)\subset B_{g_q}(x,5)\).

For a fixed \(z\in M\), set
\[
    u_z(w,q):=K(w,q;z,-1).
\]
This is a positive solution of the forward heat equation in the source
variables. On \(Q_2(y,r)\), the elapsed time \(q+1\) is bounded above and
below by positive dimensional constants. Thus
\eqref{eq:rescaled-regular-cylinder},
\eqref{eq:global-scalar-lower-for-source-estimates}, and Bamler's heat-kernel
upper bound \cite[Theorem~7.1]{Bamler2020A} give a dimensional constant
\(A_0\ge e\), independent of \(z\), such that
\begin{equation}
\label{eq:source-heat-upper}
    0<u_z\le A_0
    \qquad\text{on }Q_2(y,r).
\end{equation}
For every \((w,q)\in Q_1(y,r)\), Bamler's source-gradient estimate
\cite[Theorem~7.5]{Bamler2020A}, applied with
\((x,t;y,s)=(w,q;z,-1)\), gives
\[
    |\nabla_w\log u_z|
    \le
    \frac{C_n}{\sqrt{q+1}}
    \left[
        \log\left(
            \frac{C_{0,n}e^{-\Nash_{(w,q)}(q+1)}}
            {(q+1)^{n/2}u_z}
        \right)
    \right]^{1/2}.
\]
Because \(q+1\in[1/2,1]\) and
\(\Nash_{(w,q)}(q+1)\ge-\eta\), we may increase \(A_0\), still
independently of \(z\), so that
\[
    \frac{C_{0,n}e^{-\Nash_{(w,q)}(q+1)}}{(q+1)^{n/2}}
    \le A_0
\]
throughout \(Q_1(y,r)\). Therefore
\[
    |\nabla_w\log u_z|
    \le
    C_n\left(\log\frac{A_0}{u_z}\right)^{1/2}.
\]
Together with \eqref{eq:source-heat-upper}, the chain rule yields
\begin{equation}
\label{eq:sqrt-log-lipschitz}
    \left|
       \nabla_w\sqrt{1+\log\frac{A_0}{u_z}}
    \right|
    \le C_n
    \qquad\text{on }Q_1(y,r).
\end{equation}

The local differential Harnack estimate of Baileseanu, Cao, and Pulemotov
\cite[Theorem~2.7]{MR2601627}, applied on \(Q_2(y,r)\)
after translating its lower time face to time zero, gives at \((y,q)\)
\[
    |\nabla\log u_z|^2-2\partial_q\log u_z\le C_n
    \qquad(r-\vartheta\le q\le r).
\]
Here the two-sided Ricci bound is dimensional and the elapsed time from the
lower face of \(Q_2(y,r)\) is at least \(\vartheta\). In particular,
\begin{equation}
\label{eq:local-time-harnack}
    \partial_q\log u_z(y,q)\ge-C_n
    \qquad(r-\vartheta\le q\le r).
\end{equation}

Fix \(\gamma\in(0,1)\). For
\(w\in B_{g_q}(y,1)\), integration of
\eqref{eq:sqrt-log-lipschitz} along a minimizing \(g_q\)-geodesic gives
\[
    \sqrt{1+\log\frac{A_0}{u_z(w,q)}}
    \ge
    \left(
        \sqrt{1+\log\frac{A_0}{u_z(y,q)}}-C_n
    \right)_+.
\]
The elementary inequality
\[
    \bigl(\sqrt{1+h}-C_n\bigr)_+^2-1
    \ge(1-\gamma)h-C_{n,\gamma}
    \qquad(h\ge0)
\]
therefore implies
\[
    \log\frac{A_0}{u_z(w,q)}
    \ge
    (1-\gamma)\log\frac{A_0}{u_z(y,q)}-C_{n,\gamma}.
\]
Integrating \eqref{eq:local-time-harnack} from \(q\) to \(r\) gives
\[
    \log\frac{A_0}{u_z(y,q)}
    \ge
    \log\frac{A_0}{u_z(y,r)}-C_n.
\]
Consequently,
\begin{equation}
\label{eq:source-sup-interpolation}
    \sup_{Q_1(y,r)}u_z
    \le
    C_{n,\gamma}A_0^\gamma u_z(y,r)^{1-\gamma}.
\end{equation}

Apply the local derivative estimates for heat solutions of Huang
\cite[Theorems~1.1 and~2.1]{MR4466164} to \(u_z\) on
\(Q_1(y,r)\), after translating \(r-\vartheta\) to time zero. Since
\(|\Rm|\le1\), \(\vartheta\) is fixed, and \((y,r)\) lies a fixed positive
parabolic distance from the lower face, these estimates give
\begin{equation}
\label{eq:source-derivatives-of-K}
    |\nabla_w^j u_z|(y,r)
    \le
    C_{n,j}\sup_{Q_1(y,r)}u_z,
    \qquad j=1,2,3.
\end{equation}
This estimate is local in the source variables, but it holds for every fixed
parameter \(z\), with constants independent of \(z\). From
\eqref{eq:source-sup-interpolation}--\eqref{eq:source-derivatives-of-K},
\[
    \frac{|\nabla_w^i u_z|(y,r)}{u_z(y,r)}
    \le
    C_{n,i,\gamma}
    \left(\frac{A_0}{K(z)}\right)^\gamma,
    \qquad i=1,2,3.
\]
Every term in \(\nabla^j\log u_z\) is a product of at most \(j\) such
factors. Since \(A_0/K(z)\ge1\), the chain rule therefore implies
\begin{equation}
\label{eq:log-source-pointwise}
    \left|\nabla_y^j\log K(y,r;z,-1)\right|
    \le
    C_{n,j,\gamma}
    \left(\frac{A_0}{K(z)}\right)^{j\gamma},
    \qquad
    j=1,2,3.
\end{equation}

Write
\[
    K(z)=(4\pi\rho)^{-n/2}e^{-f(z)},
    \qquad
    F(z):=f(z)-\Nash_{(y,r)}(\rho).
\]
Since \(\rho\in[3/4,1]\) and \(\Nash_{(y,r)}(\rho)\le0\),
\begin{equation}
\label{eq:A0-over-K-by-F}
    \frac{A_0}{K(z)}
    =
    A_0(4\pi\rho)^{n/2}
    e^{F(z)+\Nash_{(y,r)}(\rho)}
    \le C_ne^{F(z)}.
\end{equation}
Moreover,
\([r-3\rho,r]\subset[-3,0]\subset[-c_*^{-1},0]\). Hence
Proposition~\ref{prop:zhang-moments}, applied at \((y,r)\), gives, for any
fixed \(\beta\in(0,1/2)\),
\[
    \int_Me^{\beta F}\,d\nu\le C(n,\beta).
\]
Given \(p<\infty\), choose \(\gamma>0\) so that \(3p\gamma<\beta\). For
\(j\le3\), H\"older's inequality and the fact that \(d\nu\) is a probability
measure give
\[
    \int_Me^{jp\gamma F}\,d\nu
    \le
    \left(\int_Me^{\beta F}\,d\nu\right)^{jp\gamma/\beta}.
\]
Combining this with \eqref{eq:log-source-pointwise} and
\eqref{eq:A0-over-K-by-F} proves \eqref{eq:weighted-source-score-estimates}.
\end{proof}

Set
\[
    I(y,r)
    :=
    \int_M
    \left|\nabla_y\log K(y,r;z,-1)\right|_{g_r}^{2}
    d\nu(z).
\]

\begin{claim}\label{claim:fisher-trace-derivatives}
At every \((y,r)\) as in \eqref{eq:y-under-consideration},
\begin{equation}
\label{eq:fisher-trace-derivative-bounds}
    |\nabla_y^2I|+|\partial_rI|\le C_n.
\end{equation}
\end{claim}

\begin{proof}[Proof of Claim \ref{claim:fisher-trace-derivatives}]
Since \(M\) is closed and \(r+1\ge3/4\), the heat kernel and all of its
source derivatives are smooth on the region under consideration, so we may
differentiate under the integral.

A first source
derivative gives
\[
    \nabla_c I
    =
    \int_M
    \left[
        2\left\langle
            \nabla_y^2\log K(\partial_c,\cdot),
            \nabla_y\log K
        \right\rangle
        +
        |\nabla_y\log K|^2\nabla_c\log K
    \right]d\nu,
\]
where the last term comes from
\(\nabla_c d\nu=(\nabla_c\log K)d\nu\). Differentiating once more yields
\begin{align*}
    |\nabla_y^2 I|
    \le C_n\int_M
    \bigl(
        |\nabla_y^3\log K|\,|\nabla_y\log K|
        +|\nabla_y^2\log K|^2 +|\nabla_y^2\log K|\,|\nabla_y\log K|^2
        +|\nabla_y\log K|^4
    \bigr)\,d\nu.
\end{align*}
Thus \eqref{eq:weighted-source-score-estimates} and H\"older's inequality
give \(|\nabla_y^2I|\le C_n\).

For the time derivative, the forward heat equation in the source variables gives
\[
    \partial_r\log K
    =
    \Delta_{g_r,y}\log K
    +
    |\nabla_y\log K|^2.
\]
With the convention
\(\nabla^3h=\nabla(\nabla^2h)\):
\begin{align}\label{eq:time-derivatives-of-logK}
\begin{split}
    |\partial_r\log K|
    &\le
    C_n\bigl(|\nabla_y^2\log K|+|\nabla_y\log K|^2\bigr), \\
    |\nabla_y\partial_r\log K|
    &\le
    C_n\bigl(
        |\nabla_y^3\log K|
        +|\nabla_y^2\log K|\,|\nabla_y\log K|
    \bigr).
\end{split}
\end{align}
Since \(\partial_rg_r^{ij}=2\Ric_{g_r}^{ij}\) and
\(\partial_rd\nu=(\partial_r\log K)d\nu\), differentiation under the integral gives
\begin{align*}
\label{eq:I-time-derivative}
    \partial_rI
    =
    \int_M\bigl(
        2\Ric_{g_r}(y)(\nabla_y\log K,\nabla_y\log K)+2\langle\nabla_y\log K,
                    \nabla_y\partial_r\log K\rangle
        +|\nabla_y\log K|^2\partial_r\log K
    \bigr)\,d\nu.
\end{align*}
The Ricci tensor in this formula is evaluated at the source point \((y,r)\) and is independent of
the target variable \(z\).  In particular,
\[
    \left|
        \int_M
        \Ric_{g_r}(y)(\nabla_y\log K,\nabla_y\log K)
        \,d\nu
    \right|
    \le
    |\Ric_{g_r}(y)|I(y,r).
\] 
The local curvature bound in
\eqref{eq:rescaled-regular-cylinder},
\eqref{eq:weighted-source-score-estimates},
\eqref{eq:time-derivatives-of-logK}, and H\"older's inequality therefore
give \(|\partial_rI|\le C_n\). Here
\(I\le C_n\) follows already from
\eqref{eq:weighted-source-score-estimates} with \(j=1\) and \(p=2\).
This proves the claim.
\end{proof}

By the definition of the full Fisher deficit,
\begin{equation}
\label{eq:deficit-via-I}
    d_s(y,r)=n-2\rho I(y,r).
\end{equation}
Consequently,
\[
    \nabla_y^2d_s=-2\rho\nabla_y^2I,
    \qquad
    \partial_rd_s=-2I-2\rho\partial_rI.
\]
Claim~\ref{claim:lp-of-higher-derivative-of-logk} with \(p=2\), together with
Claim~\ref{claim:fisher-trace-derivatives}, therefore gives
\begin{align*}
    |\partial_rd_s|+|\nabla^2d_s|\le C_n \qquad \text{on} \qquad \bigcup_{-1/4\le r\le0}B_{g_r}(x,2)\times\{r\}.
\end{align*}
Under the parabolic rescaling used to normalize \(\sigma=1\),
\[
    |\partial_{r'}d_s|=\sigma|\partial_r d_s|,
    \qquad
    |\nabla_{g'}^2d_s|_{g'}=\sigma|\nabla_g^2d_s|_g.
\]
Thus undoing this normalization proves
\eqref{eq:interior-full-fisher-derivatives} in the normalization
\(\tau_0=1\).

It remains to prove \eqref{eq:short-heat-kernel-concentration}. Return to the
normalization \(\tau_0=1\), \(t=0\), and \(\sigma=c_*\). For
\(0<u\le\sigma/4\), the choice \eqref{eq:c-star-choice} gives
\[
    B_{g_0}(x,2\sqrt u)\times[-u,0]
    \subset
    B_{g_0}(x,\eta^{-1})\times[-(1-\eta),0].
\]
Hence
\[
    |\Ric|\le C_nu^{-1}
    \qquad\text{on }
    B_{g_0}(x,2\sqrt u)\times[-u,0].
\]
Bamler's local concentration estimate
\cite[Proposition~9.5]{Bamler2020A}, applied at scale \(\sqrt u\) and \(\alpha=2\), gives
\[
    W_1^{g_{-u}}\bigl(\nu_{x,0;-u},\delta_x\bigr)
    \le C_n\sqrt u.
\]
The unique coupling of \(\nu_{x,0;-u}\) with \(\delta_x\) is supported on
\(M\times\{x\}\), and therefore
\[
    W_1^{g_{-u}}\bigl(\nu_{x,0;-u},\delta_x\bigr)
    =
    \int_Md_{g_{-u}}(x,y)\,d\nu_{x,0;-u}(y).
\]
Markov's inequality now gives
\[
    \nu_{x,0;-u}
    \bigl(M\setminus B_{g_{-u}}(x,\Lambda\sqrt u)\bigr)
    \le
    \frac{C_n}{\Lambda}.
\]
Taking \(\Lambda=2C_n\) and \(m_0=1/2\), and then restoring the original
parabolic scale, proves \eqref{eq:short-heat-kernel-concentration} and completes
the proof.
\end{proof}

We are now ready to prove
Theorem~\ref{thm:quantitative-nash-to-fisher}.
The rough idea is as follows. If
\(\dfplus{n}(\mathbf x;c_n\tau_0)=A\), then the derivative bounds in
Lemma~\ref{lem:interior-full-fisher-regularity} show that the deficit remains
comparable to \(A\) on a neighborhood of \(x\) for backward times of order
\(A\tau_0\). The concentration estimate
\eqref{eq:short-heat-kernel-concentration} places a fixed amount of
conjugate heat-kernel mass in this neighborhood. The identity
\eqref{eq:nash-integrated-fisher-deficit} then gives
\(A^2\lesssim-\Nash_{\mathbf x}(\tau_0)\).

\begin{proof}[Proof of Theorem~\ref{thm:quantitative-nash-to-fisher}]
After restricting to the connected component containing \(x\), parabolically
rescale and translate time so that \(\tau_0=1,\) \(t=0.\) Set \(c_n:=c_*\) and \(\varepsilon_n:=\varepsilon_*\). Let
\[
    \sigma:=c_n,
    \qquad
    s:=-\sigma,
    \qquad
    d(y,r):=\dfplus{n}\bigl((y,r);r-s\bigr),
\]
and put \(A:=d(x,0)=\dfplus{n}(\mathbf x;\sigma).\) If \(A=0\), there is nothing to prove, so assume \(A>0\). Since \(0\le D^F\le g\),
\begin{equation}
\label{eq:d-between-zero-n}
    0\le d\le n.
\end{equation}

We first convert the Hessian estimate into a square-root gradient estimate.
Enlarge \(L\), if necessary, so that \(L\ge16n\). 

\begin{claim}
\begin{equation}
\label{eq:sqrt-deficit-lipschitz}
    |\nabla\sqrt d|
    \le
    \sqrt{\frac{L}{2\sigma}}
\end{equation}
in the local Lipschitz sense on \(\bigcup_{-\sigma/4\le r\le0} B_{g_r}(x,\sqrt\sigma)\times\{r\}.\)
\end{claim}

\begin{proof}
It suffices to prove
\begin{equation}
\label{eq:sqrt-gradient-deficit}
    |\nabla d|^2
    \le
    \frac{2L}{\sigma}d.
\end{equation}
Indeed, for every \(\delta>0\),
\[
    |\nabla\sqrt{d+\delta}|
    \le
    \sqrt{\frac{L}{2\sigma}}.
\]
Letting \(\delta\downarrow0\), we obtain the claim.

Fix \((p,r)\in  \bigcup_{-\sigma/4\le r\le0} B_{g_r}(x,\sqrt\sigma)\times\{r\}.\) If \(\nabla d(p,r)\ne0\), let \(\gamma\) be the
unit-speed \(g_r\)-geodesic starting at \(p\) in the direction
\(-\nabla d/|\nabla d|\), and set
\[
    \ell
    :=
    \min\left\{
       \frac{\sigma|\nabla d(p,r)|}{L},
       \frac12\sqrt\sigma
    \right\}.
\]
The segment \(\gamma([0,\ell])\) lies in
\(B_{g_r}(x,2\sqrt\sigma)\). Taylor's formula and
\eqref{eq:interior-full-fisher-derivatives} therefore give
\[
    0
    \le
    d(\gamma(\ell),r)
    \le
    d(p,r)-\ell|\nabla d(p,r)|
       +\frac{L}{2\sigma}\ell^2.
\]
If \(\ell=\frac12\sqrt\sigma\), then
\(|\nabla d(p,r)|\ge L/(2\sqrt\sigma)\), so
\eqref{eq:d-between-zero-n} would imply
\[
    0
    \le
    n-\frac L4+\frac L8
    =
    n-\frac L8<0,
\]
a contradiction. Hence \(\ell=\sigma|\nabla d|/L\), and substitution proves
\eqref{eq:sqrt-gradient-deficit}. The case \(\nabla d=0\) is immediate. 
\end{proof}

Choose a dimensional constant \(b>0\) such that
\begin{equation}
\label{eq:b-choice}
    b\sqrt n\le\frac12,
    \qquad
    b\sqrt{\frac L2}\le\frac{\sqrt3-1}{2},
\end{equation}
and then choose \(a>0\) so that
\begin{equation}
\label{eq:a-choice}
    a
    \le
    \min\left\{
       \frac1{4n},
       \frac1{4L},
       \frac{b^2}{\Lambda^2}
    \right\}.
\end{equation}
Set \(u_*:=aA\sigma.\) 

\begin{claim}
\begin{equation}
\label{eq:averaged-deficit-lower}
    \int_Md(y,-u)\,d\nu_{x,0;-u}(y)
    \ge
    \frac{m_0}{4}A,
    \qquad
    0<u\le u_*.
\end{equation}    
\end{claim}

\begin{proof}
By \eqref{eq:d-between-zero-n} and \eqref{eq:a-choice},
\(u_*\le\sigma/4\). For \(0\le u\le u_*\), the time-derivative estimate in
Lemma~\ref{lem:interior-full-fisher-regularity} gives
\begin{equation}
\label{eq:time-persistence}
    d(x,-u)
    \ge
    A-L\frac u\sigma
    \ge
    \frac34A.
\end{equation}
If \(d_{g_{-u}}(x,y)\le b\sqrt{A\sigma},\) then \eqref{eq:d-between-zero-n} and \eqref{eq:b-choice} imply
\(b\sqrt{A\sigma}\le\frac12\sqrt\sigma\). A minimizing geodesic from
\(x\) to \(y\) therefore stays in the region where
\eqref{eq:sqrt-deficit-lipschitz} holds. Combining
\eqref{eq:sqrt-deficit-lipschitz}, \eqref{eq:time-persistence}, and
\eqref{eq:b-choice}, we obtain
\[
    \sqrt{d(y,-u)}
    \ge
    \sqrt{d(x,-u)}
       -\sqrt{\frac{L}{2\sigma}}\,
        d_{g_{-u}}(x,y)
    \ge
    \left(\frac{\sqrt3}{2}-b\sqrt{\frac L2}\right)\sqrt A
    \ge
    \frac12\sqrt A.
\]
Thus
\begin{equation}
\label{eq:deficit-persistence-ball}
    d(y,-u)\ge\frac14A
\end{equation}
on \(B_{g_{-u}}(x,b\sqrt{A\sigma})\) for every
\(0<u\le u_*\).

Since \(u\le aA\sigma\) and \(a\le b^2/\Lambda^2\), we obtain \(\Lambda\sqrt u \le b\sqrt{A\sigma}.\) Equations \eqref{eq:short-heat-kernel-concentration} and
\eqref{eq:deficit-persistence-ball} therefore imply \eqref{eq:averaged-deficit-lower}.
\end{proof}

Applying the exact Nash--Fisher identity
\eqref{eq:nash-integrated-fisher-deficit} with initial time \(s=-\sigma\),
and then substituting \(r=-u\), gives
\[
    -\Nash_{\mathbf x}(\sigma)
    =
    \frac12
    \int_0^\sigma
    \frac1{\sigma-u}
    \int_Md(y,-u)\,d\nu_{x,0;-u}(y)\,du.
\]
Using \eqref{eq:averaged-deficit-lower} on \((0,u_*]\) and
\((\sigma-u)^{-1}\ge\sigma^{-1}\), we find
\[
    -\Nash_{\mathbf x}(\sigma)
    \ge
    \frac12\int_0^{u_*}
    \frac1\sigma\frac{m_0A}{4}\,du
    =
    \frac{am_0}{8}A^2.
\]
Since \(\sigma<1\) and the pointed Nash entropy is nonincreasing in the backward scale, \(-\Nash_{\mathbf x}(\sigma) \le -\Nash_{\mathbf x}(1).\) Consequently,
\[
    A
    \le
    \left(\frac8{am_0}\right)^{1/2}
    \bigl(-\Nash_{\mathbf x}(1)\bigr)^{1/2}.
\]
Thus we may take \(C_n=(8/(am_0))^{1/2}\). This proves
\eqref{eq:quantitative-nash-to-fisher} in the normalized scale. Parabolic rescaling proves the theorem in general.
\end{proof}

\subsection{Proof of codimension-one Fisher \texorpdfstring{\(\varepsilon\)}{epsilon}-regularity}
\label{subsec:FullRankFishEpsReg}

\begin{proof}[Proof of Theorem~\ref{thm:sec6-fisher-epsilon-regularity-n-minus-one}]
    Let \(\varepsilon_{\rm B}=\varepsilon_{\rm B}(n,Y)\in(0,1]\) be the constant in Bamler's \(\varepsilon\)-regularity theorem \cite[Proposition~16.1]{Bamler2020C}: if \(\Nash_{\mathbf x}(\rho^2)\ge -Y\) and \(\mathbf x\) is strongly \((n-1,\varepsilon_{\rm B},\rho)\)-split in the sense of \cite{Bamler2020C}, then
    \[
        r_{\Rm}(\mathbf x)\ge \varepsilon_{\rm B}\rho.
    \]

    Let \(C_*:=C(1/2)\) be the constant in Theorem~\ref{thm:sec6-fisher-rank-produces-strong-splitting-map}. Choose \(\delta=\delta(n,Y)>0\) so small that
    \[
        C_*\sqrt{\delta}\le \varepsilon_{\rm B}^2.
    \]
    Applying Theorem~\ref{thm:sec6-fisher-rank-produces-strong-splitting-map} with \(k=n-1\) and \(\theta=\frac12\), the assumption \(\dfplus{n-1}(\mathbf x;\tau)\le\delta\) gives a scale \(\sigma\in[\frac12\tau,\tau]\) at which the canonical heat-score map is a heat-kernel strong \((n-1,\varepsilon_{\rm B}^2,\sigma)\)-splitting map at \(\mathbf x\).
    
    Set \(\rho:=\sqrt{\varepsilon_{\rm B}\sigma}.\) Then \([t-\varepsilon_{\rm B}^{-1}\rho^2, t-\varepsilon_{\rm B}\rho^2] = [t-\sigma,t-\varepsilon_{\rm B}^2\sigma] \subset [t-\sigma,t].\) Restrict the canonical heat-score map to this subinterval. The heat equation and centering are unchanged. The averaged gradient error is bounded by
    \begin{align*}
        \rho^{-2}
        \int_{t-\sigma}^{t-\varepsilon_{\rm B}^2\sigma}
        \int_M
        \sum_{i,j=1}^{n-1}
        \left|
        \langle\nabla v_i,\nabla v_j\rangle-\delta_{ij}
        \right|
        d\nu_\ell\,d\ell
        &\le
        (\varepsilon_{\rm B}\sigma)^{-1}
        \int_{t-\sigma}^{t}
        \int_M
        \sum_{i,j=1}^{n-1}
        \left|
        \langle\nabla v_i,\nabla v_j\rangle-\delta_{ij}
        \right|
        d\nu_\ell\,d\ell                                      \\
        &\le
        (\varepsilon_{\rm B}\sigma)^{-1}\sigma\varepsilon_{\rm B}^2=
        \varepsilon_{\rm B}.
    \end{align*}

    The remaining splitting-map bounds follow from the \((n-1,\varepsilon_{\rm B}^2,\sigma)\)-map, since \(\varepsilon_{\rm B}\le1\). Thus the restricted map is a strong \((n-1,\varepsilon_{\rm B},\rho)\) splitting map in the sense of \cite[Definition~5.7]{Bamler2020C}.
    
    Finally, \(\rho^2=\varepsilon_{\rm B}\sigma\le\tau. \) By monotonicity of the pointed Nash entropy in the backward time scale,
    \[
        \Nash_{\mathbf x}(\rho^2)
        \ge
        \Nash_{\mathbf x}(\tau)
        \ge
        -Y. 
    \]
    Bamler's \(\varepsilon\)-regularity theorem therefore gives
    \[
        r_{\Rm}(\mathbf x)
        \ge
        \varepsilon_{\rm B}\rho
        =
        \varepsilon_{\rm B}^{3/2}\sqrt{\sigma}
        \ge
        \frac{1}{\sqrt2}\varepsilon_{\rm B}^{3/2}\sqrt{\tau}.
    \]
    Hence the theorem holds with \(\kappa:=\frac{1}{\sqrt2}\varepsilon_{\rm B}^{3/2}.\)
\end{proof}

\bibliographystyle{amsalpha} 
\bibliography{references}

@article{MR196777,
  author  = {Ali, S. M. and Silvey, S. D.},
  title   = {A general class of coefficients of divergence of one distribution from another},
  journal = {J. Roy. Statist. Soc. Ser. B},
  volume  = {28},
  year    = {1966},
  pages   = {131--142},
  mrnumber = {196777},
  url     = {http://links.jstor.org/sici?sici=0035-9246(1966)28:1<131:AGCOCO>2.0.CO;2-U}
}

@book{MR1800071,
  author    = {Amari, Shun-ichi and Nagaoka, Hiroshi},
  title     = {Methods of information geometry},
  series    = {Translations of Mathematical Monographs},
  volume    = {191},
  note      = {Translated from the 1993 Japanese original by Daishi Harada},
  publisher = {American Mathematical Society, Providence, RI; Oxford University Press, Oxford},
  year      = {2000},
  pages     = {x+206},
  isbn      = {0-8218-0531-2},
  mrnumber  = {1800071},
  doi       = {10.1090/mmono/191},
  url       = {https://doi.org/10.1090/mmono/191}
}

@incollection{MR2790368,
  author    = {Arnaudon, Marc and Coulibaly, Kol\'eh\`e Abdoulaye and Thalmaier, Anton},
  title     = {Horizontal diffusion in {$C^1$} path space},
  booktitle = {S\'eminaire de Probabilit\'es {XLIII}},
  series    = {Lecture Notes in Math.},
  volume    = {2006},
  pages     = {73--94},
  publisher = {Springer, Berlin},
  year      = {2011},
  mrnumber  = {2790368},
  doi       = {10.1007/978-3-642-15217-7_2},
  url       = {https://doi.org/10.1007/978-3-642-15217-7_2}
}

@article {MR2601627,
    AUTHOR = {Bailesteanu, Mihai and Cao, Xiaodong and Pulemotov, Artem},
     TITLE = {Gradient estimates for the heat equation under the {R}icci
              flow},
   JOURNAL = {J. Funct. Anal.},
  FJOURNAL = {Journal of Functional Analysis},
    VOLUME = {258},
      YEAR = {2010},
    NUMBER = {10},
     PAGES = {3517--3542},
      ISSN = {0022-1236,1096-0783},
   MRCLASS = {53C44 (35K05 35R01 58J35)},
  MRNUMBER = {2601627},
MRREVIEWER = {Qihua\ Ruan},
       DOI = {10.1016/j.jfa.2009.12.003},
       URL = {https://doi.org/10.1016/j.jfa.2009.12.003},
}

@book{MR3155209,
  author    = {Bakry, Dominique and Gentil, Ivan and Ledoux, Michel},
  title     = {Analysis and geometry of {M}arkov diffusion operators},
  series    = {Grundlehren der mathematischen Wissenschaften},
  volume    = {348},
  publisher = {Springer, Cham},
  year      = {2014},
  pages     = {xx+552},
  isbn      = {978-3-319-00226-2; 978-3-319-00227-9},
  mrnumber  = {3155209},
  doi       = {10.1007/978-3-319-00227-9},
  url       = {https://doi.org/10.1007/978-3-319-00227-9}
}

@unpublished{Bamler2020A,
  author = {Bamler, Richard H.},
  title  = {Entropy and heat kernel bounds on a {R}icci flow background},
  note   = {arXiv:2008.07093},
  year   = {2020},
  url    = {https://arxiv.org/abs/2008.07093}
}

@article {MR4623543,
    AUTHOR = {Bamler, Richard H.},
     TITLE = {Compactness theory of the space of super {R}icci flows},
   JOURNAL = {Invent. Math.},
  FJOURNAL = {Inventiones Mathematicae},
    VOLUME = {233},
      YEAR = {2023},
    NUMBER = {3},
     PAGES = {1121--1277},
      ISSN = {0020-9910,1432-1297},
   MRCLASS = {53E20},
  MRNUMBER = {4623543},
MRREVIEWER = {Shu-Yu\ Hsu},
       DOI = {10.1007/s00222-023-01196-3},
       URL = {https://doi.org/10.1007/s00222-023-01196-3},
}

@unpublished{Bamler2020C,
  author = {Bamler, Richard H.},
  title  = {Structure theory of non-collapsed limits of {R}icci flows},
  note   = {arXiv:2009.03243},
  year   = {2020},
  url    = {https://arxiv.org/abs/2009.03243}
}

@article{Brendle2014TwoPoint,
  author  = {Brendle, Simon},
  title   = {Two-point functions and their applications in geometry},
  journal = {Bulletin of the American Mathematical Society},
  volume  = {51},
  number  = {4},
  year    = {2014},
  pages   = {581--596},
  doi     = {10.1090/S0273-0979-2014-01461-2}
}

@article{MR2081075,
  author   = {Chafa\"i, Djalil},
  title    = {Entropies, convexity, and functional inequalities: on {$\Phi$}-entropies and {$\Phi$}-{S}obolev inequalities},
  journal  = {J. Math. Kyoto Univ.},
  volume   = {44},
  number   = {2},
  year     = {2004},
  pages    = {325--363},
  mrnumber = {2081075},
  doi      = {10.1215/kjm/1250283556},
  url      = {https://doi.org/10.1215/kjm/1250283556}
}

@book{MSM163,
  author    = {Chow, Bennett and Chu, Sun-Chin and Glickenstein, David and Guenther, Christine and Isenberg, James and Ivey, Tom and Knopf, Dan and Lu, Peng and Luo, Feng and Ni, Lei},
  title     = {The {R}icci flow: techniques and applications. {P}art {III}. Geometric-analytic aspects},
  series    = {Mathematical Surveys and Monographs},
  volume    = {163},
  publisher = {American Mathematical Society, Providence, RI},
  year      = {2010},
  pages     = {xx+517},
  isbn      = {978-0-8218-4661-2},
  mrnumber  = {2604955},
  doi       = {10.1090/surv/163},
  url       = {https://doi.org/10.1090/surv/163}
}

@article {MR3458179,
    AUTHOR = {Cibotaru, Daniel and de Lira, Jorge},
     TITLE = {A note on the area and coarea formulas for general volume
              densities and some applications},
   JOURNAL = {Manuscripta Math.},
  FJOURNAL = {Manuscripta Mathematica},
    VOLUME = {149},
      YEAR = {2016},
    NUMBER = {3-4},
     PAGES = {471--506},
      ISSN = {0025-2611,1432-1785},
   MRCLASS = {58C35 (46T12 49Q15)},
  MRNUMBER = {3458179},
MRREVIEWER = {Riikka\ Korte},
       DOI = {10.1007/s00229-015-0779-x},
       URL = {https://doi.org/10.1007/s00229-015-0779-x},
}

@article{MR219345,
  author   = {Csisz\'ar, I.},
  title    = {Information-type measures of difference of probability distributions and indirect observations},
  journal  = {Studia Sci. Math. Hungar.},
  volume   = {2},
  year     = {1967},
  pages    = {299--318},
  mrnumber = {219345}
}

@book {MR257325,
    AUTHOR = {Federer, Herbert},
     TITLE = {Geometric measure theory},
    SERIES = {Die Grundlehren der mathematischen Wissenschaften},
    VOLUME = {Band 153},
 PUBLISHER = {Springer-Verlag New York, Inc., New York},
      YEAR = {1969},
     PAGES = {xiv+676},
   MRCLASS = {28.80 (26.00)},
  MRNUMBER = {257325},
MRREVIEWER = {J.\ E.\ Brothers},
}

@article{Fisher1925,
  author  = {Fisher, R. A.},
  title   = {Theory of statistical estimation},
  journal = {Mathematical Proceedings of the Cambridge Philosophical Society},
  volume  = {22},
  number  = {5},
  year    = {1925},
  pages   = {700--725},
  doi     = {10.1017/S0305004100009580}
}

@article{Hamilton1982,
  author   = {Hamilton, Richard S.},
  title    = {Three-manifolds with positive {R}icci curvature},
  journal  = {J. Differential Geom.},
  volume   = {17},
  number   = {2},
  year     = {1982},
  pages    = {255--306},
  mrnumber = {664497},
  url      = {http://projecteuclid.org/euclid.jdg/1214436922}
}

@article{Hamilton1986,
  author   = {Hamilton, Richard S.},
  title    = {Four-manifolds with positive curvature operator},
  journal  = {J. Differential Geom.},
  volume   = {24},
  number   = {2},
  year     = {1986},
  pages    = {153--179},
  mrnumber = {862046},
  url      = {http://projecteuclid.org/euclid.jdg/1214440433}
}

@article{HeinNaber2014CPAM,
  author   = {Hein, Hans-Joachim and Naber, Aaron},
  title    = {New logarithmic {S}obolev inequalities and an {$\epsilon$}-regularity theorem for the {R}icci flow},
  journal  = {Comm. Pure Appl. Math.},
  volume   = {67},
  number   = {9},
  year     = {2014},
  pages    = {1543--1561},
  mrnumber = {3245102},
  doi      = {10.1002/cpa.21474},
  url      = {https://doi.org/10.1002/cpa.21474}
}

@article {MR4466164,
    AUTHOR = {Huang, Hong},
     TITLE = {Local derivative estimates for the heat equation coupled to
              the {R}icci flow},
   JOURNAL = {Commun. Contemp. Math.},
  FJOURNAL = {Communications in Contemporary Mathematics},
    VOLUME = {24},
      YEAR = {2022},
    NUMBER = {6},
     PAGES = {Paper No. 2150043, 30},
      ISSN = {0219-1997,1793-6683},
   MRCLASS = {53E20 (58J35)},
  MRNUMBER = {4466164},
MRREVIEWER = {Shu-Yu\ Hsu},
       DOI = {10.1142/S0219199721500437},
       URL = {https://doi.org/10.1142/S0219199721500437},
}

@article{koirala2026sharp,
  title={{Sharp Gaussian Isoperimetry along a Ricci Flow}},
  author={Koirala, Robert},
  journal={arXiv preprint arXiv:2605.21193},
  year={2026},
  pages={1-25}
}

@article {MR2433960,
    AUTHOR = {Kuang, Shilong and Zhang, Qi S.},
     TITLE = {A gradient estimate for all positive solutions of the
              conjugate heat equation under {R}icci flow},
   JOURNAL = {J. Funct. Anal.},
  FJOURNAL = {Journal of Functional Analysis},
    VOLUME = {255},
      YEAR = {2008},
    NUMBER = {4},
     PAGES = {1008--1023},
      ISSN = {0022-1236,1096-0783},
   MRCLASS = {53C44 (35B45 35B65 35K55)},
  MRNUMBER = {2433960},
MRREVIEWER = {Christine\ Guenther},
       DOI = {10.1016/j.jfa.2008.05.014},
       URL = {https://doi.org/10.1016/j.jfa.2008.05.014},
}

@article{MR2666905,
  author   = {McCann, Robert J. and Topping, Peter M.},
  title    = {Ricci flow, entropy and optimal transportation},
  journal  = {Amer. J. Math.},
  volume   = {132},
  number   = {3},
  year     = {2010},
  pages    = {711--730},
  mrnumber = {2666905},
  doi      = {10.1353/ajm.0.0110},
  url      = {https://doi.org/10.1353/ajm.0.0110}
}

@article{MR167200,
  author   = {Morimoto, Tetsuzo},
  title    = {Markov processes and the {$H$}-theorem},
  journal  = {J. Phys. Soc. Japan},
  volume   = {18},
  year     = {1963},
  pages    = {328--331},
  mrnumber = {167200},
  doi      = {10.1143/JPSJ.18.328},
  url      = {https://doi.org/10.1143/JPSJ.18.328}
}

@article{MR1760620,
  author   = {Otto, F. and Villani, C.},
  title    = {Generalization of an inequality by {T}alagrand and links with the logarithmic {S}obolev inequality},
  journal  = {J. Funct. Anal.},
  volume   = {173},
  number   = {2},
  year     = {2000},
  pages    = {361--400},
  mrnumber = {1760620},
  doi      = {10.1006/jfan.1999.3557},
  url      = {https://doi.org/10.1006/jfan.1999.3557}
}

@unpublished{Perelman1,
  author = {Perelman, Grisha},
  title  = {The entropy formula for the {R}icci flow and its geometric applications},
  note   = {arXiv:math.DG/0211159},
  year   = {2002},
  url    = {https://arxiv.org/abs/math/0211159}
}

@article {MR15748,
    AUTHOR = {Radhakrishna Rao, C.},
     TITLE = {Information and the accuracy attainable in the estimation of
              statistical parameters},
   JOURNAL = {Bull. Calcutta Math. Soc.},
  FJOURNAL = {Bulletin of the Calcutta Mathematical Society},
    VOLUME = {37},
      YEAR = {1945},
     PAGES = {81--91},
      ISSN = {0008-0659},
   MRCLASS = {62.0X},
  MRNUMBER = {15748},
MRREVIEWER = {J.\ W.\ Tukey},
}

@incollection{MR132570,
  author    = {R{\'e}nyi, Alfr\'ed},
  title     = {On measures of entropy and information},
  booktitle = {Proc. 4th Berkeley Sympos. Math. Statist. and Prob., Vol. I},
  pages     = {547--561},
  publisher = {Univ. California Press, Berkeley-Los Angeles, Calif.},
  year      = {1961},
  mrnumber  = {132570}
}

@article {MR1392331,
    AUTHOR = {Talagrand, M.},
     TITLE = {Transportation cost for {G}aussian and other product measures},
   JOURNAL = {Geom. Funct. Anal.},
  FJOURNAL = {Geometric and Functional Analysis},
    VOLUME = {6},
      YEAR = {1996},
    NUMBER = {3},
     PAGES = {587--600},
      ISSN = {1016-443X,1420-8970},
   MRCLASS = {60E99},
  MRNUMBER = {1392331},
MRREVIEWER = {Juan\ A.\ Cuesta-Albertos},
       DOI = {10.1007/BF02249265},
       URL = {https://doi.org/10.1007/BF02249265},
}

\end{document}